\documentclass[preprint]{imsart}

\RequirePackage[OT1]{fontenc}
\RequirePackage{amsthm,amsmath}
\RequirePackage[numbers]{natbib}
\RequirePackage[colorlinks,citecolor=blue,urlcolor=blue]{hyperref}

\usepackage{graphics,ifthen}
\usepackage{latexsym}
\usepackage{mathrsfs}
\usepackage{amssymb}

\startlocaldefs

\numberwithin{equation}{section} \theoremstyle{plain}
\newtheorem{thm}{Theorem}[section]

\newtheorem{prop}{Proposition}[section]
\newtheorem{lem}{Lemma}[section]
\newtheorem{rem}{Remark}[section]
\newtheorem{cor}{Corollary}[section]

\newcommand{\Hfun}[5]
{H_{#2}^{#1} \left[#3\left|
\begin{array}{l} #4 \vspace*{.1in}\\ #5
\end{array}\right.\right]} 

\newcommand{\pFq}[5]
{\, _{#1}F_{#2}\left(
\begin{array}{c}
#3\vspace*{.1in}\\
#4
\end{array}
; #5
\right)}

\newcommand{\Ip}[3]{\left(I^{#1}_+{#2}\right)\hspace*{-.05in}#3}
\newcommand{\Ineg}[3]{\left(I^{#1}_-{#2}\right)\hspace*{-.05in}#3}

\def\qed{\mbox{\rule{0.5em}{0.5em}}}

\endlocaldefs

\begin{document}

\begin{frontmatter}
\title{Gibbs Partitions, Riemann-Liouville Fractional Operators, Mittag-Leffler Functions, and Fragmentations Derived From Stable Subordinators.\protect}
\runtitle{Gibbs CRP}

\begin{aug}
\author{\fnms{Man-Wai} \snm{Ho}
\ead[label=e1]{mwho@hsmc.edu.hk}},
\author{\fnms{Lancelot F.} \snm{James}\thanksref{t3}\ead[label=e2]{lancelot@ust.hk}
\ead[label=u2,url]{www.bm.ust.hk/isom/faculty-and-staff/directory/lancelot}} and
\author{\fnms{John W.} \snm{Lau}
\ead[label=e3]{john.w.lau@googlemail.com}}
\ead[label=u3,url]{www.web.uwa.edu.au/person/john.lau}

\thankstext{t3}{Supported in part by the grant RGC-GRF 16300217 of the HKSAR.} \runauthor{M.-W. Ho, L.~F. James
and J.~W. Lau}

\affiliation{Hang Seng Management College, The Hong Kong
University of Science and Technology and 
The University of Western Australia}

\address{Department of Mathematics and Statistics\\
School of Decision Sciences\\
Hang Seng Management College\\
Siu Lek Yuen, Shatin, N.T.\\
Hong Kong\\
\printead{e1}\\
}

\address{The Hong Kong University of Science and Technology\\
Department of Information Systems, Business Statistics\\
and Operations Management\\
Clear Water Bay, Kowloon\\
Hong Kong\\
\printead{e2}\\
\printead{u2} }

\address{UWA Centre for Applied Statistics\\
The University of Western Australia (M019)\\
35 Stirling Highway\\
CRAWLEY WA 6009\\
Australia\\
\printead{e3}\\
\printead{u3} }
\end{aug}

\begin{abstract}
Pitman~\cite{Pit03}~(and subsequently Gnedin and Pitman~\cite{Gnedin06}) showed that a large class of random partitions of the integers derived from a stable subordinator of index $\alpha\in(0,1)$
have infinite Gibbs (product) structure as a characterizing feature. The most notable case are random partitions derived from the 
two-parameter Poisson-Dirichlet distribution, $\mathrm{PD}(\alpha,\theta)$, 
which are induced by mixing over variables with generalized Mittag-Leffler distributions, denoted by $\mathrm{ML}(\alpha,\theta).$ 
Our aim in this work is to provide indications on the utility of the wider class of Gibbs partitions as it relates to a study of Riemann-Liouville fractional integrals and size-biased sampling, decompositions of special functions, and its potential use in the understanding of various constructions of more exotic processes. We provide novel characterizations of general laws associated with
two nested families of $\mathrm{PD}(\alpha,\theta)$ mass partitions that are constructed from notable fragmentation operations described in Dong, Goldschmidt and Martin~\cite{DGM} and Pitman~\cite{Pit99Coag}, respectively. These operations are known to be related in distribution to various constructions of discrete random trees/graphs in $[n],$
and their scaling limits, such as stable trees. A centerpiece of 
our work are results related to Mittag-Leffler functions, which play a key role in fractional calculus and are otherwise Laplace transforms of the $\mathrm{ML}(\alpha,\theta)$ variables. Notably, this leads to an interpretation of $\mathrm{PD}(\alpha,\theta)$ laws within a mixed Poisson waiting time framework based on $\mathrm{ML}(\alpha,\theta)$ variables, which suggests connections to recent construction of P\'olya urn models with random immigration by Pek\"oz, R\"ollin and Ross~\cite{Pek2016}.  Simplifications in the Brownian case are highlighted.
\end{abstract}

\begin{keyword}[class=AMS]
\kwd[Primary ]{60C05} \kwd[; secondary ]{60E05}
\end{keyword}

\begin{keyword}
\kwd{beta gamma algebra, Brownian and Bessel processes, Gibbs partitions, Mittag-Leffler functions, stable Poisson-Kingman distributions}
\end{keyword}


\end{frontmatter}
\section{Introduction}
It is known~\cite{Pit97,Pit03,Pit06} that random partitions of the integers $[n]:=\{1,\ldots,n\},$ say $\{C_{1},\ldots,C_{K_{n}}\},$ with $K_{n}\leq n$ unique blocks and sizes $n_{j}=|C_{j}|,$ can be generated by a process of discovery of excursion intervals away from $0$ of Brownian motion or more general Bessel processes of dimension $2-2\alpha,$ for $0<\alpha<1$. That is to say, more generally, processes whose inverse local time follows a stable subordinator $(\hat{S}_{\alpha}(t): t\ge 0)$ of index $\alpha\in(0,1),$ where we can take $\hat{S}_{\alpha}(1)=c^{1/\alpha}S_{\alpha},$ with a positive stable random variable $S_{\alpha}$ having Laplace transform $\mathbb{E}[{\mbox e}^{-\lambda S_{\alpha}}]={\mbox e}^{-\lambda^{\alpha}}$ and density denoted as $f_{\alpha}(t).$ Taking $c=1$ and letting $(\Delta_{k})$ denote the ranked jumps of the subordinator, the ranked lengths of excursion can be constructed as $(P_{\ell}:=\Delta_{\ell}/S_{\alpha})\in \mathcal{P}_{\infty}=\{\mathbf{s}=(s_{1},s_{2},\ldots):s_{1}\ge
s_{2}\ge\cdots\ge 0 {\mbox { and }} \sum_{i=1}^{\infty}s_{i}=1\}$, where $\mathcal{P}_{\infty}$ denotes the space of mass partitions summing to $1$~\cite{BerFrag,Pit06}, and $(P_{\ell})\sim \mathrm{PD}(\alpha,0)$ denotes the Poisson-Dirichlet distribution with parameters $(\alpha,0)$~\cite{PY97}. 

For $K_{n}=k,$ the probability of $\{C_{1},\ldots,C_{{k}}\}$ is given by, what is referred to as the exchangeable partition probability function~(EPPF),
\begin{equation}
p_{\alpha}(n_{1},\ldots,n_{k})=\frac{\alpha^{k-1}\Gamma(k)}{\Gamma(n)}\prod_{j=1}^{k}(1-\alpha)_{n_{j}-1}.
\label{canonEPPF}
\end{equation}
where, for any non-negative integer $x$, $(x)_n = x(x+1)\cdots(x+n-1) = {\Gamma(x+n)}/{\Gamma(x)}$ denotes the Pochhammer symbol. The above EPPF~(\ref{canonEPPF}) and its two-parameter extension~(see~\cite{Pit95,Pit96}),
\begin{equation}
p_{\alpha,\theta}(n_{1},\ldots,n_{k})=\frac{\alpha(\frac{\theta}{\alpha})_k}{(\theta)_n}  \frac{\Gamma(n)}{\Gamma(k)} p_{\alpha}(n_{1},\ldots,n_{k}),
\label{twoEPPF}
\end{equation}
derived from the two-parameter Poisson-Dirichlet distribution, $(P_{\ell})\sim\mathrm{PD}(\alpha,\theta),$ as defined in~\cite{PY97}, constitute the most tractable and notable class of EPPF's that exhibit an inifinite Gibbs or product form~\cite{Pit06}. The EPPF~(\ref{twoEPPF}) is obtained by replacing $S_{\alpha}$ in the above discussion with another variable $S_{\alpha,\theta}$ having density $f_{\alpha,\theta}(t)=t^{-\theta}f_{\alpha}(t)/\mathbb{E}[S^{-\theta}_{\alpha}]$. Furthermore, it corresponds to random partitions generated by the two-parameter Chinese restaurant process, with law denoted as $\mathrm{CRP}(\alpha,\theta),$ as described in~\cite{Pit95,Pit96,Pit06}. 

An important quantity, derived from~(\ref{twoEPPF}), is the probability of the number of blocks $K_n = k$, denoted by $\mathbb{P}_{\alpha,\theta}^{(n)}(k)= \frac{\alpha({\theta}/{\alpha})_k}{(\theta)_n}  \frac{\Gamma(n)}{\Gamma(k)} \mathbb{P}_{\alpha}^{(n)}(k)$ in the $\mathrm{PD}(\alpha,\theta)$ case,
where 
$$
\mathbb{P}_{\alpha}^{(n)}(k):=\mathbb{P}_{\alpha,0}(K_n = k)=\frac{\alpha^{k-1}\Gamma(k)}{\Gamma(n)} S_\alpha(n,k),
$$ 
with $S_\alpha(n,k) =  \frac{1}{\alpha^k k!} \sum_{j=1}^k (-1)^j \binom{k}{j} (-j\alpha)_n$ denoting the generalized Stirling number of the second kind. See~\cite{Pit99,Pit06} for more details in relation to
the derivation of $\mathbb{P}_{\alpha,\theta}^{(n)}(k).$ 
\cite{Pit03} shows that, as 
$n\rightarrow \infty,$  $n^{-\alpha}K_{n}\rightarrow S^{-\alpha}_{\alpha,\theta}$ almost surely~(a.s.). 
Within this context, $S^{-\alpha}_{\alpha,\theta}$ is referred to as the $\alpha$-diversity of the $\mathrm{PD}(\alpha,\theta).$ Following~\cite{PPY92,Pit06,PY97}, a version of $S^{-\alpha}_{\alpha,\theta}$ may be interpreted in terms of the local time up to time $1$ of a generalized Bessel process, specifically,
\begin{equation}
S^{-\alpha}_{\alpha,\theta}:=\frac{1}{\Gamma(1-\alpha)}\lim_{\epsilon\rightarrow
0}\epsilon^{\alpha}|\{i:P_{i}\ge \epsilon\}|~\mathrm{a.s.}.
\label{inverselocaltime}
\end{equation}
For general $\alpha,$ they also arise in various P\'olya urn and random graph/tree growth models as described in, for instance,~\cite{AddarioCut,AldousCRTI,BloemOrbanzPref,CurienHaas,GoldschmidtHaas2015,Haas,Haas2,James2015,SvanteUrn,SvanteKuba,Pek2013,Pek2016Gengam,Pek2017jointpref,RembartWinkel2016,RembartWinkel2016a}.
$S^{-\alpha}_{\alpha},$ with density $g_{\alpha}(z):=f_{\alpha}(z^{-\frac1\alpha})z^{-\frac{1}\alpha-1}/\alpha$, is often referred to as having a Mittag-Leffler distribution. Hence, the generalized Mittag-Leffler variable, $S^{-\alpha}_{\alpha,\theta},$ with distribution denoted as~$\mathrm{ML}(\alpha,\theta),$ has the power-biased density of $g_{\alpha},$
$$
g_{\alpha,\theta}(z)=\frac{z^{\frac\theta\alpha}g_{\alpha}(z)}{\mathbb{E}[S^{-\theta}_{\alpha}]}.
$$
See~\cite{DevroyeGen,DevroyeJames2} for its simulation and other properties. The $\mathrm{CRP}(\alpha,\theta)$ partition of $[n]$ may also be generated by exchangeably sampling $n$ variables from the random distribution function, $P_{\alpha,\theta}(y)=\sum_{k=1}^{\infty}P_{k}\mathbb{I}_{\{\tilde{U}_{k}\leq y\}},$ defined for $(\tilde{U}_{k})$ an infinite collection of iid $\mathrm{Uniform}(0,1)$ variables, independent of $(P_{k}).$
$P_{\alpha,\theta}$ is now known as a Pitman-Yor process~(named in~\cite{IJ2001}), which has applications in Bayesian statistics and machine learning, and arises in numerous areas constituting combinatorial stochastic processes~\cite{BerFrag,Buntine,Goldwater,IJ2001,IJ2003,JamesLamp,Pit96,Pit06,PY97,TehPY,Wood}. 
We note the recursive procedure of \textit{size biased sampling with excision (of excursion intervals)}, as described in~\cite{PPY92}~(see also~\cite{Pit03,Pit06,PY92,PY97}), whereby newly discovered intervals are immediately excised and the remaining lengths are re-scaled to be of length $1,$ produces the stick-breaking sequence $(\tilde{P}_{\ell})\overset{d}=(\beta_{1-\alpha,\theta+\ell\alpha}\prod_{j=1}^{\ell-1}\beta_{\theta+j\alpha,1-\alpha};\ell\ge 1)\sim \mathrm{GEM}(\alpha,\theta)$ of independent beta variables, which is the size-biased re-arrangement of $(P_{k}).$ $\mathrm{GEM}(\alpha,\theta)$ stands for the two-parameter extension of the \textit{Griffiths-Engen-McCloskey} distribution. Throughout this paper, $G_{a}$ denotes a $\mathrm{Gamma}(a,1)$ variable, and $\beta_{a,b}$ denotes a $\mathrm{Beta}(a,b)$ variable with density
$$
f_{\beta_{a,b}}(u)=\frac{\Gamma(a+b)}{\Gamma(a)\Gamma(b)}
u^{a-1}
{(1-u)}^{b-1},\,\,0<u<1.
$$

\subsection{Preliminaries on Poisson-Kingman distributions and Gibbs partitions}
While we shall discuss many properties of the $\mathrm{PD}(\alpha,\theta)$ distribution, our primary focus in this paper are results centered around the general class of EPPF's constituting Gibbs partitions, as derived and discussed in~\cite{Gnedin06,Pit03,Pit06}, called the \textit{Poisson-Kingman~(PK) partitions}. Those works showed that sampling from $(P_{\ell})|S_{\alpha}=t,$ with law denoted as $\mathrm{PD}(\alpha|t),$ leads to a general class of random partitions that have 
infinite Gibbs (product) structure as a characterizing feature. Specifically, the law of $\{C_{1},\ldots,C_{{k}}\}|S_{\alpha}=t$ can be expressed as
\begin{equation}\label{PitEPPF}
p_{\alpha}(n_{1},\ldots,n_{k}|t)=
\mathbb{G}_{\alpha}^{(n,k)}(t)\prod_{j=1}^{k}(1-\alpha)_{n_{j}-1},
\end{equation}
where
\begin{equation}
\label{bigG}
\mathbb{G}_{\alpha}^{(n,k)}(t) =
\frac{\alpha^{k}t^{-n}}{\Gamma(n-k\alpha)f_{\alpha}(t)}
\left[\int_{0}^{t}f_{\alpha}(v)(t-v)^{n-k\alpha-1}dv\right].
\end{equation}
As in~\cite{Pit03}, for any non-negative function $h(t)$ satisfying $\mathbb{E}[h(S_{\alpha})]=1,$ one may mix 
$\mathrm{PD}(\alpha|t)$ over the density, $\gamma(dt)/dt:=h(t)f_{\alpha}(t)$, to obtain a huge class of distributions 
for the Gibbs random partitions. We shall write 
\begin{equation}
(P_\ell) \sim \mathrm{PK}_{\alpha}(\gamma)=\int_{0}^{\infty}\mathrm{PD}(\alpha|t)\gamma(dt),
\label{PKmodel}
\end{equation}
and also use the notation $\mathrm{PK}_{\alpha}(h\circ f_{\alpha})=\mathrm{PK}_{\alpha}(\gamma).$ For instance, $\mathrm{PD}(\alpha,\theta)$ arises when $\gamma(dt) = f_{\alpha,\theta}(t)dt$. Integrating over~(\ref{PitEPPF}) with respect to $\gamma(dt)$ leads to the EPPF of the PK 
partitions~(see~\cite[Theorem 4.6]{Pit06} and~\cite[Theorem 12]{Gnedin06}), expressed as
\begin{equation}
p^{[\gamma]}_{\alpha}(n_{1},\ldots,n_{k})=V_{n,k}\frac{\alpha^{1-k}\Gamma(n)}{\Gamma(k)}p_{\alpha}(n_{1},\ldots,n_{k}),
\label{VEPPF}
\end{equation}
where $V_{n,k}=\int_{0}^{\infty}\mathbb{G}_{\alpha}^{(n,k)}(t)\gamma(dt)$. Naturally, evaluation of~(\ref{VEPPF}) relies very much on the form of $\mathbb{G}_{\alpha}^{(n,k)}(t)$. Pitman~(see \cite[Section 8]{Pit03} and~\cite[Section 4.5, p.90]{Pit06}) developed the Brownian case of $\alpha=\frac12,$ which in many respects is the most remarkable, 
and showed that the EPPF in that case can be expressed explicitly in terms of Hermite functions, or equivalently, confluent hypergeometric functions. Some details of those results are given here in Section~\ref{sec:Hermitesection}.
For a general $0<\alpha<1$, it is nonetheless non-trivial to obtain a representation of $\mathbb{G}_{\alpha}^{(n,k)}(t)$ 
in terms of special functions or other transcendental functions, a question posed in~\cite[Problem 4.3.3, p.87]{Pit06}. In~\cite[Theorem 2.1 and Theorem 3.1]{HJL}, we provided an answer whereby, using representations in~\cite{Schneider86,Schneider87}, we gave alternative expressions of $\mathbb{G}_{\alpha}^{(n,k)}(t)$ in terms of Fox $H$ functions for any general $\alpha$, and in terms of readily computable Meijer $G$ functions for the case of $\alpha=\frac{m}r,$ with co-prime integers $m<r.$ See~\cite{Mathai} and references therein, as well as \cite{HJL}, for more on these special functions, especially their connections to fractional calculus. The Meijer $G$ representations are facilitated by the use of~\cite{Springer70} and the following distributional representatons of $\mathrm{ML}(\frac{m}{r},\theta)$ variables, which can be found in~\cite[Section 8]{JamesLamp}~(see also~\cite{Chaumont,Zolotarev86}, for $\theta=0$),
\begin{equation}
{\left(\frac{m}{S_{\frac{m}{r},\theta}}\right)}^{\frac{m}{r}}\overset{d}=r
\left(\prod_{k=1}^{m-1}\beta^{\frac1r}_{\frac{\theta}{m}+\frac{k}{n},k(\frac{1}{m}-\frac{1}{r})}\right)
\left(\prod_{k=m}^{r-1}G^{\frac1r}_{\frac{\theta}{m}+\frac{k}{r}}\right).
\label{stablegeninteger}
\end{equation}
Importantly, we provided a distributional interpretation which expresses~(\ref{PitEPPF}) as 
\begin{equation}
p_{\alpha}(n_{1},\ldots,n_{k}|t)
=\frac{f^{(n-k\alpha)}_{\alpha,k\alpha}(t)}{f_{\alpha}(t)}\times  p_{\alpha}(n_{1},\ldots,n_{k}),
\label{pitEPPFdecomp}
\end{equation}
where $f^{(n-k\alpha)}_{\alpha,k\alpha}(t)$ denotes the conditional density of $S_{\alpha}|K_{n}=k$, and it corresponds to the densities of random variables of independent pairs,
\begin{equation}
\frac{S_{\alpha,k\alpha}}{\beta_{k\alpha,n-k\alpha}}=\frac{S_{\alpha,n}}{\beta^{\frac1\alpha}_{k,\frac{n}{\alpha}-k}}.
\label{jamesidspecial}
\end{equation}
The equalities in distribution can be read from James~\cite[eq. (2.11)
]{JamesLamp}. The expression $\bigg(S_{\alpha,n}\,\beta^{-\frac1\alpha}_{k,\frac{n}{\alpha}-k}\bigg)^{-\alpha}=S^{-\alpha}_{\alpha,n}\,\beta_{k,\frac{n}{\alpha}-k}$ also arises in~\cite[Proposition 2]{FavaorLijoi2009} as the conditional $\alpha$-diversity of a $\mathrm{PD}(\alpha,0)$ distribution. 
As such, one may represent~(\ref{VEPPF}) as
\begin{equation}
p^{[\gamma]}_{\alpha}(n_{1},\ldots,n_{k})=\mathbb{E}\left[h\left(\frac{S_{\alpha,n}} {\beta^{\frac1\alpha}_{k,\frac{n}{\alpha}-k}}\right)\right]
p_{\alpha}(n_{1},\ldots,n_{k}),
\label{VEPPF2}
\end{equation}
where the expectation is also identical to $\mathbb{E}[h(S_{\alpha})|K_{n}=k].$
Although the $\mathrm{PD}(\alpha,\theta)$ class of models dominates the broad literature, there has been significant interest in the general class of Gibbs partitions. Here we note a few examples in~\cite{Bacallado,CaronFox,Deblasi,GriffithsSpano,KoaGibbs,HJL,LomeliFavaro,PYaku}. Our exposition takes another viewpoint of this general class as we begin to describe next.

\subsection{Outline}
The results in~\cite{Gnedin06,HJL,Pit03}, coupled with refinements in this work, 
allows one to describe explicit distributions and establish scaled limit theorems for a myriad of random partitions of $[n]$, and related constructions based on  $(P_{\ell})\sim \mathrm{PK}_{\alpha}(\gamma)$.
However, in general, those results have not been exploited to provide insights in terms of interpretations, or in fact how to utilize the general framework of Gibbs partitions in novel ways, for what would otherwise be interesting exotic random processes. More specifically, for a given choice of $\gamma$, how does one interpret $(P_{\ell})\sim \mathrm{PK}_{\alpha}(\gamma)$ in~(\ref{PKmodel})? 
For example, if $\gamma$ corresponds to $S_{\alpha}|Y=y,$ $(P_{\ell})\sim \mathrm{PK}_{\alpha}(\gamma)$ does not necessarily equate to the distribution of $(P_{\ell})|Y=y.$ As another example, which we shall discuss further in Section~\ref{sec:Hermitesection},~\cite[eq. (1.2)]{Pek2016} describe a class of P\'olya urn models based on randomized discrete waiting times, that induce random limits corresponding to a broad class of distributions denoted as $\mathrm{UL}(v,(a_{k})_{\{k\ge 1\}}).$ It is a simple matter to select $\gamma$ with this distribution, and thus achieve comparable limits, however there is not an immediate interpretation of $(P_{\ell})$ etc. 
The \textit{deletion of classes}~\cite[Proposition 7]{Pit03} in the case of a stable subordinator, derived from the stick-breaking regime~\cite{PPY92}, is an important case where interpretation is clear. 

In order to give some insights into issues of novel usage and distributional interpretations of the Gibbs partitions, this paper presents three broad based intertwined themes which we first sketch below.

\subsubsection{Fractional operators indexed by $f_{\alpha}$ and size-biased sampling}
Section~\ref{sec:Liouville} presents results from the viewpoint of fractional integrals indexed by $f_{\alpha}$ and a parameter $\nu>0.$ In particular, applying a simple change of variable, one can express $\alpha^{-k}t^{n}f_{\alpha}(t)\mathbb{G}_{\alpha}^{(n,k)}(t)$ as
\begin{eqnarray}
\label{bigI}
\Ip{n-k\alpha}{f_{\alpha}}{(t)} &=&
\frac{1}{\Gamma(n-k\alpha)}\int_{0}^{t}f_{\alpha}(v){(t-v)}^{n-k\alpha-1}dv
\\\nonumber
&=&\frac{\mathbb{E}\big[{(t-S_{\alpha})}^{n-k\alpha-1}\mathbb{I}_{\{S_{\alpha}<t\}}\big]}{\Gamma(n-k\alpha)}.
\end{eqnarray}
Replacing $f_{\alpha}(t)$ with any integrable function $f(t),$ one sees that these equations arise as special cases of right-sided Riemann-Liouville fractional operators of orders $\nu=n-k\alpha$, for $k=1,\ldots n$, defined by
\begin{equation}
\label{RightRL}
\Ip{\nu}{f}{(t)}= \frac{1}{\Gamma(\nu)} \int_0^t f(u) (t-u)^{\nu-1} du.
\end{equation}
One may also consider left-sided Riemann-Liouville fractional integral operators, defined by
$$
\Ineg{\nu}{f}{(t)}= \frac{1}{\Gamma(\nu)} \int_t^\infty f(u) (u-t)^{\nu-1} du,
$$
which we omit further discussion for brevity though. The identity (\ref{bigI}) leads to natural connections to the field of fractional calculus, wherein the interplay between special functions, probability theory, in particular as it relates to size-biased sampling, and fractional operator theory is illustrated. Noting that $\mathbb{G}_{\alpha}^{(1,1)}(t)=1$ leads to the equation,
\begin{equation}
\alpha \Ip{1-\alpha}{f_\alpha}{(t)}
=tf_{\alpha}(t),
\label{diffeo}
\end{equation}
which identifies $f_{\alpha}(t)$ as the unique solution to a particular Abel equation involving general functions $f(t).$ This solution arises as an example in, for instance, \cite{JedidiSimon,Pakesgstable,Schneider87}. In addition, as can be read from~\cite[eqs. (18-19)
]{Pit03}, the equation and its unicity arises as a special case of properties of infinitely divisible variables, see~\cite{Steutel}, and is directly related to size-biased sampling with $n=1.$ In this regard, it is easy to check that the expression~(\ref{diffeo}) corresponds to the identity $S_{\alpha}=S_{\alpha,\alpha}/\beta_{\alpha,1-\alpha},$ established in~\cite{PPY92}. \cite[Section~2.6.b, p.386]{JRY} obtain the equation~(\ref{diffeo}) in their derivation of the Thorin measure (hence the  L\'evy density) of the random variable $G^{1/\alpha}_{1}\overset{d}=G_{1}/S_{\alpha}, $ where the variables on  the right hand side are independent.~\cite[Lemma 2.4]{JRY}, not taking into account the points mentioned above, re-establishes the unicity of~(\ref{diffeo}) via a Laplace transform argument. It is interesting to note that for each $n,$ the random variable with distribution $T_{n}\overset{d}=G_{n}/S_{\alpha}$ plays a subtle but important role here in regards to size-biased sampling~(see~\cite{James2002,JLP}). Furthermore, $(P_{\ell})|G_{n}/S_{\alpha}=\lambda$ does indeed have a PK 
distribution. As such, we incorporate the recent exposition of~\cite{PitmanPoissonMix} on mixed Poisson processes where, in particular, new interpretations of these distributions in the size-biased sampling/species sampling setting are given. A study of the general class $
\Ip{\nu}{f_\alpha}{(t)}$ is given, which connects to various distributional results and identities, including known results for $\mathrm{PD}(\alpha,\theta)$ derived from a different perspective.

\subsubsection{Operations on nested families of mass partitions -- Fragmentations}
Section~\ref{sec:MLC}, \ref{sec:MLCmittag}, and \ref{sec:PitFrag} describe various results for $\mathrm{PK}_{\alpha}(\gamma)$ distributions in~(\ref{PKmodel}) based on well-known constructive operations in the $\mathrm{PD}(\alpha,\theta)$ setting that induce Markov chains, relevant to the construction of random graphs, trees and related structures in continuous and discrete time, and exhibit dual fragmentation/coagulation properties.
In particular, in this broader setting, we obtain distributional properties of two types of fragmentation operations described in $\mathrm{PD}(\alpha,\theta)$ setting in~Dong, Goldschmidt and Martin~\cite{DGM}, and Pitman~\cite{Pit99}. These are complementary results to those for stick-breaking operations in~\cite{PPY92,Pit03,Pit06} in the $\mathrm{PK}_{\alpha}(\gamma)$ setting. That is to say, complementary to~\cite[Proposition 7]{Pit03}, and the family of distributions, $((P_{k,r}),r\ge 0)\sim (\mathrm{PD}(\alpha,\theta+r\alpha),r\ge 0),$
produced by deletion/insertion operations~\cite[Section~6.1]{PY97} when $(P_{k,0})\sim\mathrm{PD}(\alpha,\theta).$ As such, this gives tractable descriptions of combinatorial structures, their non-trivial limits, and other properties, derived from (arguably) the three most prominent actions on nested spaces of mass partitions appearing in the literature. In fact, as far as we know, our descriptions of the fragmentation operations in Section~\ref{sec:PitFrag} have not been developed in any detail elsewhere. Our descriptions in Section~\ref{sec:MLC} and~\ref{sec:PitFrag} exploit quite subtle structural properties, in the two respective cases, to obtain rather remarkable distributional representations. 

\subsubsection{Mittag-Leffler function Gibbs classes and mixed Poisson/inter-arrival time models}
A centerpiece of our work, constituting most of our explicit examples, are results related to various generalizations of the Mittag-Leffler function which play a key role in fractional calculus~\cite{GorenfloMittag}. In Section~\ref{sec:decomp}, as special functions, we consider these in the range of $0<\alpha<1,$ which corresponds to the case where the functions are completely monotone, and, hence, in the first instance here, are Laplace transforms of  
$\mathrm{ML}(\alpha,\theta)$ variables, $S^{-\alpha}_{\alpha,\theta}.$ We show that there is a decomposition in terms of Prabhakar functions, and then describe a corresponding PK 
class and its EPPF in terms of such functions. Distributional interpretatons are further explored in Section~\ref{sec:MLCmittag}. In particular, the view as Laplace transforms of $\mathrm{ML}(\alpha,\theta)$ variables naturally links to the mixed Poisson/inter-arrival framework in~\cite{PitmanPoissonMix}.
This suggests potentially a dual interpretation to the species sampling framework involving the inverse local time $S_{\alpha}.$ In the $\alpha=\frac12$ case, we derive more explicit results,
and establish connections to variables appearing in~\cite{ChassaingJanson,Pek2016,Pit99local}. For instance, we connect the continuous waiting time framework with a class of discrete waiting time models in~\cite{Pek2016}, which also invites further study.
We also encounter further generalizations using the fragmentation operations described in Section~\ref{sec:MLC} and~\ref{sec:PitFrag}. 

An outline of the paper is as follows. Section~\ref{sec:Liouville} presents a characterization of $\mathbb{G}_{\alpha}^{(n,k)}(t)$ in terms of Riemann-Liouville fractional integrals of orders $\nu=n-k\alpha$ indexed by a stable density, and connects this with size-biased sampling and a waiting time framework discussed in~\cite{PitmanPoissonMix}. Extending this to a study for all $\nu>0,$ we encounter some interesting identities appearing in the literature from another viewpoint. Section~\ref{sec:MLC} presents detailed characterizations of the distribution of the Mittag-Leffler Markov chain in~\cite{RembartWinkel2016} under the more general $\mathrm{PK}_{\alpha}(\gamma)$ setting. Simplifications and decompositions are achieved by exploiting identities in Proposition~\ref{Propbetaid1} of Section~\ref{sec:Liouville}. Section~\ref{sec:decomp} describes how one can use the Gibbs partition framework as a method to decompose various special functions. An extensive example is given in terms of a Mittag-Leffler function Gibbs class which can be expressed in terms of Prabhakar functions discussed in~\cite{GorenfloMittag}. Section~\ref{sec:MLCmittag} gives several distributional interpretations of the Mittag-Leffler function Gibbs class. In particular, concrete~(conditional) distributional descriptions of $\mathrm{PD}(\alpha,\theta)$ mass partitions are made in terms of a mixed Poisson waiting time framework based on $S^{-\alpha}_{\alpha,\theta}\sim\mathrm{ML}(\alpha,\theta)$ variables, which exhibit distinguished properties. The results of Section~\ref{sec:MLC} are applied to this setting. Section~\ref{sec:Hermitesection} specializes to the Brownian case of $\alpha=\frac12,$ where many explicit results are given and connections are made to a $\mathrm{UL}(2\theta+j+1,(\frac{\lambda}{2\theta+j+1},\frac{1}{2\theta+j+1}))$ distribution arising in~\cite{Pek2016}. 
These distributions are exponentially tilted and power-biased Rayleigh distributions. Section~\ref{sec:PitFrag} presents the analysis of a Markov chain deduced from fragmentation operations in~\cite{Pit99}, which should, for instance, be connected to the nesting phenomena observed in~\cite{CurienHaas}. Remarkable characterizations, which exploit results deduced from dual coagulation operations, are given. 
Interesting examples are given in term of fragmentations of a Brownian excursion partition conditioned on its local time and a further fragmentation of the Mittag-Leffler function Gibbs class. The appendix contains various results deduced from the coagulation operation in~\cite{Pit99}, including, for example, an apparently new representation for the distribution of the number of blocks $K_{n}$ in the $\alpha=\frac14$ case. Throughout variations of the Mittag-Leffler function are encountered.

\section{Connections to Riemann-Liouville fractional operators}\label{sec:Liouville}
We first present a simple but revealing result, which ties in the properties of
$\mathbb{G}_{\alpha}^{(n,k)}(t)$ 
with the right-sided Riemann-Liouville fractional operator of orders $\nu=n-k\alpha$, $I^{n-k\alpha}_+f_{\alpha}$, defined in~(\ref{RightRL}).
\begin{lem}
\label{simpleprop}
There is the relation for the operator,
\begin{equation}
\Gamma(n)\sum_{k=1}^{n}\mathbb{P}^{(n)}_{\alpha,0}(k)\frac{\Ip{n-k\alpha}{f_\alpha}{(t)}
}{\Gamma(k)}=t^{n-1}\Ip{1-\alpha}{f_\alpha}{(t)}
=\frac{t^{n}f_{\alpha}(t)}\alpha
\label{sizebiasedoperators}.
\end{equation}
\begin{enumerate}

\item[(i)]Taking Laplace transforms of both sides of~(\ref{sizebiasedoperators}) yields, for $\lambda>0,$
\begin{equation}
\frac{\lambda^{n-1}}{\Gamma(n)}\int_{0}^{\infty}{\mbox e}^{-\lambda t}t^{n}f_{\alpha}(t)dt=\alpha{\mbox e}^{-\lambda^{\alpha}}\sum_{k=1}^{n}\mathbb{P}^{(n)}_{\alpha,0}(k)\frac{\lambda^{k\alpha-1}}{\Gamma(k)}.
\label{expmoment}
\end{equation}
\item[(ii)]Letting $\phi^{(n)}_{\alpha}(\lambda)$ denote the $n$-th derivative of $\phi_{\alpha}(\lambda)={\mbox e}^{-\lambda^{\alpha}},$~(\ref{expmoment}) corresponds to the known general representation,
\begin{equation}
\int_{0}^{\infty}{\mbox e}^{-\lambda t}t^{n}f_{\alpha}(t)dt={(-1)}^{n}
\phi^{(n)}_{\alpha}(\lambda)={\mbox e}^{-\lambda^{\alpha}}\sum_{\pi}\prod_{j=1}^{k}\kappa_{n_{j}}(\lambda),
\label{expmoment2}
\end{equation}
where the notation $\sum_{\pi}$ denotes the sum over all partitions of $[n]$ and $\kappa_{j}(\lambda)$, defined by $\Gamma(1-\alpha)\kappa_{j}(\lambda)=\alpha\int_{0}^{\infty}s^{j-\alpha-1}{\mbox e}^{-\lambda s}ds$, is the exponential cumulant.
\end{enumerate}
\end{lem}
\begin{proof}
(\ref{sizebiasedoperators}) arises from~(\ref{PitEPPF}). Statement (i) follows from the known fact that, for any non-negative function $f$, the Laplace transform of the operator $I^{\nu}_+f$, at a point $\lambda>0,$ is equal to  
$
\int_{0}^{\infty}{\mbox e}^{-\lambda t}\Ip{\nu}{f}{(t)}
\,dt=\lambda^{-\nu}\int_{0}^{\infty}{\mbox e}^{-\lambda s}f(s)ds,$ which specializes to
$$
\int_{0}^{\infty}{\mbox e}^{-\lambda t}\Ip{\nu}{f_\alpha}{(t)} 
\,dt=\lambda^{-\nu}{\mbox e}^{-\lambda^{\alpha}}.
$$
\end{proof}

\subsection{A mixed Poisson process viewpoint with respect to $S_{\alpha}$}
\label{sec:mixedPoisson}
The expressions in (\ref{expmoment}) correspond to the density of a random variable  $T_{n}\overset{d}=G_{n}/S_{\alpha},$ with argument $\lambda.$ Furthermore, the expression on the right indicates that $T_{n}\overset{d}=G^{1/\alpha}_{K_{n}},$ where $K_{n}$ is the random number of blocks of a $\mathrm{PD}(\alpha,0)$ partition of $[n]$. When $S_{\alpha}$ is replaced by an infinitely divisible random variable corresponding to a general  subordinator evaluated at a fixed time, such variables $T_{n}$ appear, among other places, in the form of a gamma randomization in relation to exchangeable sampling from discrete random measures formed by normalized subordinators in~\cite{James2002,JLP}. See also~\cite{Bondesson1992,Kingman75,Pit03}.
Pitman~\cite{PitmanPoissonMix} offers a fresh exposition on mixed Poisson processes and their relevance to applications involving, for instance, species sampling models and random partitions, which allows one to enrich the interpretation of $T_{n}$ and related variables. The next general facts and definitions may be read from \cite{PitmanPoissonMix}. For $r=1,2\ldots,$ let $G_{r}:=\sum_{j=1}^{r}\mathbf{e}_{j}:=G_{r-1}+\mathbf{e}_{r}$ denote increasing sums of independent standard exponential variables. For a non-negative random variable $A$ independent of the sequence $(G_{r}),$ define, for each $r,$ $T_{r}=G_{r}/A,$ whence $(T_{r})$ may be interpreted as the sequence of waiting times of a mixed Poisson process $(N_{A}(t) ;t\ge 0)$ defined as 
$$
N_{A}(t)=\sum_{r=1}^{\infty}\mathbb{I}_{\{T_{r}\leq t\}}.
$$
That is, $T_{r}=\inf{\{t:N_{A}(t)=r\}}$, for $r=1,2,\ldots$. There is the following description of the conditional distribution  of $A$ given $(N_{A}(y);0\leq y\leq \lambda),$ for $j=0,1,2,\ldots,$
\begin{equation}
\mathbb{P}(A\in da|(N_{A}(y);0\leq y\leq \lambda), N_{A}(\lambda)=j)=\frac{a^{j}{\mbox e}^{-\lambda a}\mathbb{P}(A\in da)}{\mathbb{E}[A^{j}{\mbox e}^{-\lambda A}]}.
\label{AconPoisson}
\end{equation}

Throughout this paper, let $(\tau_{\alpha}(y);y\ge 0)$ denote a generalized gamma subordinator so that, for fixed $\lambda,$ $\tau_{\alpha}(\lambda^{\alpha})/\lambda$ is a random variable with density ${\mbox e}^{\lambda^{\alpha}}{\mbox e}^{-\lambda t}f_{\alpha}(t)$. In general, for $j=0,1,2,\ldots,$ define 
\begin{equation}
\hat{f}^{[j]}_{\alpha}(t|\lambda)=\frac{\lambda^{j-1}{\mbox e}^{\lambda^{\alpha}}{\mbox e}^{-\lambda t}t^{j}f_{\alpha}(t)}{\Gamma(j)\alpha\sum_{\ell=1}^{j}\mathbb{P}^{(j)}_{\alpha,0}(\ell)\frac{\lambda^{\ell\alpha-1}}{\Gamma(\ell)}}=\frac{t^{j} \times (
{\mbox e}^{\lambda^{\alpha}}{\mbox e}^{-\lambda t}f_{\alpha}(t))}{\mathbb{E}\left[{\left(\frac{\tau_{\alpha}(\lambda^{\alpha})}{\lambda}\right)}^{j}\right]},
\label{condSU}
\end{equation}
which is the $j$-th size biased density of $\tau_{\alpha}(\lambda^{\alpha})/\lambda$, and is otherwise a special case of (\ref{AconPoisson}) with $A=S_{\alpha}.$ See~\cite{PakesKhattree} for more on $j$-biased generalized gamma distributions.
Setting, $\gamma(dt)/dt=\hat{f}^{[j]}_{\alpha}(t|\lambda),$ for each $j,$ define $\mathrm{PK}_{\alpha}(\gamma)$ 
laws 
$$
\mathbb{P}^{[j]}_{\alpha}(\lambda):=\int_{0}^{\infty}\mathrm{PD}(\alpha|t)\hat{f}^{[j]}_{\alpha}(t|\lambda)dt,
$$ 
where  $\mathbb{P}^{[0]}_{\alpha}(\lambda)$ is the popular generalized gamma case. Along with that case, $\mathbb{P}^{[1]}_{\alpha}(\lambda)$ is treated in~\cite{JamesPGarxiv}. Set $\mathcal{S}_{\alpha}(\lambda,j)=\big(N_{S_{\alpha}}(y);0\leq y\leq \lambda, N_{S_{\alpha}}(\lambda)=j\big),$ for $j=0,1,2,\ldots$ 

We use these facts to obtain the following results, which are known in some form. 
In particular, all the results related explicitly to the mixed Poisson formulation can be read from~\cite{PitmanPoissonMix}, with further details.

\begin{prop}Let $(P_{\ell})\sim\mathrm{PD}(\alpha,0),$ which may be constructed from a stable subordinator such that $\hat{S}(1):=S_{\alpha}.$
Set $A=S_{\alpha}$ and consider the mixed Poisson process $(N_{S_{\alpha}}(t);t\ge 0)$ with waiting times $(T_{n}=G_{n}/S_{\alpha};n\ge 1),$ where, for each $n,$ $T_{n}$ has density as in~(\ref{expmoment}), and, in particular, $T_{1}\overset{d}=G^{1/\alpha}_{1}.$ 
\begin{enumerate}
\item[(i)]The conditional density of $S_{\alpha}|T_{n}=\lambda$ is given by $\hat{f}^{[n]}_{\alpha}(t|\lambda),$ in~(\ref{condSU}).
\item[(ii)]The conditional density of $S_{\alpha}|T_{n}=\lambda$ corresponds to the distribution of the random variable
$$
\frac{\tau_{\alpha}\left(\lambda^{\alpha}+G_{\frac{n}{\alpha}-K_{n}}\right)}{\lambda}.
$$

\item[(iii)]When $n=1,$ (\ref{condSU}) reduces to 
$\lambda^{1-\alpha}t ({\mbox e}^{\lambda^{\alpha}}{\mbox e}^{-\lambda t}f_{\alpha}(t)) /\alpha,$ which is the size biased density of $\tau_{\alpha}(\lambda^{\alpha})/\lambda$, and corresponds to the distribution of the random variable
$$
\frac{\tau_{\alpha}(\lambda^{\alpha})}{\lambda}+\frac{G_{1-\alpha}}{\lambda}=\frac{\tau_{\alpha}\left(\lambda^{\alpha}+G_{\frac{1-\alpha}{\alpha}}\right)}{\lambda}.
$$
\item[(iv)] For $j=0,1,2,\ldots,$ $(P_{\ell})|\mathcal{S}_{\alpha}(\lambda,j)\sim \mathbb{P}^{[j]}_{\alpha}(\lambda)$ .
\end{enumerate}
\end{prop}
\begin{proof}Statements (i) and (ii), along with the case of $n=1$ in~(iii), may be deduced as a special case of~\cite[Theorem 2]{JLP}. Both statements~(i) and~(iv) follow from~\cite{PitmanPoissonMix}.
\end{proof}
\begin{rem}It is simple to extend the results to $(P_{\ell})\sim\mathrm{PD}(\alpha,\theta)$ by setting $A=S_{\alpha,\theta}.$
\end{rem}

\subsection{Properties of $I^{\nu}_+f_{\alpha},~\nu>0$}
We next provide a study of the operator $I^{\nu}_+f_{\alpha}$, for general index $\nu>0.$

\begin{thm}\label{theoremgeneoperator}
Select $h(t)\ge 0$ such that $h(t)f_{\alpha}(t)$ is the density of a random variable $T,$ implying $\mathbb{E}[h(S_{\alpha})]=1.$ Then, for any $\nu,\lambda>0,$ there is the following property,
\begin{equation}
\int_{0}^{\infty}{\mbox e}^{-\lambda t}h(t) \Ip{\nu}{f_\alpha}{(t)} 
\,dt=\frac{1}{\lambda^{\nu}{\mbox e}^{\lambda^{\alpha}}}\int_{0}^{\infty}\int_{0}^{\infty}h(u+s)f^{(\nu)}_{\alpha,\lambda}(u,s)\,du\,ds,
\label{hgenop}
\end{equation}
where, for a fixed $\lambda,$
$
f^{(\nu)}_{\alpha,\lambda}(u,s)=\lambda^{\nu}u^{\nu-1}{\mbox e}^{-\lambda u}/\Gamma(\nu) \times \big({\mbox e}^{\lambda^{\alpha}}{\mbox e}^{-\lambda s}f_{\alpha}(s)\big)
$
corresponds to the density of the conditionally independent pair of random variables
\begin{equation}
\left(\frac{G_{\nu}}{\lambda},\frac{\tau_{\alpha}(\lambda^{\alpha})}{\lambda}\right)\overset{d}=\left(\frac{\tau_{\alpha} \big(G_{\frac{\nu}{\alpha}}\big)}{\lambda},\frac{\tau_{\alpha}(\lambda^{\alpha})}{\lambda}\right).
\label{Prop21usage}
\end{equation}
Hence, one may define the sum as a random process $(\tilde{S}_{\alpha,\nu}(\lambda);\lambda>0)$ by the sum of the random variables in~(\ref{Prop21usage}), with
\begin{equation}
\tilde{S}_{\alpha,\nu}(\lambda):=\frac{\tau_{\alpha}\big(G_{\frac{\nu}{\alpha}}\big)+\tau_{\alpha}(\lambda^{\alpha})}{\lambda}=\frac{\tau_{\alpha}\big(G_{\frac{\nu}{\alpha}}+\lambda^{\alpha}\big)}{\lambda},
\label{Gbiasrep}
\end{equation}
which can be equivalently expressed as
\begin{eqnarray}
\tilde{S}_{\alpha,\nu}(\lambda)&=& 
\frac{\tau_{\alpha}(G_{\frac{\nu}{\alpha}}+\lambda^{\alpha})}{\tau_{\alpha}(\lambda^{\alpha})} \times\frac{\tau_{\alpha}(\lambda^{\alpha})}
{\lambda}\\ 
 &=& \frac{\tau_{\alpha}\big(G_{\frac{\nu}{\alpha}}+\lambda^{\alpha}\big)}{{\left(G_{\frac{\nu}{\alpha}}+\lambda^{\alpha}\right)}^{\frac1\alpha}}\times
{\left(\frac{G_{\frac{\nu}{\alpha}}+\lambda^{\alpha}}
{\lambda^{\alpha}}\right)}^{\frac1\alpha}.
\end{eqnarray} 
The variables separated by $\times$ are not independent for fixed $\lambda.$
\end{thm}

\begin{proof}
(\ref{hgenop}) is obtained by noting that the left hand side can be expressed as\break 
$\int_{0}^{\infty}[\int_{s}^{\infty}(t-s)^{\nu-1}h(t){\mbox e}^{-\lambda t}dt]
f_{\alpha}(s)ds/\Gamma(\nu).$ One may then appeal to~\cite[Proposition~21]{PY97} to obtain the representation in~(\ref{Prop21usage}). That is, $\tau_{\alpha}\big(G_{\frac{\nu}{\alpha}}\big)\overset{d}=G_{\nu},$ which is otherwise easy to verify. 
\end{proof}

\begin{cor}\label{corgeneoperator} The random variable $\tilde{S}_{\alpha,\nu}(\lambda)$ has a density in $t$ as
\begin{equation}
\frac{\lambda^{\nu}{\mbox e}^{\lambda^{\alpha}}}{\Gamma(\frac{\nu}{\alpha})}{\mbox e}^{-\lambda t}
\int_{0}^{1}f_{\alpha}(tu^{\frac{1}{\alpha}})u^{-\frac{(\nu-1)}{\alpha}-1}{(1-u)}^{\frac{\nu}{\alpha}-1}du,
\label{nudensity}
\end{equation}
and its Laplace transform is given as, for $y>0$,
$$
\mathbb{E}\big[{\mbox e}^{-y\tilde{S}_{\alpha,\nu}(\lambda)}\big]=\frac{{\mbox e}^{\lambda^{\alpha}-(\lambda+y)^{\alpha}}}{{(1+\frac y\lambda)}^{\nu}}.
$$
\begin{enumerate}
\item[(i)] If the density~(\ref{nudensity}) is exponentially tilted by ${\mbox e}^{-yt}$ for a fixed $y>0,$ then the corresponding random variable can be represented as
$\tilde{S}_{\alpha,\nu}(\lambda+y).$
\item[(ii)] When $\nu=1-\alpha$,~(\ref{nudensity}) agrees with the density of $S_{\alpha}|T_{1}=\lambda$ specified as\break $\lambda^{1-\alpha}t \big({\mbox e}^{\lambda^{\alpha}}{\mbox e}^{-\lambda t} f_{\alpha}(t)\big)/\alpha.$ This yields the known identity,
$$
\frac{1}{\Gamma(\frac{1-\alpha}{\alpha})}\int_{0}^{1}f_{\alpha}(tu^{\frac{1}{\alpha}}){(1-u)}^{\frac{1-\alpha}{\alpha}-1}du=\frac{1}{\alpha}tf_{\alpha}(t),
$$
which corresponds to the result $S_{\alpha}=S_{\alpha,1}\times \beta^{-\frac{1}{\alpha}}_{1,\frac{1-\alpha}{\alpha}}.$
\end{enumerate}
\end{cor}

\begin{proof}
The density and the Laplace transform are straightforward. Result in (i) follows readily from the density~(\ref{nudensity}). For~(ii),
one may appeal to~\cite[Theorems 1 and 2]{JLP} in the case $n=1.$
Equivalently, this can be deduced from a careful read of~\cite{PPY92,PY92}, see in particular~\cite[remark 3.6 and eq.~(3.q)
]{PY92}, which yields the appropriate form of the conditional density. Or otherwise one can use~(\ref{condSU}) in the case of $n=1.$
\end{proof}

\subsection{Associated random variables appearing in Bertoin and Yor~\cite{BerYor} and James~\cite{JamesLamp}}
In this section, we show that the general operators $I^{\nu}_+f_{\alpha}$ are directly linked to random variables appearing in~\cite{BerYor,JamesLamp}. In particular, the variables $Z^{(\frac{\nu}{\alpha})}_{\alpha,\omega}$ indexed by $(\nu,\omega)$, described below, correspond to the entire range of variables given in~\cite[Lemma 6, eq. (10)]{BerYor}, and also~\cite{JamesLamp}, as described in the forthcoming~Remark~\ref{remJames}. See also~\cite{Patieexp}.
\begin{prop}
For any $\omega>0,$
let $G_{\frac{\omega}{\alpha}}$ be a gamma random variable with parameters $(\frac{\omega}{\alpha},1)$, which is independent of $\tilde{S}_{\alpha,\nu}(\lambda).$ 
\begin{enumerate}
\item[(i)] The following random variables are equivalent.
\begin{equation}
Y^{(\nu)}_{\alpha,\omega}:=\tilde{S}_{\alpha,\nu}\left(G^{\frac1\alpha}_{\frac{\omega}{\alpha}}\right)=\frac{S_{\alpha,\omega}}{\beta_{\omega,\nu}}=\frac{S_{\alpha,\omega+\nu}}{\beta^{\frac1\alpha}_{\frac{\omega}{\alpha},\frac{\nu}{\alpha}}}
\label{jamesid},
\end{equation}
where the variables appearing in the ratios are independent.
\item[(ii)]Define $Z^{(\frac{\nu}{\alpha})}_{\alpha,\omega}={\left(Y^{(\nu)}_{\alpha,\omega}\right)}^{-\alpha}$. Then,
\begin{equation}
Z^{(\frac{\nu}{\alpha})}_{\alpha,\omega}:=\left[\tilde{S}_{\alpha,\nu}\left(G^{\frac1\alpha}_{\frac{\omega}{\alpha}}\right)\right]^{-\alpha}=\frac{\beta^{\alpha}_{\omega,\nu}}{S^{\alpha}_{\alpha,\omega}}=\frac{\beta_{\frac{\omega}{\alpha},\frac{\nu}{\alpha}}}{S^{\alpha}_{\alpha,\omega+\nu}}.
\label{jamesid2}
\end{equation}
\end{enumerate}
\end{prop}
\begin{proof}
Statement~(i) follows from the representations of $\tilde{S}_{\alpha,\nu}(\lambda)$ given in Theorem~\ref{theoremgeneoperator} coupled with~\cite[Proposition 21]{PY97}.
\end{proof}

\begin{rem}
\label{remJames}
Setting 
$\omega=\tau\sigma$ and $\nu=\tau(1-\sigma),$ for $\tau>0$ and $0<\sigma<1,$ one sees that
the equality on the right hand side of (\ref{jamesid}) agrees, in full generality, with the
identity in~James~\cite[eq. (2.11)
]{JamesLamp}, except, by construction, they are on the same space. 
\end{rem}

We now give equivalent expressions of the densities of the random variables in~(\ref{jamesid}).
\begin{prop}
Let 
 $f^{(\nu)}_{\alpha,\omega}(t)$ denote the density of $Y^{(\nu)}_{\alpha,\omega}$ defined in (\ref{jamesid}).
 \begin{enumerate}
\item[(i)]Using the form of the density indicated by 
$S_{\alpha,\omega}/\beta_{\omega,\nu},$ it follows that, for $\omega>0,$
\begin{equation}
{f}^{(\nu)}_{\alpha,\omega}(t)=\frac{\alpha \Gamma(\nu+\omega)}{\Gamma(\frac{\omega}{\alpha})}t^{-(\nu+\omega)}\Ip{\nu}{f_\alpha}{(t)}
\label{opdensity},
\end{equation}
where $\frac{\Gamma(\omega)\mathbb{E}[S^{-\omega}_{\alpha}]}{\Gamma(\nu+\omega)}=\frac{\Gamma(\omega/\alpha)}{\alpha\Gamma(\nu+\omega)}.$ 
\item[(ii)]Using~(\ref{nudensity}),  an alternate form of $f^{(\nu)}_{\alpha,\omega}(t)$ is obtained as
\begin{equation}
\frac{\alpha \Gamma(\nu+\omega)}{\Gamma(\frac{\nu}{\alpha})\Gamma(\frac{\omega}{\alpha})}t^{-(\nu+\omega)}
\int_{0}^{1}f_{\alpha}(tu^{\frac{1}{\alpha}})u^{-\frac{(\nu-1)}{\alpha}-1}{(1-u)}^{\frac{\nu}{\alpha}-1}du.
\label{gammanudensity}
\end{equation} 
\item[(iii)]Combining~(\ref{opdensity}) and~(\ref{gammanudensity}) gives an identity, for $\nu>0$,  
\begin{equation}
\Ip{\nu}{f_{\alpha}}{(t)}
=\frac{1}{\Gamma(\frac{\nu}{\alpha})}
\int_{0}^{1}f_{\alpha}(tu^{\frac{1}{\alpha}})u^{-\frac{(\nu-1)}{\alpha}-1}{(1-u)}^{\frac{\nu}{\alpha}-1}du.
\label{opnudensity2}
\end{equation}
\end{enumerate}
\end{prop} 

\begin{rem}
Theorem~\ref{theoremgeneoperator} and Corollary~\ref{corgeneoperator} show what can be understood from operators of general index $\nu.$ It is, for instance, interesting, and  a bit surprising, to us that the identity~(\ref{jamesid}) appears in this context.
\end{rem}
Combining the results in Propositions~\ref{simpleprop} and Theorem~\ref{theoremgeneoperator} leads to
\begin{equation}
\int_{0}^{\infty}\frac{\lambda^{n-1}h(t)t^{n}{\mbox e}^{-\lambda t}}{\Gamma(n)}f_{\alpha}(t)dt=\frac{\alpha}{{\mbox e}^{\lambda^{\alpha}}} \sum_{k=1}^{n}\frac{\mathbb{P}^{(n)}_{\alpha,0}(k)\lambda^{k\alpha-1}}{\Gamma(k)}\mathbb{E}[h(\tilde{S}_{\alpha,n-k\alpha}(\lambda))].
\label{expmomentH}
\end{equation}

\begin{cor}\label{corcond}
The expression~(\ref{expmomentH}) indicates that $\tilde{S}_{\alpha,n-k\alpha}(\lambda)$ corresponds to the conditional distribution of $S_{\alpha}|K_{n}=k, T_{n}=\lambda$. Integrating over $\lambda$ on both sides of (\ref{expmomentH}) indicates that the conditional distribution of $S_{\alpha}|K_{n}=k$ is equivalent to 
\begin{equation}
Y^{(n-k\alpha)}_{\alpha,k\alpha}:=\tilde{S}_{\alpha,n-k\alpha}\left(G^{\frac1\alpha}_{k}\right)=
\frac{S_{\alpha,k\alpha}}{\beta_{k\alpha,n-k\alpha}}=\frac{S_{\alpha,n}}{\beta^{\frac1\alpha}_{k,\frac{n}{\alpha}-k}}.
\label{jamesid3}
\end{equation}
which, as in~\cite{HJL}, has the density 
$$
{f}^{(n-k\alpha)}_{\alpha,k\alpha}(t)=
\frac{\alpha \Gamma(n)}{\Gamma(k)} t^{-n}\Ip{n-k\alpha}{f_\alpha}{(t)}
.
$$
\end{cor}

\subsection{Recursive density decompositions and semi-group properties}

Gnedin and Pitman~\cite[Definition~3 or eq. (8)]{Gnedin06}
establishes the following backward recursion for all $V_{n,k}$,
$n=1,2,\ldots,~k=1,2,\ldots,n$,
\begin{equation}\label{backward}
V_{n,k} = (n-k\alpha) V_{n+1,k} + V_{n+1, k+1},
\end{equation}
with $V_{1,1}=1$. We use this to obtain the next result.

\begin{prop}The densities ${f}^{(n-k\alpha)}_{\alpha,k\alpha}(t),$ and hence the operators $
I^{\nu}_+ f_\alpha,$
 satisfy the following mixture relationships.
\begin{enumerate}
\item[(i)]The stable density can be decomposed in terms of its conditional densities, and is given by
\noindent\begin{equation}
f_{\alpha}(t)=\sum_{k=1}^{n}
\mathbb{P}_{\alpha,0}^{(n)}(k)
f^{(n-k\alpha)}_{\alpha,k\alpha}(t).
\label{stabeldendecomp}
\end{equation}
\item[(ii)]Manipulating the recursive equation in~(\ref{backward}) leads to two equivalent backward recursive 2--point mixture representations,
\begin{equation}
f^{(n-k\alpha)}_{\alpha,k\alpha}(t)=
\left(\frac{k\alpha}{n}\right)f^{(n+1-(k+1)\alpha)}_{\alpha,(k+1)\alpha}(t) + \left(1-\frac{k\alpha}{n}\right) f^{(n+1-k\alpha)}_{\alpha,k\alpha}(t),
\label{recursedensity1}
\end{equation}
with $f^{(1-\alpha)}_{\alpha,\alpha}(t)=f_{\alpha}(t)$, and 
\begin{equation}
t\Ip{n-k\alpha}{f_\alpha}{(t)}
={(n-k\alpha)}\Ip{n+1-k\alpha}{f_\alpha}{(t)}
+{\alpha}\Ip{n+1-(k+1)\alpha}{f_\alpha}{(t)}
.
\label{recurseop2}
\end{equation}
\end{enumerate}
\end{prop}

\begin{rem}
A special case of~(\ref{sizebiasedoperators}) with $n=2$, 
$$
t^{2}f_{\alpha}(t)=\alpha{(1-\alpha)}\Ip{2-\alpha}{f_\alpha}{(t)}
+{\alpha^{2}}\Ip{2-2\alpha}{f_\alpha}{(t)}
,
\label{recurseop3}
$$
appears by setting $n=k=1$ in the latter equation.
\end{rem}

\subsection{Decompositions associated with the semi-group property of $I^{\nu}_+f_\alpha
$}
\label{sec:decompI}
For the relevant background in this section, see, for instance,~\cite{GorenfloMittag,Mathai}.
It is known that, for any function $f$, and $\nu_{1},\nu_{2}>0$, the operators $I^{\nu_{1}}_+$ and $I^{\nu_{2}}_+$ satisfy the following semi-group property, 
\begin{equation}
\big(\big(I^{\nu_{1}}_+\ast I^{\nu_{2}}_+\big)f\big)\hspace*{-.03in}(t):=\frac{1}{\Gamma(\nu_{1})}\int_{0}^{t}{(t-s)}^{\nu_{1}-1}\Ip{\nu_2}{f}{(s)}
\,ds=\Ip{\nu_1+\nu_2}{f}{(t)}
.
\label{semi}
\end{equation}
We also introduce Erd\'elyi-Kober operators defined as 
\begin{equation}
\big(\mathscr{E}^{\kappa,\nu}f\big)(t)=\frac{t^{-\kappa-\nu}}{\Gamma(\nu)}\int_{0}^{t}f(s)s^{\kappa}{(t-s)}^{\nu-1}ds.
\label{erdelyi}
\end{equation}
When $f(s)$ is the density for a random variable $Y,$  $\Gamma(\nu+\kappa)/\Gamma(\kappa)\times \big(\mathscr{E}^{\kappa,\nu}f\big)(t)$ is the density of the variable $Y/\beta_{\kappa,\nu}.$ 

\begin{prop}\label{Decompkk} 
For $j=1,\ldots,k,$ and $k=1,\ldots, n,$ 
\begin{equation}
\Ip{k-k\alpha}{f_\alpha}{(t)}
=\bigg(\big(\overset{k \scriptsize{\mbox{ terms}}}{\overbrace{I^{1-\alpha}_+\ast \cdots\ast I^{1-\alpha}_+}}\big)
f_{\alpha}\bigg)\hspace*{-.03in}(t)=\frac{t^{k}f^{(k-k\alpha)}_{\alpha,k\alpha}(t)}\alpha,
\label{Abelsemiproperty}
\end{equation}
where, for each fixed $k,$ $f^{(k-k\alpha)}_{\alpha,k\alpha}(t)$ is the conditional density of the random variable $S_{\alpha}|K_{k}=k,$ when the number of blocks $K_k$ of a partition of $[k]=\{1,\ldots,k\}$ is $k$ under $\mathrm{PD}(\alpha,0).$ Then,
\begin{enumerate}
\item[(i)]$\Ip{n-k\alpha}{f_\alpha}{(t)}
=\left(\big(I^{n-k}_+ \ast I^{k-k\alpha}_+\big)f_{\alpha}\right)\hspace*{-.04in}(t)$ is equivalent to
\begin{equation}
\alpha \Ip{n-k\alpha}{f_\alpha}{(t)}
=\frac{1}{\Gamma(n-k)}\int_{0}^{t}s^{k}{(t-s)}^{n-k-1}f^{(k-k\alpha)}_{\alpha,k\alpha}(s)ds.
\label{specialconvolution}
\end{equation}
\item[(ii)](\ref{specialconvolution}) indicates the following relation to Erd\'elyi-Kober operators,
\begin{equation}
f^{(n-k\alpha)}_{\alpha,k\alpha}(t)=
\frac{\alpha \Gamma(n)}{\Gamma(k)}t^{-n}
\Ip{n-k\alpha}{f_\alpha}{(t)}=
\frac{\Gamma(n)}{\Gamma(k)}\bigg(\mathscr{E}^{k,n-k}f^{(k-k\alpha)}_{\alpha,k\alpha}\bigg)\hspace*{-.03in}(t).
\label{EK1}
\end{equation}
\item[(iii)]$f^{(k-k\alpha)}_{\alpha,k\alpha}(t)$ is the density of random variables,
\begin{equation}
Y^{(k(1-\alpha))}_{\alpha,k\alpha}:=\tilde{S}_{\alpha,k(1-\alpha)} \left(G^{\frac1\alpha}_{k}\right)=
\frac{S_{\alpha,k\alpha}}{\beta_{k\alpha,k-k\alpha}}=\frac{S_{\alpha,k}}{\beta^{\frac1\alpha}_{k,k(\frac{1-\alpha}{\alpha})}}.
\label{jamesid3}
\end{equation}
\item[(iv)]It follows from~(\ref{EK1}) that $Y^{(n-k\alpha)}_{\alpha,k\alpha}=Y^{(k(1-\alpha))}_{\alpha,k\alpha}/\beta_{k,n-k}.$
\end{enumerate}
\end{prop}
\begin{proof}
(\ref{Abelsemiproperty}) can be deduced from (\ref{opdensity}) or Corollary~\ref{corcond}, coupled with the semi-group property. The remaining results follow from this.
\end{proof}

\begin{rem}Proposition~\ref{Decompkk} highlights the fact that the random variable $Y^{(k(1-\alpha))}_{\alpha,k\alpha}$, when further randomized, leads to the random variable $Y^{(K_{n}(1-\alpha))}_{\alpha,K_{n}\alpha},$ which seems worthwhile to investigate further.
\end{rem}

\subsection{Recovering results for $\mathrm{PD}(\alpha,\theta)$}
We now demonstrate how to recover some known and some not so well-known results in the $\mathrm{PD}(\alpha,\theta)$ setting from the current perspective. Here we will employ the following facts involving conditional expectations
$
\mathbb{E}[S^{-\theta}_{\alpha}|K_{n}=k]={\Gamma(n) \Gamma(\frac\theta\alpha+k)}/[{\Gamma(k) \Gamma(\theta+n)}]$
and hence
$$
d^{(n)}_{\alpha,\theta}(k)=\frac{\mathbb{E}[S^{-\theta}_{\alpha}|K_{n}=k]}{\mathbb{E}[S^{-\theta}_{\alpha}]}=\frac{\Gamma(n)\Gamma(\theta+1)\Gamma(\frac{\theta}{\alpha}+k)}{\Gamma(k)\Gamma(\theta+n)\Gamma(\frac{\theta}{\alpha}+1)} = \frac{\Gamma(n)}{\Gamma(k)} \frac{\alpha(\frac\theta\alpha)_k}{(\theta)_n}.
$$
\begin{cor}\label{PDRLresults}
Under the $\mathrm{PD}(\alpha,\theta)$ setting for $\theta>-\alpha,$ the change of measure corresponds to 
the choice of $h(t)=t^{-\theta}/\mathbb{E}[S^{-\theta}_{\alpha}].$ That is, $f_{\alpha,\theta}(t)=h(t)f_{\alpha}(t).$ Hence, for $\mathrm{PD}(\alpha,\theta)$, the joint distribution of $(S_{\alpha,\theta},K_{n}=k)$ can be expressed as
\begin{equation}
\frac{t^{-\theta}{f}^{(n-k\alpha)}_{\alpha,k\alpha}(t)}{\mathbb{E}[S^{-\theta}_{\alpha}]d^{(n)}_{\alpha,\theta}(k)}\times d^{(n)}_{\alpha,\theta}(k)\mathbb{P}^{(n)}_{\alpha,0}(k).
\label{jointSKn}
\end{equation}
\begin{enumerate}
\item[(i)]Using (\ref{opdensity}), the conditional density of $S_{\alpha,\theta}|K_{n}=k$ is $f^{(n-k\alpha)}_{\alpha,\theta+k\alpha}(t).$
\item[(ii)]Integrating (\ref{jointSKn}) over $t$ recovers the known distribution of $K_{n}$ under $\mathrm{PD}(\alpha,\theta),$ with
\begin{equation}
\mathbb{P}^{(n)}_{\alpha,\theta}(k)=
d^{(n)}_{\alpha,\theta}(k)
\mathbb{P}^{(n)}_{\alpha,0}(k).
\label{PDgenKn}
\end{equation}
\item[(iii)]An application of~(\ref{jamesid}) shows that $f^{(n-k\alpha)}_{\alpha,\theta+k\alpha}(t)$ corresponds to the density of random variables,
\begin{equation}
Y^{(n-k\alpha)}_{\alpha,\theta+k\alpha}:=\tilde{S}_{\alpha,n-k\alpha}\left(G^{\frac1\alpha}_{\frac{\theta}{\alpha}+k}\right)=
\frac{S_{\alpha,\theta+k\alpha}}{\beta_{\theta+k\alpha,n-k\alpha}}=\frac{S_{\alpha,\theta+n}}{\beta^{\frac1\alpha}_{\frac{\theta}{\alpha}+k,\frac{n}{\alpha}-k}}.
\label{conddentheta}
\end{equation}
Hence, there is a recovery of the identities
\begin{equation}
S_{\alpha,\theta}=\frac{S_{\alpha,\theta+K_{n}\alpha}}{\beta_{\theta+K_{n}\alpha,n-K_{n}\alpha}}=\frac{S_{\alpha,n+\theta}}{\beta^{\frac1\alpha}_{\frac{\theta}{\alpha}+K_{n},\frac{n}{\alpha}-K_{n}}}:=Y^{(n-K_{n}\alpha)}_{\alpha,\theta+K_{n}\alpha},
\label{SKn}
\end{equation}
and $G^{\frac1\alpha}_{\frac{\theta}{\alpha}+K_{n}}\overset{d}={G_{\theta+n}}/{S_{\alpha,\theta}}$.
\end{enumerate}
\end{cor}
The next Corollary corresponds to the notion that size-biased sampling with and without excision (of excursion intervals) agree when $n=1.$

\begin{cor}
Setting $n=1$ in~(\ref{SKn}), or otherwise noting that $\mathbb{P}_{\alpha,\theta}^{(1)}(1)
=1$ implies $S_{\alpha,\theta}|K_{1}=1$ is just $S_{\alpha,\theta}$, conclude the following.
\begin{enumerate}
\item[(i)] For  
$\nu=1-\alpha,~\omega=\theta+\alpha,$ and $\theta>-\alpha,$ 
\begin{equation}
S_{\alpha,\theta}=\tilde{S}_{\alpha,1-\alpha} \left(G^{\frac1\alpha}_{\frac{\theta+\alpha}{\alpha}}\right)=\frac{S_{\alpha,\theta+\alpha}}{\beta_{\theta+\alpha,1-\alpha}}=\frac{S_{\alpha,\theta+1}}{\beta^{\frac1\alpha}_{\frac{\theta+\alpha}{\alpha},\frac{1-\alpha}{\alpha}}}:=Y^{(1-\alpha)}_{\alpha,\theta+\alpha}.
\label{jamesidPD}
\end{equation}
\item[(ii)]The expressions in~(\ref{jamesidPD}) correspond to the size-biased sampling results in~\cite{PPY92,PY92,PY97}, specialized to the case of 
$\mathrm{PD}(\alpha,\theta).$ The natural size-biased representation, without scalings, is given by
\begin{equation}
\tilde{S}_{\alpha,1-\alpha}\left(G^{\frac1\alpha}_{\frac{\theta+\alpha}{\alpha}}\right)=
\frac{\tau_{\alpha}\left(G_{\frac{\theta+\alpha}{\alpha}}\right)}{G^{\frac1\alpha}_{\frac{\theta+\alpha}{\alpha}}}
+\frac{\tau_{\alpha}\left(G_{\frac{1-\alpha}{\alpha}}\right)}{G^{\frac1\alpha}_{\frac{\theta+\alpha}{\alpha}}},
\label{GbiasrepPPY}
\end{equation}
where the first ratio in the sum at the right hand side is equivalent to $S_{\alpha,\theta+\alpha}$, 
and the second term is the first jump picked according to size-biased sampling, when $n=1.$ The two terms are not independent.
\item[(iii)]With $G_{\frac{\theta+\alpha}{\alpha}}+G_{\frac{1-\alpha}{\alpha}}:=G_{\frac{\theta+1}{\alpha}}$,
$$
S_{\alpha,\theta+1}:=\frac{\tau_{\alpha}\left(G_{\frac{\theta+1}{\alpha}}\right)}{G^{\frac1\alpha}_{\frac{\theta+1}{\alpha}}},~
\beta_{\theta+\alpha,1-\alpha}=\frac{S_{\alpha,\theta+\alpha}}{S_{\alpha,\theta}},{\mbox { and }}\beta_{\frac{\theta+\alpha}{\alpha},\frac{1-\alpha}{\alpha}}=\frac{S^{-\alpha}_{\alpha,\theta}}{S^{-\alpha}_{\alpha,\theta+1}}.
$$

\item[(iv)]The recursions suggested by~(\ref{jamesidPD}) lead to the known representations, for each $n\ge 1,$
\begin{equation}
S_{\alpha,\theta}=\frac{S_{\alpha,\theta+n\alpha}}{\prod_{j=1}^{n}\beta_{\theta+j\alpha,1-\alpha}}=\frac{S_{\alpha,\theta+n}}{\prod_{j=1}^{n}\beta^{\frac1\alpha}_{\frac{\theta+\alpha+j-1}{\alpha},\frac{1-\alpha}{\alpha}}}.
\label{jamesidPDrecurse}
\end{equation}
\end{enumerate}
\end{cor}
\begin{rem} The beta variables are to be understood in terms of their representations as ratios of the local times that just happen to have independent beta distributions under $\mathrm{PD}(\alpha,\theta).$
\end{rem}

\subsection{Beta products and $K_{n}$}
Combining~(\ref{SKn}) and (\ref{jamesidPDrecurse}) leads to
\begin{equation}
S^{-\alpha}_{\alpha,\theta}={S^{-\alpha}_{\alpha,\theta+n}}\,{\prod_{j=1}^{n}\beta_{\frac{\theta+\alpha+j-1}{\alpha},\frac{1-\alpha}{\alpha}}}\overset{d}=
{S^{-\alpha}_{\alpha,\theta+n}}\,{\beta_{\frac{\theta}{\alpha}+K_{n},\frac{n}{\alpha}-K_{n}}},
\label{jamesidPDrecurse2}
\end{equation}
which, as pointed out in James~\cite[Proposition 6.6 (iii)
]{JamesPGarxiv}, results in the following key distributional equality,
\begin{equation}
\prod_{j=1}^{n}\beta_{\frac{\theta+\alpha+j-1}{\alpha},\frac{1-\alpha}{\alpha}}
\overset{d}=\beta_{\frac{\theta}{\alpha}+K_{n},\frac{n}{\alpha}-K_{n}}.
\label{betaKid}
\end{equation}
We provide the following identities, which play a key role in the next section.

\begin{prop}\label{Propbetaid1}
The relations in~(\ref{jamesidPDrecurse2}) and~(\ref{betaKid}) lead to the following results.
\begin{enumerate}
\item[(i)]
The density of $\prod_{j=1}^{n}\beta_{\frac{\theta+\alpha+j-1}{\alpha},\frac{1-\alpha}{\alpha}}$ can be expressed as
\begin{equation}
\sum_{k=1}^{n}\mathbb{P}^{(n)}_{\alpha,\theta}(k)f_{\beta_{\frac{\theta}{\alpha}+k,\frac{n}{\alpha}-k}}(u)
\label{prodbetadensityid}.
\end{equation}
\item[(ii)]Let $\alpha=\frac1r$, for $r=2,3,\ldots$. There is an identity
\begin{equation}
\prod_{j=1}^{n}\beta^{r}_{r(\theta+j-1)+1,r-1}\overset{d}=\prod_{i=1}^{r-1}\beta_{\theta+\frac{i}{r},n}.
\label{betaprodsimple}
\end{equation}
\item[(iii)] When $\alpha=\frac12,$ there is the easily deduced fact that
$$
\prod_{j=1}^{n}\beta^{2}_{2(\theta+j)-1,1}\overset{d}=\beta_{\theta+\frac{1}{2},n}.
$$
\item[(iv)]When $\alpha=\frac13,$
$
\prod_{j=1}^{n}\beta^{3}_{3(\theta+j)-2,2}\overset{d}=\beta_{\theta+\frac{1}{3},n}\times\beta_{\theta+\frac{2}{3},n} .
$
\end{enumerate}
\end{prop}
\begin{proof}
(\ref{prodbetadensityid}) is immediate from~(\ref{betaKid}). It remains to establish~(ii). 
Specializing (\ref{jamesidPDrecurse2})
and~(\ref{stablegeninteger}) to the case of $\alpha=\frac1r$ yields
$$
\prod_{i=1}^{r-1}G_{\theta+\frac{i}{r}}
\overset{d}=\prod_{i=1}^{r-1}G_{\theta+n+\frac{i}{r}}\times \prod_{i=1}^{n}\beta^{r}_{r(\theta+j-1)+1,r-1}.
$$
Independence of the products on the right hand side and standard beta-gamma calculus concludes the result.
\end{proof}
\begin{rem}Although the variables in~(\ref{betaKid}) are moment determinate, it is difficult to establish~(\ref{betaKid}) via direct arguments involving moments as that itself constitutes non-obvious identities. One can work through the $\alpha=\frac12$ case with some efforts. In this sense, the identity~(\ref{jamesidPDrecurse2}) is crucial.
\end{rem}
\begin{rem} Goldschmidt and Haas~\cite[Lemma 1.2]{GoldschmidtHaas2015} show that as $n\rightarrow \infty,$
$$n^{\frac{1-\alpha}{\alpha}}\prod_{j=1}^{n}\beta^{\frac1\alpha}_{\frac{\theta+\alpha+j-1}{\alpha},\frac{1-\alpha}{\alpha}}\overset{a.s.}\rightarrow \alpha^{\frac1\alpha}S^{-1}_{\alpha,\theta},
$$
which implies that $n^{1-\alpha}S_{\alpha,\theta+n}\overset{a.s.}\rightarrow \alpha^{\frac1\alpha}.$ It is easy to show directly the equivalent behavior for  $\beta^{\frac1\alpha}_{\frac{\theta}{\alpha}+K_{n},\frac{n}{\alpha}-K_{n}}.$
\end{rem}

\section{Mittag-Leffler Markov chains under $\mathrm{PK}_{\alpha}(\gamma)$ }\label{sec:MLC}
This section provides distributional characterizations of nested families induced by fragmentation operations described in 
\cite{DGM} under the general $\mathrm{PK}_{\alpha}(\gamma)$ setting. Simplifications and various decompositions are facilitated by Proposition~\ref{Propbetaid1}.
Consider nested families of mass partitions $((P_{k,r});r\ge 0)$ coupled with a family of random variables $\mathbf{Z}:=(Z_{r},r\ge 0)$, satisfying for each integer $r\ge 0,$ as in (\ref{inverselocaltime}),
\begin{equation}
Z_{r}:=\frac{1}{\Gamma(1-\alpha)}\lim_{\epsilon\rightarrow
0}\epsilon^{\alpha}|\{\ell:P_{\ell,r}\ge \epsilon\}|~\mathrm{a.s.}
\label{inverselocaltimeZ}
\end{equation}
and forming a Markov chain with stationary transition density $Z_{1}|Z_{0}=z$ given by, for $y>z,$
\begin{equation}
\mathbb{P}(Z_{1}\in dy|Z_{0}=z)/dy=\frac{\alpha {(y-z)}^{\frac{1-\alpha}{\alpha}-1}yg_{\alpha}(y)}{\Gamma(\frac{1-\alpha}{\alpha})g_{\alpha}(z)}.
\label{transitiong}
\end{equation}
When $Z_{0}\overset{d}=S^{-\alpha}_{\alpha,\theta},$ it follows that, for each $r,$ $Z_{r}\overset{d}=S^{-\alpha}_{\alpha,\theta+r}$ and it satisfies the relation indicated in~(\ref{jamesidPDrecurse2}). Correspondingly, for each $r,$ 
\begin{equation}
\begin{array}{rcl}
(P_{k,r})&\sim &\mathrm{PD}(\alpha,\theta+r),\\
(P_{k,r})|Z_{r}=z_{r},\ldots, Z_{0}=z_{0} &\sim & \mathrm{PD}\big(\alpha|z^{-\frac1\alpha}_{r}\big).\\
\end{array}
\end{equation}
In these cases, the sequence may be referred to as a \textit{Mittag-Leffler Markov chain} with law denoted as $\mathbf{Z}\sim \mathrm{MLMC}(\alpha,\theta),$ as in~\cite{RembartWinkel2016}. The Markov chain is described prominently in various generalities in~\cite{GoldschmidtHaas2015,Haas,JamesPGarxiv,James2015,RembartWinkel2016} and arises in P\'olya urn and random graph/tree growth models as described in, for instance,~\cite{AldousCRTI,DevroyeBranch,FordD,GoldschmidtHaas2015,Haas,SvanteUrn,Mori,Pek2016Gengam,Pek2017jointpref,RembartWinkel2016,Hofstad}. 
In many of these cases, one considers the spacings
$$
(Z_{0},Z_{1}-Z_{0},Z_{2}-Z_{1},\ldots).
$$
See~\cite{JamesRoss,RembartWinkel2016} for more details. We will sometimes write the distribution of the coupled family 
as $((P_{k,r}),Z_{r};r\ge 0)\sim \mathrm{MLMC}(\alpha,\theta)$ with obvious meaning.
The Brownian case of $\alpha=\frac12,$ that is, $\mathrm{MLMC}(\frac12,\theta),$ results in the simplest distributional description whereby, $Z_{0}\overset{d}=2\sqrt{G_{\theta+\frac12}},$ and
\begin{equation}
(Z_{r};r\ge 0)\overset{d}=\left(2\sqrt{G_{\theta+\frac12}+\sum_{\ell=1}^{r}\mathbf{e}_{\ell}};r\ge 0\right).
\label{BCRT}
\end{equation}
When $\theta=\frac12,$ $Z_{0}\overset{d}=2\sqrt{\mathbf{e}_{0}},$ for $\mathbf{e}_{0}\sim \mathrm{exponential}(1),$ this relates to the construction of the Brownian CRT~\cite{AldousCRTI,AldousCRTIII}. See~\cite{Pek2017jointpref}
for No-Loop and Loop preferential attachment graph models corresponding to $\alpha=\frac12$ and $\theta=0,\frac12$, respectively. As pointed out in~\cite{James2015}, the natural extensions for general $\alpha$ are the cases $\theta=1-2\alpha$ and $\theta=1-\alpha,$ corresponding to models analyzed by~\cite{Mori,Hofstad}, see also~\cite{BerUribe,DevroyeBranch,Dgraphs}.

While $\mathbf{Z}\sim \mathrm{MLMC}(\alpha,\theta)$ arises by various constructions in the literature, we focus on its description, and more directly that of $((P_{k,r});r\ge 0),$ in terms of the $\mathrm{PD}(\alpha,1-\alpha)$ single-block size-biased fragmentation operation described in~\cite{DGM}~(see also \cite{BertoinGoldschmidt2004}, for $\alpha=0$) which has a dual coagulation operation. Precisely, let $(P_{k})\in \mathcal{P}_{\infty}$ denote a mass partition, let $\tilde{P}_{1}$ denote its first size-biased pick and let $(P_{k})_{1}:=(P_{k})/\tilde{P}_{1}$ denote the remainder. A $\mathrm{PD}(\alpha,1-\alpha)$ fragmentation of $(P_{k})$ is defined as 
$$
\widehat{\mathrm{Frag}}_{\alpha,1-\alpha}((P_{\ell})):=\mathrm{Rank}((P_{k})_{1},\tilde{P}_{1}(Q_{\ell}))\in \mathcal{P}_{\infty},
$$
where, independent of $(P_{k}),$ $(Q_{\ell})\sim \mathrm{PD}(\alpha,1-\alpha),$ and $\mathrm{Rank}(\cdot)$ denotes the ranked re-arrangment. Let $((Q^{(j)}_{\ell});j\ge 1)$ denote an independent collection of $\mathrm{PD}(\alpha,1-\alpha)$ mass partitions defining a sequence of independent fragmentation operators $(\widehat{\mathrm{Frag}}^{(j)}_{\alpha,1-\alpha}(\cdot);j\ge 1).$ It follows from~\cite{DGM} that a version of the family $((P_{k,r});r\ge 0)$ may be constructed by the recursive fragmentation, for $r=1,2,\ldots,$ 
\begin{equation}
(P_{\ell,r})=\widehat{\mathrm{Frag}}^{(r)}_{\alpha,1-\alpha}((P_{\ell,r-1}))=\widehat{\mathrm{Frag}}^{(r)}_{\alpha,1-\alpha}\circ\cdots\circ \widehat{\mathrm{Frag}}^{(1)}_{\alpha,1-\alpha}((P_{\ell,0})).
\label{fraglittle}
\end{equation}
We now describe the law of $((P_{k,r}),Z_{r};r\ge 0)$ when $(P_{k,0})\sim \mathrm{PK}_{\alpha}(\gamma).$ The forthcoming initial descriptions, although not well-known, follow readily from the Markovian structure of $(Z_{r};r\ge 0)$ dictated by~(\ref{transitiong}), and also~(\ref{jamesidPDrecurse2}). With the dependent structure dictated by~(\ref{transitiong}) and~(\ref{fraglittle}), it suffices to describe the marginal distribution of $((P_{k,r}),Z_{r})$ for each integer $r\ge 0.$
For $\gamma(dt)/dt=h(t)f_{\alpha}(t),$ define 
\begin{equation*}
\begin{split}
h_{0}(t)&=h(t),\\
h_{r}(t)&= {{t}^{-r}} \mathbb{E}\bigg[h\bigg(t\displaystyle \prod_{i=1}^{r}\beta^{-\frac1\alpha}_{\frac{\alpha+i-1}{\alpha},\frac{1-\alpha}{\alpha}}\bigg)\bigg]/\mathbb{E}[S^{-r}_{\alpha}],\quad r=1,2,\ldots,\\
\end{split}
\end{equation*}
and the probability measures, for each integer $r\ge 1,$ 
\begin{equation}
\gamma_{r}(dt)/dt=h_{r}(t)f_{\alpha}(t)=\mathbb{E}\bigg[h\bigg(t\prod_{i=1}^{r} \beta^{-\frac1\alpha}_{\frac{\alpha+i-1}{\alpha},\frac{1-\alpha}{\alpha}}\bigg)\bigg]f_{\alpha,r}(t).
\label{margSr}
\end{equation}

\begin{lem}\label{LemmaMLMC}
Consider the $\mathrm{MLMC}(\alpha,0)$ sequence of local times  $\mathbf{Z}:=(Z_{r}; r\ge 0)\overset{d}=(S^{-\alpha}_{\alpha,r};r\ge 0),$ satisfying (\ref{jamesidPDrecurse2}) with $\theta=0.$ Then, the law of the sequence $\mathbf{Z}|Z_{0}=y^{-\alpha}$ is denoted as $\mathrm{MLMC}(\alpha|y)$, where the joint behaviour is determined by~(\ref{transitiong}). Further mixing over the density $\gamma(dy)/dy:=h(y)f_{\alpha}(y)$ leads to the distribution denoted as $\mathrm{MLMC}^{[\gamma]}(\alpha).$ It follows that under this law, for each $r\ge 0,$ $Z^{-\frac1\alpha}_{r}$ has distribution $\gamma_{r},$ and hence the marginal density of $Z_{r}$ can be expressed as
\begin{equation}
g_{\alpha,r}(s;\gamma)=\mathbb{E}\bigg[h\bigg(s^{-\frac1\alpha} \prod_{i=1}^{r}\beta^{-\frac1\alpha}_{\frac{\alpha+i-1}{\alpha},\frac{1-\alpha}{\alpha}}\bigg)\bigg]g_{\alpha,r}(s).
\label{margZj}
\end{equation}
In addition, the corresponding $((P_{k,r});r\ge 0)$ is such that for each $r,$ $(P_{k,r})$ has distribution $\mathrm{PK}_{\alpha}(\gamma_{r})=\int_{0}^{\infty}\mathrm{PD}(\alpha|s^{-\frac1\alpha})g_{\alpha,r}(s;\gamma)ds.$
\end{lem}

Set
$$
\tilde{V}^{(\frac{n}{\alpha}-k)}_{\alpha,\frac{\theta}{\alpha}+k}(y)=\mathbb{E}\bigg[h\bigg(y\frac{S_{\alpha,\theta+n}}{\beta^{\frac1\alpha}_{k+\frac{\theta}{\alpha},\frac{n}{\alpha}-k}}\bigg)\bigg]=\mathbb{E}_{\alpha,\theta}[h(yS_{\alpha,\theta})|K_{n}=k],
$$
and thus, 
$
\tilde{V}^{(\frac{n}{\alpha}-k)}_{\alpha,k}(1)=V_{n,k}\frac{\alpha^{1-k}\Gamma(n)}{\Gamma(k)}.
$
We now provide a description of the corresponding EPPF's and 
the distributions of the numbers of blocks for the nested sequence of random partitions of $[n].$ 
\begin{prop}\label{EPPFMLMC}Consider $((P_{k,r}),Z_{r};r\ge 0)\sim \mathrm{MLMC}^{[\gamma]}(\alpha),$ where $(P_{k,0})\sim  
\mathrm{PK}_{\alpha}(\gamma)$ with EPPF expressed as
$$
\tilde{V}^{(\frac{n}{\alpha}-k)}_{\alpha,k}(1)\times p_{\alpha}(n_{1},\ldots,n_{k}).
$$
\begin{enumerate}
\item[(i)]For each $r\ge 0,$ the $\mathrm{EPPF}$ of $(P_{k,r})\sim  
\mathrm{PK}_{\alpha}(\gamma_{r})$ can be expressed as
$$
\mathbb{E}\bigg[\tilde{V}^{(\frac{n}{\alpha}-k)}_{\alpha,\frac{r}{\alpha}+k}\left(\prod_{i=1}^{r}\beta^{-\frac1\alpha}_{\frac{\alpha+i-1}{\alpha},\frac{1-\alpha}{\alpha}}\right)\bigg]p_{\alpha,r}(n_{1},\ldots,n_{k}).
$$
\item[(ii)] Let $(K_{n,r},r \ge 0)$ denote the increasing sequence in $r$, where, for each fixed $r,$ $K_{n,r}$ is the corresponding number of blocks in a $\mathrm{PK}_{\alpha}(\gamma_{r})$ partition of $[n]$. The probability mass function of $K_{n,r}$ is, for $k=1,\ldots,n$,
$$
\mathbb{P}^{(n)}_{\alpha|\gamma_{r}}(k)=\mathbb{E}\bigg[\tilde{V}^{(\frac{n}{\alpha}-k)}_{\alpha,\frac{r}{\alpha}+k}\left(\prod_{i=1}^{r} \beta^{-\frac1\alpha}_{\frac{\alpha+i-1}{\alpha},\frac{1-\alpha}{\alpha}}\right)\bigg]\mathbb{P}^{(n)}_{\alpha,r}(k),
$$
where, for $K_{n,0},$ $\mathbb{P}^{(n)}_{\alpha|\gamma}(k)=\tilde{V}^{(\frac{n}{\alpha}-k)}_{\alpha,k}(1) \times \mathbb{P}^{(n)}_{\alpha,0}(k)$.
\item[(iii)] As $n\rightarrow \infty,$ $n^{-\alpha}K_{n,r}\overset{a.s.}\rightarrow Z_{r},$ where $Z_{r}$ has density~(\ref{margZj}). 
\end{enumerate}
\end{prop}

\begin{proof}One may express $g_{\alpha,r}(s;\gamma)$ in~(\ref{margZj}) as $h_{r}\big(s^{-\frac1\alpha}\big )g_{\alpha}(s),$ where\break $h_{r}\big(s^{-\frac1\alpha}\big)= s^{\frac{r}\alpha}\mathbb{E} \bigg[h\bigg(s^{-\frac1\alpha}\prod_{i=1}^{r} \beta^{-\frac1\alpha}_{\frac{\alpha+i-1}{\alpha},\frac{1-\alpha}{\alpha}}\bigg)\bigg] /\mathbb{E}[S^{-r}_{\alpha}].$ Using 
(\ref{VEPPF2}), the EPPF may be expressed as 
$\mathbb{E}\bigg[h_{r} \bigg(S_{\alpha,n}\beta^{-\frac1\alpha}_{k,\frac{n}{\alpha}-k}\bigg)\bigg]
p_{\alpha}(n_{1},\ldots,n_{k}).$
Statements (i) and (ii) follow by a re-arrangement of terms and by summing over the EPPF in (i), respectively. Statement (iii) is a direct consequence of~\cite[Proposition 13]{Pit03}, see also \cite[Lemma 13
]{Pit06}.
\end{proof}

\subsection{Results for the Brownian case of $\alpha=\frac12$}\label{sec:Brownmixed}
We describe further simplifications in the case of $\alpha=\frac12$ using statement (iii) of Proposition~\ref{Propbetaid1}. We first introduce some additional notation and facts. Similar to~\cite{Pit03,Pit06,PitmanMax}, we specify $(\tilde{L}_{1,r};r\ge 0)\overset{d}=(Z_{r}/\sqrt{2};r\ge 0)$ to be the sequence of local times at $0$ up till time one of a nested family of generalized bridges $(B^{(r)}:=(B^{(r)}_{t}:t\in[0,1]);r\ge 0)$ whose excursion intervals correspond in (joint) distribution to the fragmentation operations~(\ref{fraglittle}). See~\cite{PitmanMax} for a formal description of the operations of \cite{DGM} at the level of processes $B^{(r)}.$ When $\mathbf{Z}\sim \mathrm{MLMC}(\frac12,\theta),$ $\tilde{L}_{1,r}:=L_{1,\theta+r}$ is such that 
$L^{2}_{1,\theta+r}\overset{d}=2G_{\theta+r+\frac12}\sim \chi^{2}_{2\theta+2r+1},$ where $\chi^{2}_{p}$ denotes a chi-squared distribution with $p$ degrees of freedom, having a density $f_{\chi^{2}_{p}}.$ The cases of $\theta=0$ and $\theta=\frac12$ correspond to cases where one can take $B^{(0)}$ to be respectively standard Brownian motion and Brownian bridge. Throughout $B_{1}$ denotes Brownian motion at time $1,$ having a standard normal density $\phi(\lambda)=\frac{1}{\sqrt{2\pi}}{\mbox e}^{-{\lambda^{2}}/{2}}$. $L_{1,\frac12}\overset{d}=\sqrt{2G_{1}}$ has a \textit{Rayleigh} distribution with density $f_{L_{1,\frac12}}(x)=x{\mbox e}^{-{x^{2}}/{2}},x>0$. It follows from statement (iii) of Proposition~\ref{Propbetaid1} that when $\alpha=\frac12,$ 
\begin{equation}
\gamma_{r}(dt)/dt=h_{r}(t)f_{\frac12}(t)=\mathbb{E}\bigg[h\bigg(t\beta^{-1}_{\frac{1}{2},r}\bigg)\bigg]f_{\frac12,r}(t),
\label{margSrBrown}
\end{equation}
leading to the following simplified description in this case.

\begin{prop}\label{GenBrownMLMC}Consider $((P_{k,r}),Z_{r};r\ge 0)\sim \mathrm{MLMC}^{[\gamma]}(\frac12),$ where $(P_{k,0})\sim  
\mathrm{PK}_{\frac12}(\gamma),$ and, generally, $(P_{k,r})\sim\mathrm{PK}_{\frac12}(\gamma_{r})$ is specified by (\ref{margSrBrown}).
\begin{enumerate}
\item[(i)]As a straightforward generalization of~(\ref{BCRT}), there is the equivalence in joint distribution, 
\begin{equation}
(\tilde{L}_{1,r};r\ge 0)\overset{d}=\left(\frac{Z_{r}}{\sqrt{2}};r\ge 0\right)\overset{d}=\left(\sqrt{\frac{Z^{2}_{0}}{2}+2\sum_{\ell=1}^{r}\mathbf{e}_{\ell}};r\ge 0\right).
\label{BCRTgen}
\end{equation}
where, for $w(y)=h\big(\frac1y\big),$ $\tilde{L}^{2}_{1,r}\overset{d}=Z^{2}_{r}/2$ has density 
$\mathbb{E}[w(2y\beta_{\frac12,r})]f_{\chi^{2}_{2r+1}}(y).$
\item[(ii)]For each $r\ge 0,$ the $\mathrm{EPPF}$ of $(P_{k,r})\sim  
\mathrm{PK}_{\frac12}(\gamma_{r})$ can be expressed as
$$
\mathbb{E}\bigg[\tilde{V}^{(2n-k)}_{\frac12,2r+k}\left(\frac{1}{\beta_{\frac12,r}}\right)\bigg]p_{\frac12,r}(n_{1},\ldots,n_{k}).
$$
\item[(iii)]The distribution of $K_{n,r}$ for a $\mathrm{PK}_{\frac12}(\gamma_{r})$ partition of $[n]$ can be expressed as
$$
\mathbb{P}^{(n)}_{\frac12|\gamma_{r}}(k)=\mathbb{E}\bigg[\tilde{V}^{(2n-k)}_{\frac12,2r+k}\left(\frac{1}{\beta_{\frac12,r}}\right)\bigg]
\frac{\Gamma(n)}{\Gamma(k)} \frac{(2r)_k}{(r)_n} \binom{2n-k-1}{n-1} 2^{k-2n}.
$$
\item[(iii)]As $n\rightarrow \infty,$ $n^{-\frac12}K_{n,r}\overset{a.s.}\rightarrow Z_{r}\overset{d}=2\sqrt{\frac{Z^{2}_{0}}{4} +\sum_{\ell=1}^{r}\mathbf{e}_{\ell}}$. 
\end{enumerate}
\end{prop}

\subsubsection{Brownian size-biased representations}
It is the case that, for each $r\geq 0,$ $(P_{k,r})|\tilde{L}_{1,r}=s$ has the distribution $\mathrm{PD}(\frac{1}{2}\big|\frac{1}{2}s^{-2}).$ Generically, let $(\tilde{P}_{\ell}(s),\ell\ge 1)$ denote the size-biased re-arrangement of a mass partition having law $\mathrm{PD}(\frac{1}{2}\big|\frac{1}{2}s^{-2}),$ where $\tilde{P}_{1}(s)$ is the first size-biased pick. Then, from~\cite[Corollary 3]{AldousPit}, see also \cite[Proposition 14]{Pit03}, one may set
$$
\tilde{P}_{1}(s)\overset{d}=\frac{B^{2}_{1}}{B^{2}_{1}+s^{2}}
$$
with corresponding density, for $0<p<1,$
\begin{equation}
f_{\tilde{P}_{1}}(p|s)=\frac{s}{\sqrt{2\pi}}p^{-\frac12}{(1-p)}^{-\frac32}{\mbox e}^{-\frac{{s}^{2}}{2}\left(\frac{p}{1-p}\right)},
\label{structural}
\end{equation}
and, for each $\ell\ge 1$,
\begin{equation}
\tilde{P}_{\ell}(s)=\frac{s^{2}}{s^{2}+R_{\ell-1}}-\frac{s^{2}}{s^{2}+R_{\ell}},
\label{condstick}
\end{equation}
where $R_{0}=0,$ and $R_{\ell}=\sum_{i=1}^{\ell}X_{i}$ for $X_{i}$ independent with common distribution equivalent to $B^{2}_{1}.$ These facts coupled with Proposition~\ref{GenBrownMLMC} lead to the next result.

\begin{cor}Let $(P_{k,r})\sim  
\mathrm{PK}_{\frac12}(\gamma_{r}),$ specified by (\ref{margSrBrown}), for fixed $r=0,1,2,\ldots.$ Its size-biased re-arrangement is equivalent in distribution to
$$
(\tilde{P}_{\ell}(\tilde{L}_{1,r}); \ell\ge 1),
$$
which is specified by~(\ref{BCRTgen}) and (\ref{condstick}). 
\end{cor}

\subsection{Mixture representations for $\mathrm{MLMC}^{[\gamma]}(\alpha)$}
For each fixed $r\ge 1,$ and $i=1,\ldots, r$, define the probability measures  
\begin{equation}
\gamma_{r,i}(dt)/dt=\frac{\mathbb{E}\left[h\left(t\beta^{-\frac1\alpha}_{i,\frac{r}{\alpha}-i}\right)\right]}{\tilde{V}^{(\frac{r}{\alpha}-i)}_{\alpha,i}(1)}
f_{\alpha,r}(t)
\label{subgammari}
\end{equation}
of random variables $Z^{-\frac1\alpha}_{r,i}$ with corresponding mass partition $(P^{(i)}_{k,r})\sim \mathrm{PK}_{\alpha}(\gamma_{r,i}),$ satisfying~(\ref{inverselocaltimeZ}) with $(r,i)$ in place of $r.$ Recall from 
Proposition~\ref{EPPFMLMC} that $\mathbb{P}^{(r)}_{\alpha|\gamma}(i)=\tilde{V}^{(\frac{r}{\alpha}-i)}_{\alpha,i}(1) \times \mathbb{P}^{(r)}_{\alpha,0}(i)$ is the distribution of $K_{r,0}$ based on a $\mathrm{PK}_{\alpha}(\gamma)$ partition of $[r].$ We now apply Proposition~\ref{Propbetaid1} to Proposition~\ref{EPPFMLMC} to obtain mixture representations of the quantities in Proposition~\ref{EPPFMLMC}. This also leads to the identification, and description of their properties, of other mass partitions, $(P^{(i)}_{k,r})\sim \mathrm{PK}_{\alpha}(\gamma_{r,i}),$ for each $r\ge 1.$
\begin{prop}\label{propbetaMLMCdecomp}
Suppose that $(P_{\ell,0})\sim \mathrm{PK}_{\alpha}(\gamma)$ with EPPF expressed as
$$
\tilde{V}^{(\frac{n}{\alpha}-k)}_{\alpha,k}(1)\times p_{\alpha}(n_{1},\ldots,n_{k}).
$$
The nested family $((P_{\ell,r}),Z_{r};r\ge 0)\sim \mathrm{MLMC}^{[\gamma]}(\alpha)$ has the following properties. 

\begin{enumerate}
\item[(i)]The marginal distribution of $Z^{-\frac1\alpha}_{r},$ $\gamma_{r}$ in~(\ref{margSr}), can be expressed as $\gamma_{r}(dt)=\sum_{i=1}^{r}\mathbb{P}^{(r)}_{\alpha|\gamma}(i)\gamma_{r,i}(dt)$ and 
$$
(P_{k,r})\sim \mathrm{PK}_{\alpha}(\gamma_{r})=\sum_{i=1}^{r}\mathbb{P}^{(r)}_{\alpha|\gamma}(i)\mathrm{PK}_{\alpha}(\gamma_{r,i}).
$$
\item[(ii)] The EPPF of $(P^{(i)}_{k,r})\sim\mathrm{PK}_{\alpha}(\gamma_{r,i})$ based on a partition of $[n]$ can be expressed as 
$$
p^{[\gamma_{r,i}]}_{\alpha}(n_{1},\ldots,n_{k})=\frac{\mathbb{E}\bigg[\tilde{V}^{(\frac{n}{\alpha}-k)}_{\alpha,\frac{r}{\alpha}+k}\bigg(\beta^{-\frac1\alpha}_{i,\frac{r}{\alpha}-i}\bigg)\bigg]}{\tilde{V}^{(\frac{r}{\alpha}-i)}_{\alpha,i}(1)}
p_{\alpha,r}(n_{1},\ldots,n_{k}).
$$
\item[(iii)] The EPPF of $(P_{\ell,r})\sim \mathrm{PK}_{\alpha}(\gamma_{r})$ based on a partition of $[n]$ can be expressed as 
$
\sum_{i=1}^{r}\mathbb{P}^{(r)}_{\alpha|\gamma}(i)p^{[\gamma_{r,i}]}_{\alpha}(n_{1},\ldots,n_{k}).
$
\end{enumerate}
\end{prop}
\begin{proof}With regards to Proposition~\ref{EPPFMLMC}, apply a special case of the identity in~(\ref{betaKid}), 
$
\prod_{j=1}^{r}\beta_{\frac{\alpha+j-1}{\alpha},\frac{1-\alpha}{\alpha}}
\overset{d}=\beta_{K_{r},\frac{r}{\alpha}-K_{r}},
$
where its density is given in~(\ref{prodbetadensityid}) of Proposition~\ref{Propbetaid1}, taking the form,
$
\sum_{i=1}^{r}\mathbb{P}^{(r)}_{\alpha,0}(i)f_{\beta_{i,\frac{r}{\alpha}-i}}(u).
$
\end{proof}

\section{Decomposition of special functions and first examples}\label{sec:decomp}
One of the un-exploited features of the Gibbs partitions, beyond the case of inducing various distributions over partitions, is that it provides a method of obtaining decompositions for a host of special functions connected to $f_{\alpha}.$ We further note that while these decompositions will now be shown to arise from basic probabilistic principles, their derivations from other perspectives would not be so transparent. Perhaps the simplest example
is, using~(\ref{PDgenKn}),
\begin{equation}
\mathbb{E}[S^{-\theta}_{\alpha}]=
\mathbb{E}_{\alpha,0}\left[\frac{\Gamma(n)\Gamma(\frac{\theta}{\alpha}+K_{n})}{\Gamma(\theta+n)\Gamma(K_{n})}\right]=\frac{\Gamma(\frac{\theta}{\alpha}+1)}{\Gamma(\theta+1)},
\label{PitmanKnmoment}
\end{equation}
which agrees with~\cite[exercise 3.2.9, p.66]{Pit06}. This is equivalent to 
\begin{equation}
\sum_{j=1}^{n}\mathbb{P}^{(n)}_{\alpha}(j)\frac{\Gamma(\frac{\theta}{\alpha}+j)}{\Gamma(j)}=\frac{\Gamma(\theta+n)\Gamma(\frac{\theta}{\alpha}+1)}{\Gamma(n)\Gamma(\theta+1)}.
\label{PitmanKnmoment2}
\end{equation}

\begin{lem}\label{decomplemma} Let $\varphi(t)$ denote an arbitrary non-negative function such that\break $\mathbb{E}[\varphi(S_{\alpha})]<\infty.$ Set 
$h(t)=\varphi(t)/\mathbb{E}[\varphi(S_{\alpha})],$ and thus $\gamma(dt)/dt=h(t)f_{\alpha}(t).$ For each $n\ge 1,$ there is the decomposition,
\begin{equation}
\mathbb{E}[\varphi(S_{\alpha})]=\sum_{k=1}^{n}\mathbb{E}[\varphi(S_{\alpha})|K_{n}=k]\mathbb{P}_{\alpha}(K_{n}=k)
\label{gdecomp},
\end{equation}
where $\mathbb{E}[\varphi(S_{\alpha})|K_{n}=k]$ can be expressed as
\begin{equation}
\mathbb{E}\bigg[\varphi\bigg(\frac{S_{\alpha,n}}{\beta^{1/\alpha}_{k,\frac{n}{\alpha}-k}}\bigg)\bigg]=
\frac{\alpha \Gamma(n)}{\Gamma(k)}\int_{0}^{\infty}\varphi(t)t^{-n}
\Ip{n-k\alpha}{f_\alpha}{(t)}\,dt.
\label{gdecompk}
\end{equation}
Then,
\begin{enumerate}
\item[(i)] $\mathbb{E}[\varphi(S_{\alpha})]:=\mathbb{E}[\varphi(S_{\alpha})|K_{1}=1]=\alpha\int_{0}^{\infty}\varphi(t)t^{-1}
\Ip{1-\alpha}{f_{\alpha}}{(t)}\,dt$.
\item[(ii)]
$V_{n,k}:=\displaystyle\int_{0}^{\infty}\mathbb{G}_{\alpha}^{(n,k)}(t)\gamma(dt)=\frac{\alpha^{k-1}\Gamma(k)}{\Gamma(n)} \times 
\frac{\mathbb{E}[\varphi(S_{\alpha})|K_{n}=k]}{\mathbb{E}[\varphi(S_{\alpha})]}.
$
\item[(iii)]
The recursion~(\ref{backward}) shows that $\mathbb{E}[\varphi(S_{\alpha})|K_{n}=k]$ can be expressed as
$$
\left(\frac{k\alpha}{n}\right)
\mathbb{E}[\varphi(S_{\alpha})|K_{n+1}=k+1]+
\left(1-\frac{k\alpha}{n}\right)
\mathbb{E}[\varphi(S_{\alpha})|K_{n+1}=k].
$$
\end{enumerate}
\end{lem}

\begin{rem}
When not considering constructions for $V_{n,k}$, both~(\ref{gdecomp}) and~(\ref{gdecompk})
apply for any integrable real or complex valued function $\varphi.$
\end{rem}

\subsection{Example: Decomposing generalized Mittag-Leffler functions in terms of scaled Prabhakar functions}
As we mentioned in the introduction, the Mittag-Leffler function plays an important role in fractional calculus as described in the book~\cite{GorenfloMittag}. Here we show that the Mittag-Leffler function and its generalizations pertinent to the $\mathrm{PD}(\alpha,\theta)$ distribution can be decomposed in terms of scaled versions of functions introduced by
~\cite{Prab}. We also show that Laplace transforms of $Z^{(\frac{\nu}{\alpha})}_{\alpha,\omega}$ are special cases of functions in~\cite{Prab} up to a constant of proportionality. This example is also inspired by some
results in~\cite{JamesLamp}. Recall that the
Mittag-Leffler function may be defined by
$$\mathrm{E}_{\alpha,1}(-\lambda)=\mathbb{E}\big[{\mbox
e}^{-\lambda
S^{-\alpha}_{\alpha}}\big]=\sum_{\ell=0}^{\infty}\frac{{(-\lambda)}^{\ell}}{\Gamma(\alpha\ell+1)}=\mathbb{E}\big[{\mbox
e}^{-\lambda^{\frac1\alpha}X_{\alpha}}\big],
$$
where, for $S'_{\alpha}\overset{d}=S_{\alpha},$ and otherwise independent, 
$
X_{\alpha}:={S_{\alpha}}/{S'_{\alpha}}.
$
Remarkably although $S_{\alpha}$ does not have a simple density,
except for $\alpha=\frac12$,~\cite{Zolotarev57}~(see also \cite{BPY,Lamperti,PY92} and
~\cite[exercise 4.2.1]{Chaumont}) shows that the density
of $X_{\alpha}$ is, for $y>0,$
\begin{equation}
f_{X_{\alpha}}(y)=\frac{\sin(\pi
\alpha)}{\pi}\frac{y^{\alpha-1}}{y^{2\alpha}+2\cos(\pi
\alpha)y^{\alpha}+1}\label{denX}.
\end{equation}
This coincides with  the integral representation
$$
\mathrm{E}_{\alpha,1}(-\lambda)=\frac{\sin(\pi
\alpha)}{\pi}\int_{0}^{\infty}\frac{{\mbox e}^{-\lambda^{1/\alpha}
y}y^{\alpha-1}}{y^{2\alpha}+2\cos(\pi \alpha)y^{\alpha}+1}dy.
$$
James~\cite[Section~3]{JamesLamp}, here we use a slight adjustment in notation, showed that, for $\theta>-\alpha$,
\begin{equation}
\mathbb{E}\big[{\mbox
e}^{-\lambda S^{-\alpha}_{\alpha,\theta}}\big]=\mathbb{E}\big[{\mbox
e}^{-\lambda^{1/\alpha}X_{\alpha,\theta}}\big]=\mathrm{E}^{(\frac{\theta}{\alpha}+1)}_{\alpha,\theta+1}(-\lambda),
\label{MittagLeffler}
\end{equation}
where $X_{\alpha,\theta}:=S_{\alpha}/S_{\alpha,\theta}$ are the Lamperti variables studied in~\cite{JamesLamp}, and
\begin{equation}
\mathrm{E}^{(\frac{\theta}{\alpha}+1)}_{\alpha,\theta+1}(-\lambda)=\sum_{\ell=0}^{\infty}\frac{{(-\lambda)}^{\ell}}{\ell!}\frac{
\Gamma(\frac{\theta}{\alpha}+1+\ell)\Gamma(\theta+1)}
{\Gamma(\frac{\theta}{\alpha}+1)\Gamma(\alpha
\ell+\theta+1)},\qquad
\theta>-\alpha,
\label{altM}
\end{equation}
which further reduces to 
$\mathrm{E}^{(\frac{\theta}{\alpha})}_{\alpha,\theta}(-\lambda)$, for $\theta>0$. We now extend these results for the general case of $\omega$ and $\nu.$
\begin{prop}\label{GenLaplaceMittag}Consider the random variables $Z^{(\frac{\nu}{\alpha})}_{\alpha,\omega}$ defined in~(\ref{jamesid2}). 
Their Laplace transforms are equal to
\begin{equation}
\mathrm{E}^{(\frac{\omega}{\alpha})}_{\alpha,\omega+\nu}(-\lambda)=
\sum_{\ell=0}^{\infty}\frac{{(-\lambda)}^{\ell}}{\ell!}\frac{\Gamma(\frac{\omega}{\alpha}+\ell)\Gamma(\omega+\nu)}{\Gamma(\frac{\omega}{\alpha})\Gamma(\alpha
\ell+\omega+\nu)}.
\label{altMGen}
\end{equation}
\end{prop}
\begin{proof}Using~(\ref{MittagLeffler}),
$$
\mathbb{E}\big[{\mbox e}^{-\lambda Z^{(\frac{\nu}{\alpha})}_{\alpha,\omega}}\big]=\mathbb{E}\bigg[{\mbox e}^{-\lambda^{1/\alpha}{\beta^{1/\alpha}_{\frac{\omega}{\alpha},\frac{\nu}{\alpha}}} X_{\alpha,\omega+\nu}}\bigg]
=\mathbb{E}\bigg[\mathrm{E}^{(\frac{\omega+\nu}{\alpha}+1)}_{\alpha,\omega+\nu+1}\big(-\lambda {\beta_{\frac{\omega}{\alpha},\frac{\nu}{\alpha}}}\big)\bigg].
$$
The result is concluded by substituting $\mathbb{E}\big[{\beta^{\ell}_{\frac{\omega}{\alpha},\frac{\nu}{\alpha}}}\big]=
\dfrac{\Gamma(\frac{\omega+\nu}{\alpha})\Gamma(\frac{\omega}{\alpha}+\ell)}{\Gamma(\frac{\omega}{\alpha})\Gamma(\frac{\omega+\nu}{\alpha}+\ell)}.
$
\end{proof}
We now show that the generalized Mittag-Leffler functions can be expressed in terms of special cases of the previous result. 

\begin{prop}\label{PropMittagdecomp} Following Lemma~\ref{decomplemma}, set $\varphi(t)={\mbox e}^{-\lambda t^{-\alpha}}t^{-\theta}/\mathbb{E}[S^{-\theta}_{\alpha}]$. Then,\break $\mathbb{E}[\varphi(S_{\alpha})]=\mathrm{E}^{(\frac{\theta}{\alpha}+1)}_{\alpha,\theta+1}(-\lambda),$ and there is the decomposition, for each fixed $\lambda>0,$
$$
\mathrm{E}^{(\frac{\theta}{\alpha}+1)}_{\alpha,\theta+1}(-\lambda)=
\sum_{k=1}^{n}\mathbb{P}^{(n)}_{\alpha,\theta}(k)\mathrm{E}^{(\frac{\theta}{\alpha}+k)}_{\alpha,\theta+n}(-\lambda)=\mathbb{E}_{\alpha,\theta}\bigg[\mathrm{E}^{(\frac{\theta}{\alpha}+K_{n})}_{\alpha,\theta+n}(-\lambda)\bigg],
$$
where $\mathrm{E}^{(\frac{\theta}{\alpha}+k)}_{\alpha,\theta+n}(-\lambda)=\mathbb{E}\bigg[\mathrm{E}^{(\frac{\theta+n}{\alpha})}_{\alpha,\theta+n}\bigg(-\lambda {\beta_{\frac{\theta}{\alpha}+k,\frac{n}{\alpha}-k}}\bigg)\bigg]$ can be expressed as
$$
\mathrm{E}^{(\frac{\theta}{\alpha}+k)}_{\alpha,\theta+n}(-\lambda)=
\sum_{\ell=0}^{\infty}\frac{{(-\lambda)}^{\ell}}{\ell!}\frac{\Gamma(\frac{\theta}{\alpha}+k+\ell)\Gamma(
\theta+n)}{\Gamma(\frac{\theta}{\alpha}+k)\Gamma(\alpha
\ell+\theta+n)},
$$
as read from~(\ref{altMGen}).
\end{prop}
\begin{proof}
The result follows by combining Proposition~\ref{GenLaplaceMittag} with Lemma~\ref{decomplemma}, where 
$$
\mathbb{E}[\varphi(S_{\alpha})|K_{n}=k]=\mathbb{E}\bigg[{\mbox e}^{-\lambda Z^{(\frac{n-k\alpha}{\alpha})}_{\alpha,\theta+k\alpha}}\bigg]\times \frac{\mathbb{E}[S^{-\theta}_{\alpha}|K_{n}=k]}{\mathbb{E}[S^{-\theta}_{\alpha}]}.
$$
\end{proof}

\subsection{Results for the corresponding mass partition}
It follows that, in this case, there is corresponding $(P_{k})\sim \mathrm{PK}_{\alpha}(\gamma),$ with
\begin{equation}
\gamma(dt)/dt=\frac{{\mbox e}^{-\lambda t^{-\alpha}}t^{-\theta}}{\mathbb{E}[S^{-\theta}_{\alpha}]\mathrm{E}^{(\frac{\theta}{\alpha}+1)}_{\alpha,\theta+1}(-\lambda)}f_{\alpha}(t)=\frac{{\mbox e}^{-\lambda t^{-\alpha}}f_{\alpha,\theta}(t)}{\mathrm{E}^{(\frac{\theta}{\alpha}+1)}_{\alpha,\theta+1}(-\lambda)}.
\label{GammaMittag}
\end{equation}
A change of variable leads to the exponentially tilted density of $S^{-\alpha}_{\alpha,\theta}$ given by
\begin{equation}
g^{(0)}_{\alpha,\theta}(s|\lambda):=\frac{{\mbox e}^{-\lambda s}g_{\alpha,\theta}(s)}{\mathrm{E}^{(\frac{\theta}{\alpha}+1)}_{\alpha,\theta+1}(-\lambda)},
\label{expMittag}
\end{equation}
which reflects the local time up till time one or the $\alpha$-diversity in this setting. We now describe the distribution of the relevant mass partition and its EPPF.

\begin{prop}\label{PropMittagEPPF}Consider the setting in Proposition~\ref{PropMittagdecomp}. The choice of $\varphi$ corresponds to a mass partition $(P_{\ell,0}(\lambda))\sim \mathrm{PK}_{\alpha}(\gamma),$ specified by~(\ref{GammaMittag}), or (\ref{expMittag}), with distribution otherwise denoted by
\begin{equation}
\mathbb{L}^{(0)}_{\alpha,\theta}(\lambda):=\int_{0}^{\infty}\mathrm{PD}(\alpha|s^{-\frac1\alpha})
g^{(0)}_{\alpha,\theta}(s|\lambda)ds.
\label{genMittaglambda}
\end{equation}

\begin{enumerate}
\item[(i)]The EPPF of a partition of $[n]$ is given by
\begin{equation}
p^{(0)}_{\alpha,\theta}(n_{1},\ldots,n_{k}|\lambda)=
\frac{\mathrm{E}^{(\frac{\theta}{\alpha}+k)}_{\alpha,\theta+n}(-\lambda)}{\mathrm{E}^{(\frac{\theta}{\alpha}+1)}_{\alpha,\theta+1}(-\lambda)}p_{\alpha,\theta}(n_{1},\ldots,n_{k}).
\label{MittagEPPF}
\end{equation}
\item[(ii)]The distribution of the number of blocks, $K_{n}(\lambda),$ is 
 \begin{equation}
\mathbb{P}^{(n)}_{\alpha|\gamma}(k):=\omega^{(n)}_{\alpha,\theta}(k|\lambda)=
\frac{\mathrm{E}^{(\frac{\theta}{\alpha}+k)}_{\alpha,\theta+n}(-\lambda)}{\mathrm{E}^{(\frac{\theta}{\alpha}+1)}_{\alpha,\theta+1}(-\lambda)}\mathbb{P}^{(n)}_{\alpha,\theta}(k).
\label{MittagEPPFKn}
\end{equation}
\item[(iii)]$n^{-\alpha}K_{n}(\lambda)\overset{a.s.}\rightarrow Z_{0}(\lambda),$ where $Z^{-\frac1\alpha}_{0}(\lambda)$ has density in~(\ref{GammaMittag}).
\end{enumerate}
\end{prop}
\begin{rem}For clarity, in~(\ref{MittagEPPF}),
\begin{equation}
\frac{\alpha^{1-k}\Gamma(n)}{\Gamma(k)}V_{n,k}=\frac{\mathrm{E}^{(\frac{\theta}{\alpha}+k)}_{\alpha,\theta+n}(-\lambda)}{\mathrm{E}^{(\frac{\theta}{\alpha}+1)}_{\alpha,\theta+1}(-\lambda)}\frac{\mathbb{E}[S^{-\theta}_{\alpha}|K_{n}=k]}{
\mathbb{E}[S^{-\theta}_{\alpha}]}.
\label{MittagVnk}
\end{equation}
\end{rem}
Using the addition rules associated with the EPPF~(\ref{MittagEPPF}) leads to the following property.

\begin{cor}For positive integers $k\leq n$, and $\theta>-\alpha$, 
$$
\mathrm{E}^{(\frac{\theta}{\alpha}+k)}_{\alpha,\theta+n}(-\lambda)=\left(\frac{\theta+k\alpha}{\theta+n}\right)\mathrm{E}^{(\frac{\theta}{\alpha}+k+1)}_{\alpha,\theta+n+1}(-\lambda)+
\left(1-\frac{\theta+k\alpha}{\theta+n}\right)\mathrm{E}^{(\frac{\theta}{\alpha}+k)}_{\alpha,\theta+n+1}(-\lambda).
$$
\end{cor} 
\begin{rem}\label{MittagHermite}
When $\alpha=\frac12,$ as remarked in~\cite[Remark 3.1, p.1320]{JamesLamp}, from the setting in~\cite[Section~8]{Pit03}, there is the identity
\begin{eqnarray*}
\mathbb{E}_{\frac{1}{2},0} \bigg[{\big(1-\tilde{G}_{1}\big)}^{\theta+\frac{1}{2}}\bigg|L_{1}=\lambda\bigg]
&=&\mathbb{E}\big[{|B_{1}|}^{2\theta+1}\big]\tilde{h}_{-(2\theta+1)}(\lambda)\\
&=& \mathrm{E}^{(2\theta+1)}_{\frac{1}{2},\theta+1}\left(-\frac{\lambda}{\sqrt{2}}\right),
\end{eqnarray*} 
where the first expression represents the moments of the meander length, $(1-\tilde{G}_{1}),$ of a Brownian motion  $B:=(B_{t}:t\in [0,1]),$ conditioned on its local time $L_{1},$ as in Pitman~\cite[eq. (88)]{Pit03}, $\mathbb{E}\big[{|B_{1}|}^{2\theta+1}\big]$ replaces the incorrect value $\mathbb{E}\big[{|B_{1}|}^{\theta+\frac{1}{2}}\big]$ in~\cite[Remark 3.1, p.1320]{JamesLamp}, and $\tilde{h}_{-(2\theta+1)}(\lambda)$ denotes a Hermite function of index $-(2\theta+1)$.
\end{rem}
\begin{rem}The function
introduced by~\cite{Prab} takes the form
\begin{equation}
\mathrm{\tilde{E}}^{\kappa}_{\rho,\mu}(-\lambda)=\sum_{\ell=0}^{\infty}\frac{{(-\lambda)}^{\ell}}{\ell!}\frac{(\kappa)_{\ell}}{\Gamma(\rho
\ell+\mu)}\label{Mittaggen2},
\end{equation}
where $\rho,\mu,\kappa\in \mathbb{C}, $ and $Re(\rho)>0.$
It follows that 
(\ref{altMGen}) may be expressed as
$$
\mathrm{E}^{(\frac{\omega}{\alpha})}_{\alpha,\omega+\nu}(-\lambda)=
\Gamma(\omega+\nu)\mathrm{\tilde{E}}^{\frac{\omega}{\alpha}}_{\alpha,\omega+\nu}(-\lambda).
$$
See Gorenflo, Kilbas, Mainardi, and Rogosin
\cite[Chapter 5]{GorenfloMittag} for more discussion on these functions. 
\end{rem}

\subsection{More general Mittag-Leffler functions}
There are a multitude of generalizations of the Mittag-Leffler function in the literature. \cite[Chapter 5]{GorenfloMittag}, although not exhaustive, describes quite a few. Within our exposition, we shall encounter the $3m$-parametric Mittag-Leffler function~(see~\cite[eq. (6.3.8)
]{GorenfloMittag}), which generalizes the Prabhakar's function~(\ref{Mittaggen2}), defined as
\begin{equation}
\mathrm{\tilde{E}}^{(\kappa_{i}),m}_{(\rho_{i}),(\mu_{i})}(-\lambda)=\sum_{\ell=0}^{\infty}\frac{{(-\lambda)}^{\ell}}{\ell!}\frac{\prod_{i=1}^{m}(\kappa_{i})_{\ell}}{\prod_{i=1}^{m}\Gamma(\rho_{i}
\ell+\mu_{i})}\label{Mittaggen3m},
\end{equation}
where $\rho_i,\mu_i,\kappa_i\in \mathbb{C}$ and $Re(\rho_i)>0,$ for $i=1,\ldots,m.$

\begin{lem}\label{3mMittag}Let $\beta_{\frac{\omega_{j}}{\alpha},\frac{\nu_{j}}{\alpha}}$, for $j=1,\ldots,r$, denote $r$ independent beta random variables. Then, 
$$
\mathbb{E}\bigg[\mathrm{E}^{(\frac{\omega_{0}}{\alpha})}_{\alpha,\omega_{0}+\nu_{0}}\bigg(-\lambda\prod_{j=1}^{r}\beta_{\frac{\omega_{j}}{\alpha},\frac{\nu_{j}}{\alpha}}\bigg)\bigg]:=
\sum_{\ell=0}^{\infty}\frac{{(-\lambda)}^{\ell}}{\ell!}\frac{\Gamma(\omega_{0}+\nu_{0})}{\Gamma(\alpha\ell+\omega_{0}+\nu_{0})}\frac{\prod_{i=0}^{r}(\frac{\omega_{i}}{\alpha})_{\ell}}{\prod_{i=1}^{r}
(\frac{\omega_{i}+\nu_{i}}{\alpha})_{\ell}
},
$$ 
is proportional to the $3m$-parametric Mittag-Leffler function in~(\ref{Mittaggen3m}) 
with $m=r+1,\kappa_i = \frac{\omega_i}\alpha$ and $\mu_{i}=\frac{\rho_i(\omega_{i}+\nu_{i})}\alpha,$ for $i=1,\ldots, r+1$, $\rho_{i}=\left\{\begin{array}{cl}1,&i=1,\ldots,r\\\alpha, &i=r+1\end{array}\right.$, and the constant of proportionality is $\prod_{i=1}^{r+1} \Gamma(\mu_i)$.
\end{lem}

\section{Another view of the Mittag-Leffler function Gibbs class}\label{sec:MLCmittag}
Propositions~\ref{PropMittagdecomp} and \ref{PropMittagEPPF} present a nice illustration of the idea of how to obtain interesting decompositions of special functions in a systematic manner. Furthermore, the results indicate a non-obvious EPPF based on notable special functions (in this case, Mittag-Leffler type functions) that appear in the broad literature~\cite{GorenfloMittag}, and connect this with a mass partition $(P_{\ell}(\lambda)) \sim\mathbb{L}^{(0)}_{\alpha,\theta}(\lambda) $ as in~(\ref{genMittaglambda}). That is,  $\mathrm{PK}_{\alpha}(\gamma)$, with
$$
h(s^{-\frac1\alpha})=\frac{{\mbox e}^{-\lambda s}s^{\frac\theta\alpha}}{\mathbb{E}[S^{-\theta}_{\alpha}]\mathrm{E}^{(\frac{\theta}{\alpha}+1)}_{\alpha,\theta+1}(-\lambda)}.
$$
In this section, we offer two related derivations of this Gibbs class, based on conditioning on a variable $\mathscr{M}^{(\eta)}_{\alpha,\theta}:=G_{\frac{\theta-\eta}{\alpha}}{S^{\alpha}_{\alpha,\eta}},$ defined for $\theta>\eta>-\alpha,$ and based on waiting times with marginals $G_{j}S^{\alpha}_{\alpha,\theta}\overset{d}=\mathscr{M}^{(\theta)}_{\alpha,\theta+j\alpha}, j=1,2,\ldots,$ from a mixed Poisson waiting time perspective following~\cite{PitmanPoissonMix}. We also define an $\mathrm{MLMC}^{[\gamma]}(\alpha)$ class where $(P_{\ell,0}(\lambda))$ has law (\ref{genMittaglambda}).   
We then specialize the results to the interesting Brownian case of $\alpha=\frac12,$ where more explicit results are obtained. As a highlight, in the case of $\alpha=\frac12,$ we focus on explicit properties of a variable which is related to variables appearing in~\cite{ChassaingJanson,Pek2016,Pit99local}.

\subsection{Conditioning on $\mathscr{M}^{(\eta)}_{\alpha,\theta}$ for $\theta>\eta>-\alpha$}
\begin{prop}\label{PropGenMittaglaw}Let $(P_{\ell})\sim \mathrm{PD}(\alpha,\eta)$ with corresponding local time $S^{-\alpha}_{\alpha,\eta},$ for $\eta>-\alpha$, having density $g_{\alpha,\eta}(s).$ 
Let $S^{-\alpha}_{\alpha,\theta}$ be a variable with density $g_{\alpha,\theta}(s),$ for $\theta>\eta.$
Independent of $(P_{\ell})$, consider a variable
$G_{\frac{\theta-\eta}{\alpha}}\sim \mathrm{Gamma}(\frac{\theta-\eta}{\alpha},1)$, and define the random variable
\begin{equation}
\mathscr{M}^{(\eta)}_{\alpha,\theta}=G_{\frac{\theta-\eta}{\alpha}}{S^{\alpha}_{\alpha,\eta}}
={\left(G^{\frac1\alpha}_{\frac{\theta-\eta}{\alpha}}{S_{\alpha,\eta}}\right)}^{\alpha}.
\label{Mittagrv}
\end{equation}
\begin{enumerate}
\item[(i)]The density of $\mathscr{M}^{(\eta)}_{\alpha,\theta}$ is, for $\lambda>0,$
\begin{equation}
\frac{\lambda^{\frac{\theta-\eta}{\alpha}-1}\mathbb{E}[S^{-\theta}_{\alpha}]}
{\Gamma(\frac{\theta-\eta}{\alpha})\mathbb{E}[S^{-\eta}_{\alpha}]}\mathrm{E}^{(\frac{\theta}{\alpha}+1)}_{\alpha,\theta+1}(-\lambda).
\label{PillaiDist}
\end{equation}
\item[(ii)]The conditional density of $S^{-\alpha}_{\alpha,\eta}|\mathscr{M}^{(\eta)}_{\alpha,\theta}=\lambda$ is
$$
{g}^{(0)}_{\alpha,\theta}(s|\lambda)=\frac{{\mbox e}^{-\lambda s}g_{\alpha,\theta}(s)}{\mathrm{E}^{(\frac{\theta}{\alpha}+1)}_{\alpha,\theta+1}(-\lambda)}.
$$
\item[(iii)]$(P_{\ell})|S^{-\alpha}_{\alpha,\eta}=s,\mathscr{M}^{(\eta)}_{\alpha,\theta}=\lambda$ has distribution $\mathrm{PD}(\alpha|s^{-\frac1\alpha})$.
\item[(iv)]$(P_{\ell})|\mathscr{M}^{(\eta)}_{\alpha,\theta}=\lambda$ has the distribution $\mathbb{L}^{(0)}_{\alpha,\theta}(\lambda)$ in~(\ref{genMittaglambda})
with $\mathrm{EPPF}$ in~(\ref{MittagEPPF}).
\end{enumerate}
\end{prop}
\begin{proof}Using~(\ref{Mittagrv}) and the density of $S_{\alpha,\eta},$ $f_{\alpha,\eta},$ it follows that the joint density of $(S_{\alpha,\eta},\mathscr{M}^{(\eta)}_{\alpha,\theta})$ is proportional to
$$
\lambda^{\frac{\theta-\eta}{\alpha}-1}s^{-\theta+\eta}{\mbox e}^{-\lambda s^{-\alpha}}f_{\alpha,\eta}(s),
$$
leading to statements (i) and (ii). Statements (iii) and (iv) follow, since by construction $(P_{\ell})|S_{\alpha,\eta}=t, \mathscr{M}^{(\eta)}_{\alpha,\theta}=\lambda $ is $\mathrm{PD}(\alpha|t).$
\end{proof}

\begin{rem}When $\eta=0,$ $\mathscr{M}^{(0)}_{\alpha,\theta}=G_{\frac{\theta}{\alpha}}{S^{\alpha}_{\alpha}}
={\left(G^{\frac1\alpha}_{\frac{\theta}{\alpha}}{S_{\alpha}}\right)}^{\alpha},$ where 
$G_{\theta}X_{\alpha,\theta}\overset{d}=G^{\frac1\alpha}_{\frac{\theta}{\alpha}}{S_{\alpha}}\overset{d}=\hat{S}_{\alpha}(G_{\frac{\theta}{\alpha}})$ are generalized positive Linnik variables as described in \cite[Section~2.3, p.1314]{JamesLamp}. See also~\cite{Bondesson1992,Pillai}.
\end{rem}

\subsection{Results for $\mathrm{MLMC}(\alpha,\eta)$ nested families conditioned on $\mathscr{M}^{(\eta)}_{\alpha,\theta}=\lambda$}
Here we now consider $(Z_{r};r\ge 0)\overset{d}=(S^{-\alpha}_{\alpha,\eta+r};r\ge 0)\sim \mathrm{MLMC}(\alpha,\eta)$ conditioned on $\mathscr{M}^{(\eta)}_{\alpha,\theta}=\lambda.$ We denote the $\mathrm{MLMC}$ sequence with this conditional law as $(Z_{r}(\lambda);r\ge 0),$ where $Z_{0}(\lambda)$ has density ${g}^{(0)}_{\alpha,\theta}(s|\lambda).$
In this section, 
$$
h_{r}(t)=\frac{t^{-(\theta+r)}\mathbb{E}\left[{\mbox e}^{-\lambda t^{-\alpha}\prod_{i=1}^{r}\beta_{\frac{\alpha+i-1}{\alpha},\frac{1-\alpha}{\alpha}}}\prod_{i=1}^{r}\beta^{\frac\theta\alpha}_{\frac{\alpha+i-1}{\alpha},\frac{1-\alpha}{\alpha}}\right]}{\mathbb{E}[S^{-\theta}_{\alpha}]\mathbb{E}[S^{-r}_{\alpha}]\mathrm{E}^{(\frac{\theta}{\alpha}+1)}_{\alpha,\theta+1}(-\lambda)}.
$$
Hence, after some manipulations, it follows that~(\ref{margSr}) in this case is
\begin{equation}
\gamma_{r}(dt)/dt= \frac{\mathbb{E}\left[{\mbox e}^{-\lambda t^{-\alpha}\prod_{i=1}^{r}\beta_{\frac{\theta+\alpha+i-1}{\alpha},\frac{1-\alpha}{\alpha}}}\right]}{\mathrm{E}^{(\frac{\theta}{\alpha}+1)}_{\alpha,\theta+1}(-\lambda)}f_{\alpha,\theta+r}(t),
\label{gammarMittag}
\end{equation}
where $\mathbb{E}\big[{\mbox e}^{-\lambda s\prod_{i=1}^{r}\beta_{\frac{\theta+\alpha+i-1}{\alpha},\frac{1-\alpha}{\alpha}}}\big]=\mathbb{E}_{\alpha,\theta}\big[{\mbox e}^{-\lambda s\beta_{\frac{\theta}{\alpha}+K_{r},\frac{r}{\alpha}-K_{r}}}\big].$

\begin{prop}\label{MittagEPPFMLMC}Suppose that for $\theta>\eta>-\alpha,$ $(P_{\ell,0})\sim \mathrm{PD}(\alpha,\eta),$ with corresponding local time $Z_{0}\overset{d}=S^{-\alpha}_{\alpha,\eta}.$ Furthermore, consider the family $((P_{\ell,r}),Z_{r};r\ge 0)\sim \mathrm{MLMC}(\alpha,\eta).$ Let $(P_{\ell,0}(\lambda),Z_{0}(\lambda))$ denote the mass partition and local time corresponding to the distribution of $(P_{\ell,0})|\mathscr{M}^{(\eta)}_{\alpha,\theta}=\lambda$, which is described in~(\ref{genMittaglambda}). The family of mass partitions $((P_{\ell,r});r\ge 0)|\mathscr{M}^{(\eta)}_{\alpha,\theta}=\lambda$ has the distribution of the family of mass partitions
$(P_{\ell,r}(\lambda);r\ge 0),$ defined by the recursive fragmentation, for $r=1,2,\ldots,$ 
$$
(P_{\ell,r}(\lambda))=\widehat{\mathrm{Frag}}^{(r)}_{\alpha,1-\alpha}((P_{\ell,r-1}(\lambda))),
$$
with corresponding local times denoted as $(Z_{r}(\lambda);r\ge 0).$
Then, for each fixed $r\ge 0,$ the followings hold.
\begin{enumerate}
\item[(i)] The marginal density of $Z_{r}(\lambda),$ which is the density of $Z_{r}|\mathscr{M}^{(\eta)}_{\alpha,\theta}=\lambda,$ is given by
\begin{equation}
g^{(r)}_{\alpha,\theta}(s|\lambda):=\frac{\mathbb{E}\big[{\mbox e}^{-\lambda s\prod_{i=1}^{r}\beta_{\frac{\theta+\alpha+i-1}{\alpha},\frac{1-\alpha}{\alpha}}}\big]}
{\mathrm{E}^{(\frac{\theta}{\alpha}+1)}_{\alpha,\theta+1}(-\lambda)}g_{\alpha,\theta+r}(s).
\label{margZlambdar}
\end{equation}
The density of $\big(Z_{r}(\lambda)\big)^{-\frac1\alpha}$ is denoted as $f^{(r)}_{\alpha,\theta}(t|\lambda)$.
\item[(ii)] The marginal distribution of $(P_{\ell,r}(\lambda))\sim \mathrm{PK}_{\alpha}(\gamma_{r}),$ specified by~(\ref{gammarMittag}), can be expressed as
$$
\mathbb{L}^{(r)}_{\alpha,\theta}(\lambda)=\int_{0}^{\infty}\mathrm{PD}(\alpha|s^{-\frac1\alpha})g^{(r)}_{\alpha,\theta}(s|\lambda)ds.
$$
\item[(iii)] The $\mathrm{EPPF}$ of $(P_{\ell,r}(\lambda))$ can be expressed as
\begin{equation}
\frac{\mathbb{E}\left[\mathrm{E}^{(\frac{\theta+r}{\alpha}+k)}_{\alpha,\theta+r+n}\big(-\lambda
\prod_{i=1}^{r}\beta_{\frac{\theta+\alpha+i-1}{\alpha},\frac{1-\alpha}{\alpha}}
\big)\right]}{\mathrm{E}^{(\frac{\theta}{\alpha}+1)}_{\alpha,\theta+1}(-\lambda)}
\times p_{\alpha,\theta+r}(n_{1},\ldots,n_{k}),
\end{equation}
where the numerator can be further explicitly expressed in terms of a\break $3(r+1)$-parametric Mittag-Leffler function using Lemma~\ref{3mMittag}.
\end{enumerate}
\end{prop}


\subsubsection{Mixture representations}
For each fixed $r\ge 1,$ and $i=1,\ldots, r$, define the probability measures $\gamma_{r,i}$ of random variables $\big(Z_{r,i}(\lambda)\big)^{-\frac1\alpha}$ corresponding to the appropriate specification of~(\ref{subgammari}) in this setting as   
\begin{equation}
\gamma_{r,i}(dt)/dt=\frac{_1F_1(\frac{\theta}{\alpha}+i;\frac{\theta}{\alpha}+\frac{r}{\alpha};-\lambda t^{-\alpha})}{\mathrm{E}^{(\frac{\theta}{\alpha}+i)}_{\alpha,\theta+r}(-\lambda)}f_{\alpha,\theta+r}(t)
\label{subgammariMittag},
\end{equation}
where
$
{_1F_1\left(\frac{\theta}{\alpha}+i;\frac{\theta}{\alpha}+\frac{r}{\alpha};-\lambda s\right)}=\mathbb{E}\left[{\mbox e}^{-\lambda s\beta_{\frac{\theta}{\alpha}+ i,\frac{r}{\alpha}-i}}\right]
$
is a confluent hypergeometric\break function of the first kind. Furthermore, $\mathbb{P}^{(r)}_{\alpha|\gamma}(i):=\omega^{(r)}_{\alpha,\theta}(i|\lambda),$ as in~(\ref{MittagEPPFKn}).
The next result follows from an application of Proposition~\ref{propbetaMLMCdecomp}.
\begin{prop}
Consider the same settings as in Proposition~\ref{MittagEPPFMLMC}, and $\gamma_{r,i}$ specified by
(\ref{subgammariMittag}). Then, for $r\ge 1,$ and $i=1,\ldots,r,$
\begin{enumerate}
\item[(i)]$(P_{\ell,r}(\lambda))\sim 
\mathbb{L}^{(r)}_{\alpha,\theta}(\lambda)=\sum_{i=1}^{r}\omega^{(r)}_{\alpha,\theta}(i|\lambda)\mathrm{PK}_{\alpha}(\gamma_{r,i})$.
\item[(ii)] The EPPF of $(P^{(i)}_{k,r})\sim\mathrm{PK}_{\alpha}(\gamma_{r,i})$ based on a partition of $[n]$ can be expressed as 
\begin{equation}
\frac{\mathbb{E}\left[\mathrm{E}^{(\frac{\theta+r}{\alpha}+k)}_{\alpha,\theta+n+r}\big(-\lambda\beta_{\frac{\theta}{\alpha}+ i,\frac{r}{\alpha}-i}\big)\right]}
{\mathrm{E}^{(\frac{\theta}{\alpha}+i)}_{\alpha,\theta+r}(-\lambda)}
p_{\alpha,\theta+r}(n_{1},\ldots,n_{k}),
\label{BrownianMittagEPPFj}
\end{equation}
where the expectation at the numerator equals
$$
\sum_{\ell=0}^{\infty}\frac{{(-\lambda)}^{\ell}}{\ell!}\frac{\Gamma(\theta+n+r)}{\Gamma(\alpha\ell+\theta+n+r)}\frac{(\frac{\theta+r}{\alpha}+k)_{\ell}(\frac{\theta}{\alpha}+i)_{\ell}}{(\frac{\theta+r}{\alpha})_{\ell}
},
$$
and corresponds to a special case of Lemma~\ref{3mMittag}.
\end{enumerate}
\end{prop}

\subsection{An interesting distributional result}
We now highlight the following interesting property. 
\begin{prop}\label{Propprojection}
Consider the family $\big((P_{\ell,r}),S^{-\alpha}_{\alpha,\eta+r}\big)\sim \mathrm{MLMC}(\alpha,\eta),$ and let\break $\big((P_{\ell,r}(\lambda)),S^{-\alpha}_{\alpha,r}(\lambda)\big)$ denote the random quantities corresponding to the conditional distribution $\big((P_{\ell,r}),S^{-\alpha}_{\alpha,\eta+r}\big)\big|\mathscr{M}^{(\eta)}_{\alpha,\theta}=\lambda.$ Then, for any $\theta'$ such that $\theta>\theta'\ge\eta,$ mixing $\lambda$ over the density,
$$
f_{\mathscr{M}^{(\theta')}_{\alpha,\theta}}(\lambda)=
\frac{\lambda^{\frac{\theta-\theta'}{\alpha}-1}\mathbb{E}[S^{-\theta}_{\alpha}]}
{\Gamma(\frac{\theta-\theta'}{\alpha})\mathbb{E}[S^{-\theta'}_{\alpha}]}\mathrm{E}^{(\frac{\theta}{\alpha}+1)}_{\alpha,\theta+1}(-\lambda),
$$
of a random variable equivalent in distribution to  $\mathscr{M}^{(\theta')}_{\alpha,\theta}=
G_{\frac{\theta-\theta'}{\alpha}}{S^{\alpha}_{\alpha,\theta'}}
$ leads to 
$$
\bigg((P_{\ell,r}(\mathscr{M}^{(\theta')}_{\alpha,\theta})),S^{-\alpha}_{\alpha,r}(\mathscr{M}^{(\theta')}_{\alpha,\theta})\bigg)\sim \mathrm{MLMC}(\alpha,\theta').
$$
That is, for $r=0,1,2,\ldots,$ $\big(P_{\ell,r}(\mathscr{M}^{(\theta')}_{\alpha,\theta})\big)\sim \mathrm{PD}(\alpha,\theta'+r)$ and 
$S^{-\alpha}_{\alpha,r}(\mathscr{M}^{(\theta')}_{\alpha,\theta})\overset{d}=S^{-\alpha}_{\alpha,\theta'+r}.$
\end{prop}
\begin{proof}The result follows by a simple gamma integral calculation which shows that the density of $Z_{0}\overset{d}=S^{-\alpha}_{\alpha,0}(\mathscr{M}^{(\theta')}_{\alpha,\theta})$ can be expressed as
$$
\int_{0}^{\infty}\frac{{\mbox e}^{-\lambda s}g_{\alpha,\theta}(s)}{\mathrm{E}^{(\frac{\theta}{\alpha}+1)}_{\alpha,\theta+1}(-\lambda)}f_{\mathscr{M}^{(\theta')}_{\alpha,\theta}}(\lambda)d\lambda=g_{\alpha,\theta'}(s).
$$
\end{proof}

\subsection{A mixed Poisson process viewpoint with $A=S^{-\alpha}_{\alpha,\theta}\sim \mathrm{ML}(\alpha,\theta)$}
We now revisit the mixed Poisson framework in~\cite{PitmanPoissonMix}, as described in Section~\ref{sec:mixedPoisson}. In particular, here we set $A=S^{-\alpha}_{\alpha,\theta}\sim \mathrm{ML}(\alpha,\theta),$ the local time or $\alpha$-diversity of a $\mathrm{PD}(
\alpha,\theta)$ distribution, whereas Section~\ref{sec:mixedPoisson} focused on the inverse local time $S_{\alpha}.$ While this explicit case does not yet appear there, \cite{PitmanPoissonMix} is timely in terms of helping us offer another very interesting interpretation of the Mittag-Leffler distributions we encounter in terms of waiting times. As in Section~\ref{sec:mixedPoisson}, $G_{j}:=\sum_{i=1}^{j}\mathbf{e}_{i},$ $j\ge 1,$
$(\mathbf{e}_{i})$ are iid standard exponential random variables.
Let $\mathcal{N}_{\alpha,\theta}(\lambda,j)=\big(N_{S^{-\alpha}_{\alpha,\theta}}(y);0\leq y\leq \lambda, N_{S^{-\alpha}_{\alpha,\theta}}(\lambda)=j\big)$.

\begin{prop}\label{PropmixedL}Suppose that $(P_{\ell})\sim\mathrm{PD}(\alpha,\theta),$ with local time/$\alpha$-diversity $S^{-\alpha}_{\alpha,\theta}.$ Set $A=S^{-\alpha}_{\alpha,\theta},$ with density $g_{\alpha,\theta}(s),$  for $\theta>-\alpha,$ and consider the mixed Poisson process $\big(N_{S^{-\alpha}_{\alpha,\theta}}(t);~t\ge 0\big)$ with waiting times $(T_{j}=G_{j}S^{\alpha}_{\alpha,\theta};~j\ge 1).$ The joint distribution of the waiting times can be expressed as
$$
(T_{j};~j\ge 1)\overset{d}=\big(\mathscr{M}^{(\theta)}_{\alpha,\theta+j\alpha};~j\ge 1\big),
$$
where the joint dependence on the right hand side is dictated by the representation of $(T_{j};~j\ge 1)$, and, for each $j,$ the marginal density of $T_{j}\overset{d}=\mathscr{M}^{(\theta)}_{\alpha,\theta+j\alpha}$ is given by
\begin{equation}
f_{T_{j}}(\lambda)=\frac{\lambda^{j-1}\mathbb{E}\big[S^{-(\theta+j\alpha)}_{\alpha}\big]}
{\Gamma(j)\mathbb{E}[S^{-\theta}_{\alpha}]}\mathrm{E}^{(\frac{\theta}{\alpha}+j+1)}_{\alpha,\theta+j\alpha+1}(-\lambda).
\label{PillaiDistwaiting}
\end{equation}
\begin{enumerate}
\item[(i)]For any $\theta>-\alpha,$ the mixing distribution~(conditional local time) in~Proposition~\ref{PropMittagdecomp} may be interpreted as
$$
\mathbb{P}\big(S^{-\alpha}_{\alpha,\theta}\in ds|N_{S^{-\alpha}_{\alpha,\theta}}(\lambda)=0\big)/ds=\frac{{\mbox e}^{-\lambda s}g_{\alpha,\theta}(s)}{\mathrm{E}^{(\frac{\theta}{\alpha}+1)}_{\alpha,\theta+1}(-\lambda)}:=
g^{(0)}_{\alpha,\theta}(s|\lambda).
$$
\item[(ii)]The conditional density of $S^{-\alpha}_{\alpha,\theta}|T_{j}=\lambda$, for $j=1,2,\ldots,$ is equivalent to
$$
\mathbb{P}\big(S^{-\alpha}_{\alpha,\theta}\in ds|\mathcal{N}_{\alpha,\theta}(\lambda,j)\big)/ds=\frac{{\mbox e}^{-\lambda s}g_{\alpha,\theta+j\alpha}(s)}{\mathrm{E}^{(\frac{\theta}{\alpha}+j+1)}_{\alpha,\theta+j\alpha+1}(-\lambda)}:=g^{(0)}_{\alpha,\theta+j\alpha}(s|\lambda).
$$
\item[(iii)]
$
\mathbb{P}\big(S^{-\alpha}_{\alpha,\theta}\in ds|\mathcal{N}_{\alpha,\theta}(\lambda,j)\big)=\mathbb{P}\big(S^{-\alpha}_{\alpha,\theta+j\alpha}\in ds|N_{S^{-\alpha}_{\alpha,\theta+j\alpha}}(\lambda)=0\big).
$
\item[(iv)]
$(P_{\ell})|\mathcal{N}_{\alpha,\theta}(\lambda,j)\sim\mathbb{L}^{(0)}_{\alpha,\theta+j\alpha}(\lambda),$ for $j=0,1,2,\ldots.$
\end{enumerate}
\end{prop}
\begin{proof}Beyond the descriptions from~\cite{PitmanPoissonMix}, the result in~(\ref{PillaiDistwaiting}) and statement (i) follow from Proposition~\ref{PropGenMittaglaw} with $\eta =\theta$ and $\theta$ otherwise replaced by $\theta+j\alpha.$ (\ref{AconPoisson}) indicates that the relevant density appearing in statement~(ii) should be expressed to be proportional to ${\mbox e}^{-\lambda s}s^{j}g_{\alpha,\theta}(s),$ which is indeed the case, since $s^{j}g_{\alpha,\theta}(s)\propto g_{\alpha,\theta+j\alpha}(s).$
\end{proof}

The next result, which is now straightforward to verify, describes the marginal distribution of the mixed Poisson process that illustrates a specific case of the \textit{Poisson switching identity} described in~\cite[Lemma 4.5]{PitmanPoissonMix}.
\begin{prop}For $\theta>-\alpha$, and $j=0,1,\ldots$,
$$
\mathbb{P}(N_{S^{-\alpha}_{\alpha,\theta}}(\lambda)=j)=\frac{\lambda^{j}\mathbb{E} \big[S^{-(\theta+j\alpha)}_{\alpha}\big]}
{j!\mathbb{E}[S^{-\theta}_{\alpha}]}\mathrm{E}^{(\frac{\theta}{\alpha}+j+1)}_{\alpha,\theta+j\alpha+1}(-\lambda),
$$
which is the same as $(\lambda/j)f_{T_{j}}(\lambda)$, for $j\neq 0.$ Furthermore, this implies the identity
$$
\mathbb{E}[S^{-\theta}_{\alpha}]=
\frac{\Gamma(\frac{\theta}{\alpha}+1)}
{\Gamma({\theta}+1)}=
\sum_{j=0}^{\infty}\frac{\lambda^{j}\Gamma(\frac{\theta}{\alpha}+j+1)}
{j!\Gamma({\theta}+j\alpha+1)}\mathrm{E}^{(\frac{\theta}{\alpha}+j+1)}_{\alpha,\theta+j\alpha+1}(-\lambda).
$$
\end{prop}
\begin{rem}The mixed Poisson process $(N_{S^{-\alpha}_{\alpha,\theta}}(t);~t\ge 0)$ has similarities to generalizations of the fractional Poisson process, see~\cite[Section~9.6]{GorenfloMittag} for references and further details, except in that case the waiting times are iid.  
\end{rem}

\subsection{The Brownian case $\alpha=\frac12,$ Hermite functions and Mills ratio}\label{sec:Hermitesection}
We now specialize the results of the previous sections to the Brownian case of $\alpha=\frac12.$ First, with respect to $(P_{\ell}(s))\sim\mathrm{PD}(\frac{1}{2}|\frac{1}{2}s^{-2}),$ we describe the special $\alpha=\frac12$ explicit case of the Gibbs partitions (EPPF) of $[n]$ in terms of Hermite functions as derived in~\cite{Pit03}, see also~\cite[Section~4.5]{Pit06}, as
\begin{equation}
p_{\frac12}(n_{1},\ldots,n_{k}|s)= s^{k-1}\tilde{h}_{k+1-2n}(s)\frac{\Gamma(n)}{2^{1-n}\Gamma(k)}
p_{\frac{1}{2}}(n_{1},\ldots,n_{k}),
\label{hermiteEPPF}
\end{equation}
where, for $U(a,b,c)$ a confluent hypergeometric function of the
second\break kind~(see~\cite[p.263]{Lebedev72}),
$$
\tilde{h}_{-2q}(s) = 2^{-q} U\left(q,
\frac{1}{2}, \frac{s^2}{2}\right)=\sum_{\ell=0}^{\infty}\frac{{(-s)}^{\ell}}{\ell!}
\frac{\Gamma(q+\frac{\ell}{2})}{2\Gamma(2q)}
2^{q+\frac{\ell}{2}}.
$$
is a Hermite function of index $-2q.$ That is to say 
\begin{equation}
\mathbb{G}^{(n,k)}_{\frac12}\bigg(\frac{1}{2}s^{-2}\bigg)=2^{n-k}\lambda^{k-1}\tilde{h}_{k+1-2n}(s).
\end{equation} 
The distribution of the number of blocks $\tilde{K}_{n}(s)$ can be read from~\cite[eq.~(115)]{Pit03} as
\begin{equation}
\mathbb{P}(\tilde{K}_{n}(s)=k)=\frac{(2n-k-1)!s^{k-1}\tilde{h}_{k+1-2n}(s)}{(n-k)!(k-1)!2^{n-k}}.
\label{localblocks}
\end{equation}
Let ${L}^{(0)}_{\frac12,\theta}(\lambda)$ denote a random variable with density, defined for $\theta>-\frac12,$
\begin{equation}
\tilde{g}^{(0)}_{\frac12,\theta}(x|\lambda):=\frac{x^{2\theta}{\mbox e}^{-\frac{1}{2}x^{2}}{\mbox e}^{-\lambda x}}{\Gamma(2\theta+1)\tilde{h}_{-(2\theta+1)}(\lambda)}=
\frac{{\mbox e}^{-\lambda x}f_{L_{1,\theta}}(x)}{\mathbb{E}\big[{|B_{1}|}^{2\theta+1}\big]\tilde{h}_{-(2\theta+1)}(\lambda)},
\label{halfMittag}
\end{equation}
which arises as the density of the conditional local time in the next few results. We first present the special case of Proposition~\ref{PropGenMittaglaw}.

\begin{prop}\label{PropHermite}For $\eta>-\frac12,$ as in Proposition~\ref{PropGenMittaglaw} with $\alpha=\frac12,$
let ${L}_{1,\eta}:=\big(2S_{\frac{1}{2},\eta}\big)^{-\frac12}\overset{d}=
\sqrt{2G_{\eta+\frac{1}{2}}}$ denote the local time at $0$ up till time $1$ of a process $B=(B_{t}:~t\in [0,1])$
whose ranked excursion lengths $(P_{\ell})\sim \mathrm{PD}(\frac{1}{2},\eta).$ Then, for $\theta>\eta,$  
set $\sqrt{2}\mathscr{M}^{(\eta)}_{\frac{1}{2},\theta}:=\frac{G_{2(\theta-\eta)}}{{L}_{1,\eta}}\overset{d}=\frac{G_{2(\theta-\eta)}}{\sqrt{2G_{\eta+\frac{1}{2}}}}.
$
\begin{enumerate}
\item[(i)]The density of $\sqrt{2}\mathscr{M}^{(\eta)}_{\frac{1}{2},\theta}$ can be expressed in terms of the Hermite\break function $\tilde{h}_{-(2\theta+1)}(\lambda)$ as
$$
\frac{2^{(\theta-\eta)}\Gamma(2\theta+1)}{\Gamma(\eta+\frac{1}{2})\Gamma(2(\theta-\eta))}
\lambda^{2(\theta-\eta)-1}\mathbb{E}\big[{|B_{1}|}^{2\theta+1}\big]\tilde{h}_{-(2\theta+1)}(\lambda).
$$
\item[(ii)]The density of $L_{1,\eta}\big|\sqrt{2}\mathscr{M}^{(\eta)}_{\frac{1}{2},\theta}=\lambda$ is $\tilde{g}^{(0)}_{\frac12,\theta}(x|\lambda)$, defined in~(\ref{halfMittag}).
\item[(iii)]$(P_{\ell})\big|L_{1,\eta}=s,\sqrt{2}\mathscr{M}^{(\eta)}_{\frac{1}{2},\theta}=\lambda$ is $\mathrm{PD}(\frac{1}{2}\big|\frac{1}{2}s^{-2}),$ with $\mathrm{EPPF}$ as in~(\ref{hermiteEPPF}).
\item[(iv)]The distribution of $(P_{\ell})\big|\sqrt{2}\mathscr{M}^{(\eta)}_{\frac{1}{2},\theta}=\lambda$ is 
$
\mathbb{L}^{(0)}_{\frac12,\theta}\left(\frac{\lambda}{\sqrt{2}}\right)
$
with $\mathrm{EPPF},$
\begin{equation}
\frac{\mathrm{E}^{(2\theta+k)}_{\frac{1}{2},\theta+n}(-\frac{\lambda}{\sqrt{2}})}{\mathbb{E}\big[{|B_{1}|}^{2\theta+1}\big]\tilde{h}_{-(2\theta+1)}(\lambda)}p_{\frac{1}{2},\theta}(n_{1},\ldots,n_{k}),
\label{BrownianMittagEPPF}
\end{equation}
as in~(\ref{MittagEPPF}).
\end{enumerate}
\end{prop}

We now specialize the mixed Poisson waiting time framework of Proposition~\ref{PropmixedL} to this setting with $A:=L_{1,\theta}.$

\begin{prop}\label{mixedPoissonBrown}Let $B=(B_{t};~t\in [0,1])$ with 
ranked excursion lengths $(V_{\ell})\sim \mathrm{PD}(\frac{1}{2},\theta),$ and corresponding local time ${L}_{1,\theta}:=\big(2S_{\frac{1}{2},\theta}\big)^{-\frac12},$ for $\theta>-\frac12$. Consider the mixed Poisson process $\big(N_{L_{1,\theta}}(t);t\ge 0\big)=\bigg(N_{S^{-1/2}_{\frac{1}{2},\theta}}\left(\frac{t}{\sqrt{2}}\right);t\ge 0\bigg),$ with waiting times 
$$
\left(\tilde{T}_{j}=\frac{G_{j}}{L_{1,\theta}};j\ge 1\right)\overset{d}=\bigg(\sqrt{2} \mathcal{M}^{(\theta)}_{\frac12,\theta+\frac{j}2};j\ge 1\bigg).
$$ 
Then, for $\lambda>0$ and  $j=0,1,2,\ldots,$ 
\begin{enumerate}
\item[(i)] $(V_{\ell})\left|\mathcal{N}_{\frac12,\theta}\left(\frac{\lambda}{\sqrt{2}},j\right)\right.\sim\mathbb{L}^{(0)}_{\frac12,\theta+\frac{j}2}\left(\frac{\lambda}{\sqrt{2}}\right).$
\item[(ii)] $\mathbb{P}\left(L_{1,\theta}\in dx\left|\mathcal{N}_{\frac12,\theta}\left(\frac{\lambda}{\sqrt{2}},j\right)\right.\right)=
\mathbb{P}\left(L_{1,\theta+\frac{j}{2}}\in dx\left|N_{L_{1,\theta+\frac{j}{2}}}(\lambda)=0\right.\right)$ has density equivalent to that of the random variable ${L}^{(0)}_{\frac12,\theta+\frac{j}2}(\lambda),$
\begin{equation}
\tilde{g}^{(0)}_{\frac12,\theta+\frac{j}{2}}(x|\lambda):=\frac{x^{2\theta+j}{\mbox e}^{-\frac{1}{2}x^{2}}{\mbox e}^{-\lambda x}}{\Gamma(2\theta+j+1)\tilde{h}_{-(2\theta+j+1)}(\lambda)}.
\label{halfMittag2}
\end{equation}
\end{enumerate}  
\end{prop}
\begin{rem}\label{ULdist}
The density~(\ref{halfMittag2}) of the local time variable ${L}^{(0)}_{\frac12,\theta+\frac{j}2}(\lambda)$ agrees, up to a scaling factor, with that of~\cite[example 9, eq. (3.1)]{Pek2016} arising from a P\'olya urn model where additional black balls are added at geometrically distributed waiting times. In particular, they have 
the distribution
$\mathrm{UL}(2\theta+j+1,(\frac{\lambda}{2\theta+j+1},\frac{1}{2\theta+j+1})),$ as defined in~\cite{Pek2016}. 
\end{rem}

\subsubsection{Properties of the variable ${L}^{(0)}_{\frac12,\theta}(\lambda)$ with density in~(\ref{halfMittag})}
We now focus on obtaining a more explicit description of the variable ${L}^{(0)}_{\frac12,\theta}(\lambda)$ with density~(\ref{halfMittag}), which is related to variables appearing in \cite{ChassaingJanson,Pek2016,Pit99local}. We shall show some parallels using the continuous waiting time framework in Proposition~\ref{mixedPoissonBrown}.

First, in addition to Remark~\ref{MittagHermite}, there are the identities,
\begin{equation}
\mathbb{E}\bigg[{(\tilde{P}_{1}(\lambda))}^{\theta+\frac{1}{2}}\bigg]=\mathbb{E}\big[{|B_{1}|}^{2\theta+1}\big]\tilde{h}_{-(2\theta+1)}(\lambda)=\mathbb{E}\bigg[{\mbox e}^{-\lambda\sqrt{2G_{\theta+\frac{1}{2}}}}\bigg].
\label{structuralmoments}
\end{equation}
Furthermore, from~\cite[p. 25]{Pit03} or~\cite[p. 93]{Pit06},
$$
\tilde{h}_{-1}(\lambda)=\frac{\mathbb{P}(B_{1}>\lambda)}{\phi(\lambda)} \quad{\mbox { and }}\quad\tilde{h}_{-2}(\lambda)=1-\lambda \tilde{h}_{-1}(\lambda)=\mathbb{P}(\tilde{K}_{2}(\lambda)=1),
$$
where $\tilde{h}_{-1}(\lambda)$ equates with \textit{Mill's ratio} and $\tilde{h}_{-2}(\lambda)$ equates with the probability that two independent uniform random variables fall into the excursion interval of $(B_{t}:~t\in [0,1])$ given $L_{1,0}=\lambda,$ when the sample size is $n=2.$ From this, one may express the cases of $\theta=0$ and $\theta=\frac12$ of the densities in~(\ref{halfMittag}) as 
$$
\frac{{\mbox e}^{-\frac{1}{2}x^{2}}{\mbox e}^{-\lambda x}\phi(\lambda)}{\mathbb{P}(B_{1}>\lambda)}\quad
{\mbox { and }}\quad
\frac{x{\mbox e}^{-\frac{1}{2}x^{2}}{\mbox e}^{-\lambda x}\phi(\lambda)}{\phi(\lambda)-\lambda\mathbb{P}(B_{1}>\lambda)},
$$
respectively. In addition, see~\cite[p.1758]{ChassaingJanson} for various interpretations, there is the simple formula,
\begin{equation}
\mathbb{P}\big((L_{1,\frac{1}{2}}-\lambda)\ge x\big|L_{1,\frac{1}{2}}\ge \lambda\big.\big)={\mbox e}^{-\frac{1}{2}x^{2}}{\mbox e}^{-\lambda x},
\label{LFR}
\end{equation}
corresponding to distributions with linear hazard rates. Pitman~\cite[eq. (18)]{Pit99local}
shows that~(\ref{LFR}) is the survival distribution of a random variable representable as 
\begin{equation}
L_{1,\frac{1}{2}}\sqrt{\frac{B^{2}_{1}}{B^{2}_{1}+\lambda^{2}}},
\label{pitid}
\end{equation}
where $L_{1,\frac{1}{2}}$ is independent of $B_{1}.$ See~\cite[eq. (17)]{Pit99local} for various interpretations of the random variable $\frac{B^{2}_{1}}{B^{2}_{1}+\lambda^{2}}\overset{d}=\tilde{P}_{1}(\lambda).$ Next, we provide a mixture representation of random variables having the density in~(\ref{halfMittag}).

\begin{prop}\label{PropL}
For $\theta>-\frac12,$ let $\tilde{P}^{(\theta+\frac{1}{2})}_{1}(\lambda)$ denote the random variable with density corresponding to the  $\theta+\frac12$ bias of the density~(\ref{structural}) of $\tilde{P}_{1}(\lambda).$ That is, 
\begin{equation}
f_{\tilde{P}^{(\theta+\frac{1}{2})}_{1}}(p|\lambda)=\frac{p^{\theta+\frac12}f_{\tilde{P}_{1}}(p|\lambda)}{\mathbb{E}[{|B_{1}|}^{2\theta+1}]\tilde{h}_{-(2\theta+1)}(\lambda)}.
\label{biasedstructuraldensity}
\end{equation}
\begin{enumerate}
\item[(i)] ${L}^{(0)}_{\frac12,\theta}(\lambda)\overset{d}={L}_{1,\theta}\sqrt{\tilde{P}^{(\theta+\frac{1}{2})}_{1}(\lambda)}$ with density $\tilde{g}^{(0)}_{\frac12,\theta}(x|\lambda)$ in~(\ref{halfMittag}).
\item[(ii)] For $\theta>\eta,$
\begin{equation}
{L}^{2}_{1,\eta}\overset{d}={L}^{2}_{1,\theta}\tilde{P}^{(\theta+\frac{1}{2})}_{1}\big(\sqrt{2}\mathscr{M}^{(\eta)}_{\frac{1}{2},\theta}\big)\sim\chi^{2}_{2\eta+1}.
\label{localtimeMittagid}
\end{equation}
\item[(iii)] $\tilde{P}^{(\theta+\frac{1}{2})}_{1}\big(\sqrt{2}\mathscr{M}^{(\eta)}_{\frac{1}{2},\theta}\big)\overset{d}=
\beta_{\eta+\frac{1}{2},\theta-\eta}$.
\item[(iv)]When $\theta=0,$
\begin{equation}
f_{\tilde{P}^{(\frac{1}{2})}_{1}}(p|\lambda)=\frac{\lambda\phi(\lambda)}{\mathbb{P}(B_{1}>\lambda)\sqrt{2\pi}}{(1-p)}^{-\frac{3}{2}}{\mbox e}^{-\frac{{\lambda}^{2}}{2}\left(\frac{p}{1-p}\right)}.
\label{structural2}
\end{equation}
\item[(v)]When $\theta=\frac12,$ $\tilde{P}^{(1)}_{1}(\lambda)$ has the size-biased distribution of $\tilde{P}_{1}(\lambda)$ with density
\begin{equation}
f_{\tilde{P}^{(1)}_{1}}(p|\lambda)=\frac{\lambda\phi(\lambda)}{[\phi(\lambda)-\lambda\mathbb{P}(B_{1}>\lambda)]\sqrt{2\pi}} p^{\frac12}{(1-p)}^{-\frac{3}{2}}{\mbox e}^{-\frac{{\lambda}^{2}}{2}\left(\frac{p}{1-p}\right)}.
\label{structural3}
\end{equation}
\end{enumerate}
\end{prop}

\begin{proof}The form of the density in~(\ref{biasedstructuraldensity}) is determined by~(\ref{structural}) and~(\ref{structuralmoments}). Statement (i) follows by the usual argument to derive the density of a randomly scaled gamma variable, which in this case involves cancellations from the choice of the density~(\ref{biasedstructuraldensity}). The form of the density is concluded by noting that, (compare with (\ref{LFR})),
$$
\mathbb{E}\bigg[{\mbox e}^{-\frac{1}{2}\frac{x^{2}}{\tilde{P}_{1}(\lambda)}}\bigg]=\mathbb{E}\bigg[{\mbox e}^{-\frac{x^{2}}{2}\big(2\lambda^{2}S_{{1}/{2}}+1\big)}\bigg]={\mbox e}^{-\frac{1}{2}x^{2}}{\mbox e}^{-\lambda x}.
$$
Statement (ii) follows from Proposition~\ref{PropHermite}. Statement (iii) is obtained by standard beta-gamma algebra. Statements (iv) and (v) follow from the preceding discussion. 
\end{proof}
In view of~(\ref{structural2}) and~(\ref{structural3}), we establish a simple mixture relationship to the random variable defined by~(\ref{LFR}) and (\ref{pitid}) appearing in~\cite{ChassaingJanson,Pit99local}.

\begin{prop}Consider the random variable defined by~(\ref{LFR}) and (\ref{pitid}), and let 
$\tilde{K}_{2}(\lambda)$ denote the random variable with distribution in~(\ref{localblocks}) for $n=2,$ $s=\lambda,$ and hence $k=1,2.$ Then, setting $\theta=1-\tilde{K}_{2}(\lambda)/2$ in Proposition~\ref{PropL} gives
$$
L_{1,\frac{1}{2}}\sqrt{\frac{B^{2}_{1}}{B^{2}_{1}+\lambda^{2}}}\overset{d}={L}^{(0)}_{\frac{1}{2},\frac{2-\tilde{K}_{2}(\lambda)}{2}}(\lambda)\overset{d}={L}_{1,\frac{2-\tilde{K}_{2}(\lambda)}{2}}\sqrt{\tilde{P}^{\big(\frac{3-\tilde{K}_{2}(\lambda)}{2}\big)}_{1}(\lambda)}.
$$
\end{prop}
\begin{proof}From (\ref{LFR}), it follows that the density of (\ref{pitid}) can be expressed as 
$$
(x+\lambda){\mbox e}^{-\frac{1}{2}x^{2}}{\mbox e}^{-\lambda x}=\mathbb{P}(\tilde{K}_{2}(\lambda)=1)\tilde{g}^{(0)}_{\frac12,\frac{1}{2}}(x|\lambda)+\mathbb{P}(\tilde{K}_{2}(\lambda)=2)\tilde{g}^{(0)}_{\frac12,0}(x|\lambda).
$$
\end{proof}

The results above allows us to express the $\frac{1}{2}$-diversity in terms of a specific random variable, which we now describe relative to the mixed Poisson waiting time setting.
Recall that the distribution of the number of blocks $K_{n}$ of a $\mathrm{PD}(\frac12,\theta)$ partition of $[n]$ has the explicit form
\begin{equation}
\mathbb{P}^{(n)}_{\frac{1}{2},\theta}(k)= \frac{\Gamma(n)}{\Gamma(k)} \frac{(2\theta)_k}{(\theta)_n} \binom{2n-k-1}{n-1} 2^{k-2n}.
\label{KnBrowntheta}
\end{equation}

\begin{cor}\label{CorollaryKnBrown} Let $K_{n}$ denote the number of blocks in a $\mathrm{PD}(\frac12,\theta)$ partition of $[n],$ having distribution
$\mathbb{P}^{(n)}_{\frac{1}{2},\theta}(k)$ as in~(\ref{KnBrowntheta}). Consider the mixed Poisson/waiting time setting in Proposition~\ref{mixedPoissonBrown}, and, for each $j=0,1,2,\ldots,$ let $K_{n}(\frac{\lambda}{\sqrt{2}},j)$ denote the random variable that equates with the distribution of $K_{n}|\mathcal{N}_{\frac12,\theta}\big(\frac{\lambda}{\sqrt{2}},j\big).$ 
\begin{enumerate}
\item[(i)]From~(\ref{BrownianMittagEPPF}),
$$
\mathbb{P}\left(K_{n}\bigg(\frac{\lambda}{\sqrt{2}},j\bigg)=k\right)=
\frac{\mathrm{E}^{(2\theta+j+k)}_{\frac{1}{2},\theta+\frac{j}2+n}\big(-\frac{\lambda}{\sqrt{2}}\big)}{\mathbb{E}\big[{|B_{1}|}^{2\theta+j+1}\big]\tilde{h}_{-(2\theta+j+1)}(\lambda)}\mathbb{P}^{(n)}_{\frac{1}{2},\theta+\frac{j}2}(k).
$$
\item[(ii)]As $n\rightarrow\infty,$ ${(2n)}^{-\frac12}K_{n}\big(\frac{\lambda}{\sqrt{2}},j\big)\overset{a.s.}\rightarrow {L}^{(0)}_{\frac12,\theta+\frac{j}2}(\lambda)\overset{d}={L}_{1,\theta+\frac{j}2}\sqrt{\tilde{P}^{\big(\theta+\frac{j+1}{2}\big)}_{1}(\lambda)},$
with density $\tilde{g}^{(0)}_{\frac12,\theta+\frac{j}2}(x|\lambda).$
\end{enumerate}
\end{cor}

\begin{rem}One may further compare the conditional limiting results for discrete geometric waiting times as described in~\cite[p.6]{Pek2016} with statement~(ii) of Corollary~\ref{CorollaryKnBrown}. 
\end{rem}

\subsubsection{$\mathrm{MLMC}(\frac12,\eta)$ conditioned on $\sqrt{2}\mathscr{M}^{(\eta)}_{\frac{1}{2},\theta}=\lambda$}
Proposition~\ref{mixedPoissonBrown} and Corollary~\ref{CorollaryKnBrown} seem to suggest strong relationships to the discrete waiting time framework in \cite{Pek2016}, which is worthy of future investigations and also lends support to the notion of looking at more exotic choices of $\gamma$ in constructive models. Here,
relative to an $\mathrm{MLMC}(\frac12,\eta)$ conditioned on $\sqrt{2}\mathscr{M}^{(\eta)}_{\frac{1}{2},\theta}=\lambda,$ setting $\alpha=\frac12,$ in~(\ref{gammarMittag}), applying a change of variable, Proposition~\ref{GenBrownMLMC} and the results in the previous section lead to a sequence of local times $({L}^{(r)}_{\frac12,\theta}(\lambda),r\ge 0)$
satisfying (in joint distribution)
\begin{equation}
\bigg({L}^{(r)}_{\frac{1}{2},\theta}(\lambda);~r\ge 0\bigg)\overset{d}=\left(\sqrt{L^{2}_{1,\theta}\tilde{P}^{(\theta+\frac{1}{2})}_{1}(\lambda)+2\sum_{i=1}^{r}\mathbf{e}_{i}};~r\ge 0\right),
\label{jointMLMCBrownianlocal}
\end{equation}
with respective marginal densities, for $r=0,1,2,\ldots,$ 
\begin{equation}
\tilde{g}^{(r)}_{\frac12,\theta}(s|\lambda)=\frac{\mathbb{E}\bigg[{\mbox e}^{-\lambda s\sqrt{\beta_{\theta+\frac{1}{2},r}}}\bigg] f_{L_{1,\theta+r}}(s)}
{\mathbb{E}\big[{|B_{1}|}^{2\theta+1}\big]\tilde{h}_{-(2\theta+1)}(\lambda)},
\label{HermiteTreelocaltimedensity}
\end{equation}
and 
$$
\mathbb{E}\bigg[{\mbox e}^{-\lambda s\sqrt{\beta_{\theta+\frac{1}{2},r}}}\bigg]=\sum_{\ell=0}^{\infty}\frac{{(-\lambda s)}^{\ell}}{\ell!}\frac{\Gamma(\theta+\frac{1}{2}+r)\Gamma(\theta+\frac{\ell+1}{2})}{\Gamma(\theta+\frac{1}{2})\Gamma(\theta+\frac{\ell+1}{2}+r)}.
$$
One has $\mathbb{L}^{(r)}_{\frac12,\theta}(\frac{\lambda}{\sqrt{2}})=\int_{0}^{\infty}\mathrm{PD}\big(\frac{1}{2}|\frac{1}{2}s^{-2}\big)\tilde{g}^{(r)}_{\frac12,\theta}(s|\lambda)ds.$

\begin{prop}\label{prophermitetree}Consider $\big((P_{k,r}),\sqrt{2}L_{1,\eta+r}; r\ge 0\big)\sim \mathrm{MLMC}(\frac12,\eta).$ Then, the conditional distribution of $\big((P_{k,r}),\sqrt{2}L_{1,\eta+r}; r\ge 0\big)\left|\sqrt{2}\mathscr{M}^{(\eta)}_{\frac{1}{2},\theta}=\lambda\right.$ is such that $(P_{k,r})\left|\sqrt{2}\mathscr{M}^{(\eta)}_{\frac{1}{2},\theta}=\lambda\right.\sim \mathbb{L}^{(r)}_{\frac12,\theta}\big(\frac{\lambda}{\sqrt{2}}\big)$, and the nested family is jointly equivalent in distribution to $((P_{\ell,r}(\lambda));r\ge 0)$ defined by the recursive fragmentation, for $r=1,2,\ldots$,
$$
(P_{\ell,r}(\lambda))=\widehat{\mathrm{Frag}}^{(r)}_{\frac{1}{2},\frac{1}{2}}\circ\cdots\circ \widehat{\mathrm{Frag}}^{(1)}_{\frac{1}{2},\frac{1}{2}}((P_{\ell,0}(\lambda))),
$$
where $(P_{\ell,0}(\lambda))\sim \mathbb{L}^{(0)}_{\frac12,\theta}\big(\frac{\lambda}{\sqrt{2}}\big).$

\begin{enumerate}
\item[(i)]The joint conditional distribution of $\left(L_{1,\eta+r};r\ge 0\right)\left|\sqrt{2}\mathscr{M}^{(\eta)}_{\frac{1}{2},\theta}=\lambda\right.$ is equivalent in joint distribution to the collection $\big({L}^{(r)}_{\frac{1}{2},\theta}(\lambda);~r\ge 0\big)$ in (\ref{jointMLMCBrownianlocal}) with marginal densities~(\ref{HermiteTreelocaltimedensity}).
\item[(ii)]For each fixed $r,$ the $\mathrm{EPPF}$ of the $\mathbb{L}^{(r)}_{\frac12,\theta}\big(\frac{\lambda}{\sqrt{2}}\big)$ partition of $[n]$ can be expressed as
$$
\frac{\mathbb{E}\left[\mathrm{E}^{(2\theta+2r+k)}_{\frac{1}{2},\theta+r+n}\bigg(-\frac{\lambda\sqrt{\beta_{\theta+\frac{1}{2},r}}}{\sqrt{2}}\bigg)\right]}{\mathbb{E}\big[{|B_{1}|}^{2\theta+1}\big]\tilde{h}_{-(2\theta+1)}(\lambda)}p_{\frac{1}{2},\theta+r}(n_{1},\ldots,n_{k}),
$$
where the expectation at the numerator can be expressed as
$$
\sum_{\ell=0}^{\infty}\frac{{(-\lambda)}^{\ell}}{\ell!}\frac{2^{-\frac{\ell}{2}}\Gamma(
\theta+r+n)}{\Gamma(
\frac{\ell}{2}+\theta+r+n)} \frac{(\theta+\frac{1}{2})_r (2(\theta+r)+k)_\ell}{(\theta+\frac{\ell+1}{2})_r}.
$$
\end{enumerate}
\end{prop}
The next result combines the waiting time framework with the MLMC models.

\begin{cor}\label{BCRTwaiting}Suppose that for $\theta>-\frac12,$ $((P_{k,r}),\sqrt{2}L_{1,\theta+r}, r\ge 0)\sim \mathrm{MLMC}(\frac12,\theta).$ Consider the waiting time framework in Proposition~\ref{mixedPoissonBrown} and Corollary~\ref{CorollaryKnBrown}. Then, $((P_{k,r}),\sqrt{2}L_{1,\theta+r}; r\ge 0)|\mathcal{N}_{\frac12,\theta}\big(\frac{\lambda}{\sqrt{2}},j\big)$ has the distributional properties in Proposition~\ref{prophermitetree}, with $\theta$ replaced by $\theta+\frac{j}2$ for $j=0,1,2,\ldots.$
\end{cor}

\begin{rem}As is known, the $\mathrm{MLMC}(\frac12,\frac12)$ distribution corresponds to the components of the line-breaking construction of the BCRT as described in Aldous\cite{AldousCRTI, AldousCRTIII}. Corollary~\ref{BCRTwaiting} shows that a BCRT conditioned appropriately on the events $\mathcal{N}_{\frac12,\frac12}\big(\frac{\lambda}{\sqrt{2}},j\big)$ is equivalent in distribution to the line-breaking construction based on the collection $\bigg(\sqrt{2}{L}^{(r)}_{\frac{1}{2},\frac{j+1}{2}}(\lambda);~r\ge 0\bigg)$ where ${L}^{(0)}_{\frac{1}{2},\frac{j+1}{2}}(\lambda)$ has a $\mathrm{UL}(2+j,(\frac{\lambda}{2+j},\frac{1}{2+j}))$ distribution, described in~\cite{Pek2016}, as noted in Remark~\ref{ULdist}. Less obvious is that Corollary~\ref{BCRTwaiting}, coupled with the results in \cite{Pek2016,Pek2017jointpref}, shows that these limits may be achieved by a preferential attachment scheme with Bernoulli immigration and appropriate conditioning. 
\end{rem}
\begin{rem}It follows from Proposition~\ref{Propprojection} that, for any $\theta'$ satisfying $\theta>\theta'\ge \eta,$  
$$
\bigg({L}^{(r)}_{\frac{1}{2},\theta}\big(\sqrt{2}\mathscr{M}^{(\theta')}_{\frac{1}{2},\theta}\big);r\ge 0\bigg)\overset{d}=\left(\sqrt{L^{2}_{1,\theta'}+2\sum_{\ell=1}^{r}\mathbf{e}_{\ell}};r\ge 0\right).
$$
See \cite{Pek2017jointpref} for some possible interpretations.
\end{rem}
\begin{rem}\cite{PitmanMax} describes various properties and constructions of the class of generalized Brownian bridges $B=(B_{t}:t\in [0,1])$ with local times $L_{1,\theta}$ for $\theta>-\frac12$, where $L^{2}_{1,\theta}\sim \chi^{2}_{2\theta+1}$ has a chi-squared distribution with $2\theta+1$ degrees of freedom. 
\end{rem}

\section{Gibbs partitions derived from the $\mathrm{Frag}_{\alpha,-\alpha\delta}$ fragmentation operator}
\label{sec:PitFrag}
Let $(P_{k,0})\in\mathcal{P}_{\infty}$ denote a general mass partition. Independent of this, consider a countable collection of iid mass partitions $((Q^{(\ell)}_{k});~\ell\ge 1)$ with common distribution $\mathrm{PD}(
\alpha,-\alpha\delta)$, where $\alpha,\delta \in (0,1).$ Then, a $\mathrm{PD}(\alpha,-\alpha\delta)$ fragmentation operator, $\mathrm{Frag}_{\alpha,-\alpha\delta},$ defined in \cite{BerPit2000,Pit99Coag}, creates a mass partition $(P_{k,1})$ by fragmenting $(P_{k,0})$ as follows.
\begin{equation}
(P_{k,1}):=\mathrm{Frag}_{\alpha,-\alpha\delta}((P_{k,0}))=
\mathrm{Rank}\big(P_{\ell,0}(Q^{(\ell)}_{j});~\ell\ge 1\big).
\label{fragoperator}
\end{equation}
This in fact corresponds to a generic notion of an infinite fragmentation where one can replace $\mathrm{PD}(
\alpha,-\alpha\delta),$ as descrbed in~\cite{BerFrag,Pit06}. However,~\cite{Pit99Coag} shows that when $(P_{k,0})\sim \mathrm{PD}(\alpha\delta,\theta)$ for $\theta>-\alpha\delta$, it follows that $(P_{k,1})\sim \mathrm{PD}(\alpha,\theta),$ which generalizes many important special cases appearing in the literature. 
Specialized to the setting where $(P_{k,1})\sim \mathrm{PD}(\alpha,\theta)$, \cite{Pit99Coag} shows there is a dual coagulation operation, $\mathrm{Coag}_{\delta,\frac{\theta}{\alpha}},$ such that 
$$
(P_{k,0}):=\mathrm{Coag}_{\delta,\frac{\theta}{\alpha}}((P_{k,1})),
$$
where the coagulator is based on a mass partition $(W_{k,1})\sim \mathrm{PD}(\delta,\frac{\theta}{\alpha})$ independent of $(P_{k,1}).$ This operation may be understood in terms of the following equivalent property of composition of independent distribution functions $P_{\alpha\delta,\theta}=P_{\alpha,\theta}\circ P_{\delta,\frac{\theta}{\alpha}},$ details are omitted for brevity. 
It follows, for general $\theta>-\alpha\delta,$ that $((P_{k,0}),(P_{k,1}),(W_{k,1}))$ have respective local times $(S^{-\alpha\delta}_{\alpha\delta,\theta},S^{-\alpha}_{\alpha,\theta},S^{-\delta}_{\delta,\frac{\theta}{\alpha}})$ satisfying the exact relationship,
\begin{equation}
S^{-\alpha\delta}_{\alpha\delta,\theta}=S^{-\alpha\delta}_{\alpha,\theta}\times S^{-\delta}_{\delta,\frac{\theta}{\alpha}},
\label{alphadeltalocal}
\end{equation}
where $S_{\alpha,\theta}$ and $S_{\delta,\frac{\theta}{\alpha}}$ are independent.
The next lemma, which is simple to prove, shows that the conditional densities of $S_{\alpha,\theta,}|S_{\alpha\delta,\theta}=y$ and 
$S_{\delta,\frac{\theta}{\alpha}}|S_{\alpha\delta,\theta}=y$ do not depend on $\theta.$
\begin{lem}For all $\theta>-\alpha\delta,$ consider the random variables $(S_{\alpha\delta,\theta},S_{\alpha,\theta},S_{\delta,\frac{\theta}{\alpha}})$ satisfying~(\ref{alphadeltalocal}). 
\begin{enumerate}
\item[(i)]The conditional density of $S_{\alpha,\theta}|S_{\alpha\delta,\theta}=y$ is the same as the case of $\theta=0,$ and is given by
\begin{equation}
f_{\alpha|\alpha\delta}(t|y)
=\left[\frac{\alpha f_{\delta}({(\frac{y}{t})}^{\alpha})t^{-\alpha}}
{y^{1-\alpha}f_{\alpha\delta}(y)}\right]f_{\alpha}(t).
\label{densityalphagivenalphadelta}
\end{equation}
\item[(ii)]The conditional density of $S_{\delta,\frac{\theta}{\alpha}}|S_{\alpha\delta,\theta}=y$ is the same as the case of $\theta=0,$ and is given by
\begin{equation}
f^{\dagger}_{\delta|\alpha\delta}(\omega|y)
=\frac{f_{\alpha}(y\omega^{-\frac1\alpha})}
{\omega^{\frac1\alpha}f_{\alpha\delta}(y)}f_{\delta}(\omega).
\end{equation}
\item[(iii)]When $S_{\alpha\delta}$ has density $\gamma(dy)/dy=h(y)f_{\alpha\delta}(y),$ for 
$\mathbb{E}[h(S_{\alpha\delta})]=1$, the joint density of $(S_{\alpha\delta},S_{\alpha})$ is given by
\begin{equation}
\frac{\alpha f_{\delta}({\big(\frac{y}{t}\big)}^{\alpha})t^{-\alpha}f_{\alpha}(t)\gamma(dy)}
{y^{1-\alpha}f_{\alpha\delta}(y)dy}=\alpha y^{\alpha-1}h(y) f_{\delta}\bigg({\bigg(\frac{y}{t}\bigg)}^{\alpha}\bigg)t^{-\alpha}f_{\alpha}(t),
\label{densityalphagivenalphadelta2}
\end{equation}
and $S_{\alpha}$ has density
\begin{equation}
\gamma_{\alpha|\delta}(dt)/dt=\left[\int_{0}^{\infty}h(r^{\frac1\alpha}t)f_{\delta}(r)dr\right]f_{\alpha}(t)=\mathbb{E}\big[h(tS^{\frac1\alpha}_{\delta})\big] f_{\alpha}(t).
\label{margalphadelta}
\end{equation} 
\item[(iv)]The joint density of $(S_{\alpha},S_{\delta})$ can be expressed as
\begin{equation}
\kappa^{[\gamma]}_{\alpha,\delta}(t,\omega)=f_{\alpha}(t)f_{\delta}(\omega)h(t\omega^{\frac1\alpha}).
\end{equation}
\end{enumerate}
\end{lem}

As discussed in~\cite{BerFrag,BerLegall00,BerPit2000,Pit99Coag,Pit06},~when $\theta=0,$ and $\alpha$ and $\delta$ vary as time parameters, the dynamic of the $\mathrm{Coag}_{\delta,0}$ operations corresponds in distribution to that of the Bolthausen-Sznitman coalescent. Here given the structural relationship between $(P_{k,0})~\sim \mathrm{PD}(\alpha\delta,0)$ and $(P_{k,1})~\sim \mathrm{PD}(\alpha,0)$ described by the fragmentation operation in~(\ref{fragoperator}), we shall proceed to present a rather interesting description of the law of $(P_{k,1})|S_{\alpha\delta}=y,$ where $S_{\alpha\delta}$ is the inverse local time at zero up till time $1,$ satisfying $(P_{k,0})|S_{\alpha\delta}=y\sim \mathrm{PD}(\alpha\delta|y)$.

\begin{prop}\label{PropPDarcondlaw}
Suppose that $(P_{k,0})\sim \mathrm{PD}(\alpha\delta,0).$ Then, for a version of $(P_{k,0}),$ the following facts can be read from~\cite{Pit99Coag}.
\begin{enumerate}
\item[(i)] $(P_{k,1}):=\mathrm{Frag}_{\alpha,-\alpha\delta}((P_{k,0}))\sim \mathrm{PD}(\alpha,0).$  
\item[(ii)] The relation between their respective local times $(S^{-\alpha\delta}_{\alpha\delta},S^{-\alpha}_{\alpha})$ is given by $S^{-\alpha\delta}_{\alpha\delta}=(S^{-\alpha}_{\alpha})^\delta \times S^{-\delta}_{\delta}$,
where $S^{-\delta}_{\delta}$ is the local time of $(W_{k,1})\sim\mathrm{PD}(\delta,0),$ a mass partition independent of $(P_{k,1}),$ defining the coagulation operator that gives, 
$(P_{k,0}):=\mathrm{Coag}_{\delta,0}((P_{k,1}))$. 
\end{enumerate}
These facts indicate that the distribution of $(P_{k,1})|S_{\alpha\delta}=y$ is specified by the PK 
distribution, $\mathrm{PD}_{\alpha|\delta}(\alpha|S_{\alpha\delta}=y):=\mathrm{PD}_{\alpha|\delta}(\alpha|y),$ defined as
\begin{equation}
\mathrm{PD}_{\alpha|\delta}(\alpha|y):=\int_{0}^{\infty}\mathrm{PD}(\alpha|t)f_{\alpha|\alpha\delta}(t|y)\,dt,
\label{PDarcond}
\end{equation}
where $f_{\alpha|\alpha\delta}(t|y)$ is the density of $S_{\alpha}|S_{\alpha\delta}=y$, described in~(\ref{densityalphagivenalphadelta}). 
\end{prop}
\begin{proof}Set $Y=S_{\alpha}\times S_{\delta}^{\frac1\alpha}$. It follows that $(P_{k,1})|S_{\alpha}=t,Y=y$ is the same as $(P_{k,1})|S_{\alpha}=t,$ which is $\mathrm{PD}(\alpha|t).$ The derivation of (\ref{PDarcond}) is a straightforward application of Bayes rule.
\end{proof}

\begin{rem}
Refer to~\cite[Corollary 10]{Haas2},
in the case where $\alpha\in \big(\frac{1}{2},1\big)$ and $\delta=\frac{1-\alpha}{\alpha},$ $(P_{\ell,0})\sim \mathrm{PD}(1-\alpha,1-\alpha),$ with inverse local time $S_{1-\alpha,1-\alpha},$ corresponds (in distribution) to the coarse spinal partition of a stable tree of dimension $1<\frac{1}{\alpha}<2,$ $(P_{\ell,1})=\mathrm{Frag}_{\alpha,\alpha-1}((P_{\ell,0}))\sim \mathrm{PD}(\alpha,1-\alpha)$ corresponds to the fine spinal partition. The coagulator $(W_{\ell,1})\sim \mathrm{PD}\big(\frac{1-\alpha}{\alpha},\frac{1-\alpha}{\alpha}\big).$ Hence, $(P_{\ell,1})|S_{1-\alpha,1-\alpha}=y$ is $\mathrm{PD}_{\alpha|\frac{1-\alpha}{\alpha}}(\alpha|y).$
\end{rem}

Proposition~\ref{PropPDarcondlaw} indicates laws for the general case where $(P_{k,0})\sim \mathrm{PK}_{\alpha\delta}(\gamma).$ Thus it provides a description of all laws generated by a stable $(\alpha\delta)$ subordinator via the $\mathrm{PD}(\alpha,-\alpha\delta)$ fragmentation operation. 
\begin{cor}\label{corPDarcondlaw}
In the setting of Proposition~\ref{PropPDarcondlaw}, if
$(P_{k,0})\sim \mathrm{PK}_{\alpha\delta}(\gamma),$ then the joint density of $(S_{\alpha\delta},S_{\alpha})$ is given by~(\ref{densityalphagivenalphadelta2}), and hence the marginal distribution of $(P_{k,1}):=\mathrm{Frag}_{\alpha,-\alpha\delta}((P_{k,0}))$ is 
$$
\mathrm{PK}_{\alpha}(\gamma_{\alpha|\delta}):=\int_{0}^{\infty}\mathrm{PD}(\alpha|t)\gamma_{\alpha|\delta}(dt)=\int_{0}^{\infty}\mathrm{PD}_{\alpha|\delta}(\alpha|y)\gamma(dy),
$$
where 
$\gamma_{\alpha|\delta}(dt)/dt=\mathbb{E}\big[h\big(tS^{\frac1\alpha}_{\delta}\big)\big] f_{\alpha}(t),$ as in
(\ref{margalphadelta}).
\end{cor}

\begin{rem}
Setting $h(y)=y^{-\theta}/\mathbb{E}[S^{-\theta}_{\alpha\delta}]$ yields the case of $(P_{k,0})\sim\mathrm{PD}(\alpha\delta,\theta)$ and $(P_{k,1})\sim\mathrm{PD}(\alpha,\theta),$ which follows from $\mathbb{E}[S^{-\theta}_{\alpha\delta}]= \mathbb{E}[S^{-\theta}_{\alpha}] \mathbb{E}\big[S^{-\frac\theta\alpha}_{\delta}\big].$
\end{rem}

\subsection{The $\mathrm{PD}_{\alpha|\delta}(\alpha|y)$ Gibbs partition of $[n]$}\label{sec:GibbsEPPFfraglocal}
Recall that when $(P_{\ell,0}) \sim \mathrm{PD}(\alpha\delta,0)$, $(P_{\ell,0}) |S_{\alpha\delta} = y$ has the associated Gibbs partition of $[n]$ described by the conditional EPPF,
\begin{equation}
p_{\alpha\delta}(n_{1},\ldots,n_{k}|y):=\frac{f^{(n-k\alpha\delta)}_{\alpha\delta,k\alpha\delta}(y)}{f_{\alpha\delta}(y)} p_{\alpha\delta}(n_{1},\ldots,n_{k}),
\label{GibbsalphadeltaEPPF}
\end{equation}
where
$$
\frac{f^{(n-k\alpha\delta)}_{\alpha\delta,k\alpha\delta}(y)}{f_{\alpha\delta}(y)} = \mathbb{G}^{(n,k)}_{\alpha\delta}(y)\frac{{(\alpha\delta)}^{1-k}\Gamma(n)}{\Gamma(k)}.
$$
The next theorem 
provides the rather remarkable description of the EPPF of $\mathrm{PD}_{\alpha|\delta}(\alpha|y)$ under the same setting as specified in Proposition~\ref{PropPDarcondlaw}.

\begin{thm}\label{TheoremGibbsparitionFrag}
Consider $(P_{\ell,0})\sim \mathrm{PD}(\alpha\delta,0)$ and $(P_{\ell,1}):=\mathrm{Frag}_{\alpha,-\alpha\delta}((P_{\ell,0}))\sim \mathrm{PD}(\alpha,0),$ as specified in Proposition~\ref{PropPDarcondlaw}. Then, $(P_{\ell,1})|S_{\alpha\delta}=y$ has law $\mathrm{PD}_{\alpha|\delta}(\alpha|y)$ as defined in~(\ref{PDarcond}). The EPPF of the $\mathrm{PD}_{\alpha|\delta}(\alpha|y)$ Gibbs partition of~$[n]$ 
can be expressed as
\begin{equation}
p_{\alpha|\delta}(n_{1},\ldots,n_{k}|y):=\left[\sum_{j=1}^{k}\mathbb{P}^{(k)}_{\delta,0}(j)\frac{f^{(n-j\alpha\delta)}_{\alpha\delta,j\alpha\delta}(y)}{f_{\alpha\delta}(y)} \right]p_{\alpha}(n_{1},\ldots,n_{k}),
\label{Gibbsalpha|delta}
\end{equation}
where $\sum_{j=1}^{k}\mathbb{P}^{(k)}_{\delta,0}(j){f^{(n-j\alpha\delta)}_{\alpha\delta,j\alpha\delta}(y)}$ is, for $K^{(1)}_{n}$ the number of blocks in a $\mathrm{PD}(\alpha,0)$ partition of $[n],$ the conditional density of $S_{\alpha\delta}|K^{(1)}_{n}=k.$  $\mathbb{P}^{(k)}_{\delta,0}(j)=\mathbb{P}_{\delta,0}(K_{k}=j)$ is the distribution of the number of blocks in a $\mathrm{PD}(\delta,0)$ partition of $[k].$
\end{thm}
\begin{proof}The $\mathrm{EPPF}$ is the conditional distribution of a $\mathrm{PD}(\alpha,0)$ partition of $[n]$ given $S_{\alpha\delta}=y.$ (\ref{Gibbsalpha|delta}) is such a representation in terms of the marginal EPPF $p_{\alpha}(n_{1},\ldots,n_{k})$ and the conditional density of $S_{\alpha\delta}|K^{(1)}_{n}=k.$
It remains to show that $S_{\alpha\delta}|K^{(1)}_{n}=k$ agrees with the expression in (\ref{Gibbsalpha|delta}) as indicated, which is rather challenging. The formal proof of this result is given in Theorem~\ref{PropCondSKn}
of the appendix.
\end{proof}
We now describe the distribution of the number of blocks and its limiting behavior.

\begin{cor}Suppose that, for fixed $y>0,$ a mass partition $(\hat{P}_{\ell}(y))$ has law $\mathrm{PD}_{\alpha|\delta}(\alpha|y)$ as defined in~(\ref{PDarcond}). Its EPPF is given by~(\ref{Gibbsalpha|delta}) based on a partition of $[n]$, with $\hat{K}_{n}(y)$ denoting the random number of unique blocks. Then, for $k=1,\ldots,n,$
$$
\mathbb{P}(\hat{K}_{n}(y)=k)=\left[\sum_{j=1}^{k}\mathbb{P}^{(k)}_{\delta,0}(j)\frac{f^{(n-j\alpha\delta)}_{\alpha\delta,j\alpha\delta}(y)}{f_{\alpha\delta}(y)} \right]\mathbb{P}^{(n)}_{\alpha,0}(k),
$$
and, as $n\rightarrow \infty,$ $n^{-\alpha}\hat{K}_{n}(y)\overset{a.s.}\rightarrow \hat{Z}(y),$ where $\hat{Z}(y)$ is equivalent in distribution to\break $S^{-\alpha}_{\alpha}|S_{\alpha\delta}=y,$ specified by~(\ref{densityalphagivenalphadelta}).
\end{cor}

Now, suppose that $(P_{\ell,0})\sim \mathrm{PK}_{\alpha\delta}(\gamma),$ where $\gamma(dy)/dy=h(y)f_{\alpha\delta}(y)$ with\break $\mathbb{E}[h(S_{\alpha\delta})] = 1$. The associated Gibbs partition of $[n]$ is described by the $\mathrm{EPPF}$ 
expressed as
\begin{equation}
p^{[\gamma]}_{\alpha\delta}(n_{1},\ldots,n_{k})=\tilde{V}_{n,k} \times p_{\alpha\delta}(n_{1},\ldots,n_{k}),
\label{alphadeltaEPPF}
\end{equation}
where  $\tilde{V}_{n,k}=\mathbb{E}_{\alpha\delta}[h(S_{\alpha\delta})|K_{n}=k]
=V_{n,k}\frac{{(\alpha\delta)}^{1-k}\Gamma(n)}{\Gamma(k)}$ and, for clarity, $K_{n}$ is the number of blocks of a $\mathrm{PD}(\alpha\delta,0)$ partition of $[n].$ 

\begin{prop}\label{PropVnfrag}
Suppose that $(P_{\ell,0})\sim \mathrm{PK}_{\alpha\delta}(\gamma)$ with $\mathrm{EPPF}$ in~(\ref{alphadeltaEPPF}). Corollary~\ref{corPDarcondlaw} gives
$
(P_{\ell,1}):=\mathrm{Frag}_{\alpha,-\alpha\delta}((P_{\ell,0}))\sim\mathrm{PK}_{\alpha}(\gamma_{\alpha|\delta}).
$
The $\mathrm{PK}_{\alpha}(\gamma_{\alpha|\delta})$ $\mathrm{EPPF}$ of the associated Gibbs partition of $[n]$ can be expressed as
\begin{equation}
\left[\sum_{j=1}^{k}\mathbb{P}^{(k)}_{\delta,0}(j)\tilde{V}_{n,j}\right]p_{\alpha}(n_{1},\ldots,n_{k})
\label{Gibbsalpha|deltaV}.
\end{equation}
\end{prop}
\begin{proof}The $\mathrm{EPPF}$ is equivalent to $\int_{0}^{\infty}p_{\alpha|\delta}(n_{1},\ldots,n_{k}|y)\gamma(dy),$ indicated by~(\ref{Gibbsalpha|delta}).
\end{proof}

\begin{rem}(\ref{Gibbsalpha|deltaV}) provides a description of any mass partition with distribution $\mathrm{PK}_{\alpha}(\gamma_{\alpha|\delta}),$ 
where 
$\gamma_{\alpha|\delta}(dt)/dt=\mathbb{E}\big[h\big(tS^{\frac1\alpha}_{\delta}\big)\big] f_{\alpha}(t),$
regardless of whether or not it actually arises from a fragmentation operation.
\end{rem}

As a check, in the case where $(P_{\ell,0})\sim \mathrm{PD}(\alpha\delta,\theta),$ (\ref{Gibbsalpha|deltaV}) must satisfy
\begin{equation}
\sum_{j=1}^{k}\mathbb{P}^{(k)}_{\delta,0}(j)\frac{\Gamma\big(\frac{\theta}{\alpha\delta}+j\big)}{\Gamma\big(\frac{\theta}{\alpha\delta}+1\big)\Gamma(j)}=
\frac{\Gamma\big(\frac{\theta}{\alpha}+k\big)}{\Gamma\big(\frac{\theta}{\alpha}+1\big)\Gamma(k)}.
\label{Pitmanmoments}
\end{equation}
However,~(\ref{Pitmanmoments}) is verified since, similar to~(\ref{PitmanKnmoment}), it agrees with \cite[exercise 3.2.9, p.66]{Pit06}, with $k$ in place of $n.$
There is the following Corollary in the case of $\delta=\frac12$.

\begin{cor}Specializing Theorem~\ref{PropVnfrag} to the case of $\delta=\frac12,$ where $(P_{k,0})\sim \mathrm{PK}_{\frac\alpha2}(\gamma),$ and $\tilde{V}_{n,k}=\mathbb{E}_{\frac\alpha2}[h(S_{\frac\alpha2})|K_{n}=k],$ the $\mathrm{PK}_{\alpha}(\gamma_{\alpha|\frac{1}{2}})$ $\mathrm{EPPF}$ 
in~(\ref{Gibbsalpha|deltaV}) can be expressed as
\begin{equation}
\left[\sum_{j=1}^{k}
{{2k-j-1}\choose{k-1}}2^{j+1-2k}
\tilde{V}_{n,j}\right]p_{\alpha}(n_{1},\ldots,n_{k}).
\label{GibbsalphahalfV}
\end{equation}
\end{cor}

\subsubsection{An equivalent expression for the EPPF}\label{sec:equivEPPF}
A remarkable feature of the EPPF's in Theorem~\ref{TheoremGibbsparitionFrag} and (\ref{Gibbsalpha|deltaV}) is that in general, in view of the results of \cite{HJL}, the most tractable results only require $\alpha\delta=m/r$ and otherwise $\alpha$ can be quite general. As we shall show in the appendix, those results require more delicate distributional arguments exploiting properties of the dual coagulation operation. It is however worthwhile to describe an alternative expression for the EPPF, which we do next. It is the case for the expression below that, in general, in order to get expressions in terms of Meijer $G$ functions, both $\alpha$ and $\delta$ have to be rational numbers. In that case, the Meijer $G$ representation is not difficult to obtain, but for brevity we do not include it here. 
\begin{prop}
Suppose that $(P_{k})\sim\mathrm{PD}_{\alpha|\delta}(\alpha|y)$ as in Theorem~\ref{TheoremGibbsparitionFrag}. The EPPF may be expressed in terms of Fox $H$ functions as
$$
\frac{\alpha\Hfun{0,2}{2,2}{y}{\left(1-\frac{1}{\alpha\delta},\frac{1}{\alpha\delta}\right), \left(1-\frac{1}{\alpha}-k,\frac{1}{\alpha}\right)}
{\left(1-\frac{1}{\alpha},\frac{1}{\alpha}\right),(-n,1)}}
{\Hfun{0,1}{1,1}{y}{\left(1-\frac{1}{\alpha\delta},\frac{1}{\alpha\delta}\right)}
{\left(0,1\right)}}\frac{\Gamma(n)}{\Gamma(k)}p_{\alpha}(n_{1},\ldots,n_{k}).
$$
\end{prop}
\begin{proof} 
$\alpha y^{\alpha-1}\int_{0}^{\infty}f_{\delta}({(y/t)}^{\alpha})t^{-\alpha}f^{(n-k\alpha)}_{\alpha,k\alpha}(t)dt$ is, by the definition of $f^{(n-k\alpha)}_{\alpha,k\alpha}(t),$ the density of the random variable
\begin{equation}
\frac{S_{\alpha,n}}{\beta^{\frac1\alpha}_{k,\frac{n}{\alpha}-k}}\times
S_{\delta}^{\frac1\alpha}.
\label{basicSadcond}
\end{equation}
The Fox $H$ expression for this is obtained by noting the Fox $H$ representations for $f_{\delta}$ and $f^{(n-k\alpha)}_{\alpha,k\alpha},$ followed by applying~\cite[Theorem 4.1]{CarterSpringer}. One then uses the Fox $H$ representation for $f_{\alpha\delta}(y).$ Otherwise details are similar to the arguments in \cite{HJL}.
\end{proof}

\subsection{Generating $\mathrm{PD}_{\alpha|\delta}(\alpha|y)$ partitions via $\mathrm{PD}(\alpha,-\alpha\delta)$ fragmentation of partitions}\label{sec:sampling}
While the EPPF's~(\ref{Gibbsalpha|delta}) and~(\ref{Gibbsalpha|deltaV}) are quite interesting from various perspectives, it is not entirely necessary to employ them directly to obtain random partitions from $\mathrm{PD}_{\alpha|\delta}(\alpha|y)$ and  $\mathrm{PK}_{\alpha}(\gamma_{\alpha|\delta}).$ A two-stage sampling scheme may be employed utilizing the dual partition based interpretation of the $\mathrm{Frag}_{\alpha,-\alpha\delta}$ operator. The following scheme can be deduced from Bertoin~\cite{BerFrag}, see also \cite{Pit99Coag,Pit06}. 
\begin{itemize}
\item[1.]Generate $n$ iid $\mathrm{PD}(\alpha,-\alpha\delta)$ partitions of $[n],$ say, $\mathcal{A}_{1},\ldots,\mathcal{A}_{n},$ where, for each $i,$ $\mathcal{A}_{i}:=\{A^{(i)}_{1},\ldots,A^{(i)}_{M^{(i)}_{n}}\}$ with $M^{(i)}_{n}$ blocks.
\item[2.]Independent of this, for each fixed $y,$ generate a $\mathrm{PD}(\alpha\delta|y)$ partition of $[n]$, say, $\{H_{1},\ldots,H_{K_{n}(y)}\}$ where $K_{n}(y)$ denotes the number of blocks.
\item[3.] Consider the pairs $(H_{i}, \mathcal{A}_{i})$, for $i=1,\ldots, K_{n}(y)=\ell\leq n.$
\item[4.]For $i=1,\ldots,\ell$, fragment $H_{i}$ by $\mathcal{A}_{i}$ according to
$$
\mathcal{H}_{i}=\bigg\{H_{i,j}:=H_{i}\cap A^{(i)}_{j}: H_{i}\cap A^{(i)}_{j}\neq \emptyset, j\in\big\{1,\ldots,M^{(i)}_{n}\big\}\bigg\}
$$
\item[5.]The collection $\{H_{i,j}\in \mathcal{H}_{i}: i\in [K_{n}(y)]\}$ (arranged according to least element) constitutes a $\mathrm{PD}_{\alpha|\delta}(\alpha|y)$ partition of $[n],$ with $\hat{K}_{n}(y):=\sum_{i=1}^{K_{n}(y)}|\mathcal{H}_{i}|$ number of blocks. 
\item[6.]Replace step 2 with a $\mathrm{PK}_{\alpha\delta}(\gamma)$ partition of $[n]$ to obtain a 
 $\mathrm{PK}_{\alpha}(\gamma_{\alpha|\delta})$ partition of $[n]$.
\end{itemize}
\begin{rem}
The scheme above requires sampling of a $\mathrm{PD}(\alpha\delta|y)$ partition of $[n].$ The 
relevant results of~\cite{HJL} show that this is the easiest when $\alpha\delta$ is a rational number. In that case, $\mathbb{G}^{(n,k)}_{\alpha\delta}(y)$ has a tractable representation in terms of Meijer $G$ functions. So, this applies to, in particular, $(P_{k,1})\sim \mathrm{PD}_{\alpha|\frac{m}{r\alpha}}(\alpha|y)$ for every $\alpha>\frac{m}r,$ where $m<r$ are coprime positive integers. We look at perhaps the most remarkable case, $\mathrm{PD}_{\alpha|\frac{1}{2\alpha}}(\alpha|y),$ in the next section.
\end{rem}

\subsection{$\mathrm{PD}(\alpha,-\frac12)$ Fragmentation of a Brownian excursion partition conditioned on its local time}
Following Pitman~\cite[Section~8]{Pit03} and \cite[Section~4.5, p.90]{Pit06}, let $(P_{\ell,0})\sim \mathrm{PD}\big(\frac12,0\big)$ denote the ranked excursion lengths of a standard Brownian motion $B:=(B_{t}:t\in [0,1])$, with corresponding local time at $0$ up till time $1$ given by $L_{1}:=\big(2S_{\frac12}\big)^{-\frac12} \overset{d}=\sqrt{2G_{\frac12}}\overset{d}=|B_{1}|.$ 
Then, it follows that $(P_{\ell,0})|L_{1}=\lambda$ has a $\mathrm{PD}(\frac{1}{2}|\frac{1}{2}\lambda^{-2})$ distribution. 

\begin{prop}\label{PropHermitealpha} As in~\cite[Proposition 14]{Pit03}, consider the sequence $(\tilde{P}_{\ell,0}(\lambda))$ of length-biased random permutations of the lengths of excursions of a Brownian motion or a standard
Brownian bridge over $[0,1]$ conditioning on $L_{1}=\lambda$, as defined by~(\ref{condstick}), and the corresponding $\mathrm{EPPF}$ in~(\ref{hermiteEPPF}). Then, for a fixed $\alpha\in \big(\frac12,1\big)$, the ranked mass partition $(P_{\ell,1}(\lambda))$ defined as the $\mathrm{PD}(\alpha,-\frac12)$ fragmentation of $(\tilde{P}_{\ell,0}(\lambda))$ has a $\mathrm{PD}_{\alpha|\frac{1}{2\alpha}}\big(\alpha\big|\frac{1}{2}\lambda^{-2}\big)$ distribution. That is,
\begin{equation}
({P_{\ell,1}(\lambda)}):=
\mathrm{Frag}_{\alpha,-\frac12}((\tilde{P}_{\ell,0}(\lambda)))\sim \mathrm{PD}_{\alpha|\frac{1}{2\alpha}}\left(\alpha\bigg|\frac{1}{2}\lambda^{-2}\right)
\label{Fragalpha1/2}.
\end{equation}
The corresponding EPPF of the $\mathrm{PD}_{\alpha|\frac{1}{2\alpha}}\big(\alpha|\frac{1}{2}\lambda^{-2}\big)$ partition of $[n]$ can be expressed in terms of a mixture of Hermite functions,
\begin{equation}
\left[\sum_{j=1}^{k}\mathbb{P}^{(k)}_{\frac{1}{2\alpha}}(j)2^{n-1}\lambda^{j-1}\tilde{h}_{j+1-2n}(\lambda)\frac{\Gamma(n)}{\Gamma(j)}\right]p_{\alpha}(n_{1},\ldots,n_{k}),
\label{MixedalphahermiteEPPF}
\end{equation}
where $\mathbb{P}^{(k)}_{\frac{1}{2\alpha}}(j)=\mathbb{P}_{\frac{1}{2\alpha},0}(K_{k}=j)$ is the distribution of the number of blocks in a\break $\mathrm{PD}(\frac{1}{2\alpha},0)$ partition of $[k].$
\end{prop}
\begin{proof}The result is just a special case of Proposition~\ref{PropPDarcondlaw} and Theorem~\ref{TheoremGibbsparitionFrag}, with $\delta=\frac{1}{2\alpha},$  and otherwise using the explicit form of the EPPF in~(\ref{hermiteEPPF}). 
\end{proof}
\begin{rem}In order to obtain a partition of $[n]$ corresponding to the EPPF in~(\ref{MixedalphahermiteEPPF}), one can sample from (\ref{hermiteEPPF}) via the prediction rules indicated in~\cite[eqs.~(111) and (112)]{Pit03}, or otherwise employ the scheme described in Section~\ref{sec:sampling}.
\end{rem}
\begin{rem}\label{remarkfragextend}
Besides properties~(\ref{condstick}) and~(\ref{hermiteEPPF}), a striking feature of the $\mathrm{PD}(\frac{1}{2}|t)$ distribution is that one can construct continuous time fragmentation processes based on processes $(\Pi_{\infty}(\lambda);~\lambda\ge 0)$ and 
$((\tilde{P}_{\ell,0}(\lambda));~\lambda\ge 0)$ by varying in $\lambda$ as described in~\cite{AldousPit,BerBrown,BerFrag,Pit06}. \cite{MiermontSchweinsberg2003} showed that one cannot extend such constructions to the distribution of $(P_{\ell})|S_{\alpha}=t,$ $\mathrm{PD}(\alpha|t),$ for other choices of $\alpha.$ However, in the present setting, since the $\mathrm{Frag}_{\alpha,-\frac12}$ operator is independent of $((\tilde{P}_{\ell,0}(\lambda));~\lambda\ge 0)$ and is not affected by the time $\lambda,$ it would be interesting to investigate the properties of $(({P_{\ell,1}(\lambda)});~\lambda\ge 0),$ defined by (\ref{Fragalpha1/2}), in this regard. Note the range $\alpha=\frac{1}{\tilde{\alpha}}\in \big(\frac12,1\big)$ suggests possible connections to stable processes or trees of index $1<\tilde{\alpha}<2$~(See, for instance,~\cite{CurienHaas,Miermont2005}).  
\end{rem}

\subsection{$\mathrm{PD}(\alpha,-\alpha\delta)$ Fragmenting the Mittag-Leffler function case}
We now show that a $\mathrm{PD}(\alpha,-\alpha\delta)$ fragmentation of $(P_{\ell,0}(\lambda))\sim \mathbb{L}^{(0)}_{\alpha\delta,\theta}(\lambda)$ leads to a mass partition $(P_{\ell,1}(\lambda)),$ whose corresponding $\alpha$-diversity, say $\hat{Z}_{1}(\lambda),$ has density expressed in terms of a ratio of two Mittag-Leffler functions, 
\begin{equation}
\hat{g}^{(1)}_{\alpha|\delta,\theta}(s|\lambda)=\frac{\mathrm{E}^{(\frac{\theta}{\alpha\delta}+1)}_{\delta,\frac{\theta}{\alpha}+1}\left(-{\lambda s^{\delta}}\right)
g_{\alpha,\theta}(s)}{\mathrm{E}^{(\frac{\theta}{\alpha\delta}+1)}_{\alpha\delta,\theta+1}(-\lambda)}.
\label{MittagFragdensity}
\end{equation}

\begin{prop}As in~(\ref{MittagEPPF}), and~(\ref{genMittaglambda}), where $\alpha$ is replaced by $\alpha\delta,$ let $(P_{\ell,0}(\lambda))\sim \mathbb{L}^{(0)}_{\alpha\delta,\theta}(\lambda),$
with EPPF
\begin{equation}
\frac{\mathrm{E}^{(\frac{\theta}{\alpha\delta}+k)}_{\alpha\delta,\theta+n}(-\lambda)}{\mathrm{E}^{(\frac{\theta}{\alpha\delta}+1)}_{\alpha\delta,\theta+1}(-\lambda)}p_{\alpha\delta,\theta}(n_{1},\ldots,n_{k}).
\label{MittagEPPF3}
\end{equation}
Then,
$
(P_{\ell,1}(\lambda)):=\mathrm{Frag}_{\alpha,-\alpha\delta}((P_{\ell,0}(\lambda)))\sim\mathbb{L}^{(1)}_{\alpha|\delta,\theta}(\lambda),$ where 
\begin{equation}
\mathbb{L}^{(1)}_{\alpha|\delta,\theta}(\lambda)=\int_{0}^{\infty}\mathrm{PD}(\alpha|s^{-\frac1\alpha})\,\hat{g}^{(1)}_{\alpha|\delta,\theta}(s|\lambda)ds,
\label{MittagFragdist}
\end{equation}
with $\hat{g}^{(1)}_{\alpha|\delta,\theta}(s|\lambda)$ given in~(\ref{MittagFragdensity}).
\begin{enumerate}
\item[(i)]The $\mathbb{L}^{(1)}_{\alpha|\delta,\theta}(\lambda)~\mathrm{EPPF}$ of $[n]$ can be expressed as 
\begin{equation}
\left[\sum_{j=1}^{k}\mathbb{P}^{(k)}_{\delta,\frac{\theta}{\alpha}}(j)\frac{\mathrm{E}^{(\frac{\theta}{\alpha\delta}+j)}_{\alpha\delta,\theta+n}(-\lambda)}{\mathrm{E}^{(\frac{\theta}{\alpha\delta}+1)}_{\alpha\delta,\theta+1}(-\lambda)}\right]p_{\alpha,\theta}(n_{1},\ldots,n_{k})
\label{Mittagalpha|deltaV2}.
\end{equation}
\item[(ii)]
In~(\ref{Mittagalpha|deltaV2}), $\sum_{j=1}^{k}\mathbb{P}^{(k)}_{\delta,\frac{\theta}{\alpha}}(j)\mathrm{E}^{(\frac{\theta}{\alpha\delta}+j)}_{\alpha\delta,\theta+n}(-\lambda)$ is equivalent to
$$
\mathbb{E}\left[\mathrm{E}^{(\frac{\theta}{\alpha\delta}+1)}_{\delta,\frac{\theta}{\alpha}+1}\big(-\lambda S^{-\alpha\delta}_{\alpha,n}\beta^{\delta}_{k,\frac{n}{\alpha}-k}\big) \times 
S^{-\theta}_{\alpha,n}\beta^{\frac{\theta}{\alpha}}_{k,\frac{n}{\alpha}-k}\right]/\mathbb{E}[S^{-\theta}_{\alpha}],
$$
which is given by
$$
\sum_{\ell=0}^{\infty}\frac{{(-\lambda)}^{\ell}}{\ell!}\frac{\Gamma(\theta+1)}
{\Gamma(\alpha\delta\ell+\theta+n)}
\frac{\Gamma(\frac{\theta}{\alpha\delta}+1+\ell)\Gamma(\delta\ell+\frac{\theta}{\alpha}+k)}
{\Gamma(\frac{\theta}{\alpha\delta}+1)\Gamma(\delta
\ell+\frac{\theta}{\alpha}+1)}.
$$
\end{enumerate}
\end{prop}
\begin{proof}
Using Corollary~\ref{corPDarcondlaw} and Proposition~\ref{PropVnfrag},~(\ref{MittagFragdist}) follows from the fact that
$$
h(y)=\frac{{\mbox e}^{-\lambda y^{-\alpha\delta}}y^{-\theta}}{\mathrm{E}^{(\frac{\theta}{\alpha\delta}+1)}_{\alpha\delta,\theta+1}(-\lambda)\mathbb{E}\big[S^{-\theta}_{\alpha\delta}\big]} \quad{\mbox { and }}\quad\mathbb{E}\left[{\mbox e}^{-{\lambda}{ y^{-\alpha\delta}}S^{-\delta}_{\delta,\frac{\theta}{\alpha}}}\right]=\mathrm{E}^{(\frac{\theta}{\alpha\delta}+1)}_{\delta,\frac{\theta}{\alpha}+1}\left(-\frac{\lambda}{ y^{\alpha\delta}}
\right).
$$
From~(\ref{MittagVnk}),~(\ref{Gibbsalpha|deltaV}), and~(\ref{MittagEPPF3}) the EPPF can be represented as
$$
\left[\sum_{j=1}^{k}\mathbb{P}^{(k)}_{\delta,0}(j)\frac{\mathrm{E}^{(\frac{\theta}{\alpha\delta}+j)}_{\alpha\delta,\theta+n}(-\lambda)}{\mathrm{E}^{(\frac{\theta}{\alpha\delta}+1)}_{\alpha\delta,\theta+1}(-\lambda)}\frac{\mathbb{E}\big[S^{-\theta}_{\alpha\delta}|K_{n}=j\big]}{
\mathbb{E}\big[S^{-\theta}_{\alpha\delta}\big]}\right]p_{\alpha}(n_{1},\ldots,n_{k}),
$$
which, along with~(\ref{appendixdecomp}), yields (i). Statement~(ii) follows from $\mathbb{E}[S^{-\alpha\delta\ell-\theta}_{\alpha,n}]\mathbb{E}\big[\beta^{\delta\ell+\frac{\theta}{\alpha}}_{k,\frac{n}{\alpha}-k}\big]=\dfrac{\Gamma(n)\Gamma(\delta\ell+\frac{\theta}{\alpha}+k)}{\Gamma(k)\Gamma(\alpha\delta\ell+\theta+n)}.$ 
\end{proof}
In the case of a $\mathrm{PD}(\alpha|-\frac{\alpha}{2})$ fragmentation of
$(P_{\ell,0}(\lambda))\sim \mathbb{L}^{(0)}_{\frac{\alpha}{2},\theta}(\lambda),$
where $\delta=\frac12,$ we see from~(\ref{MittagFragdist}) that the density of $\hat{Z}_{1}(\lambda)$ can be expressed in terms of a Hermite function,
$$
\hat{g}^{(1)}_{\alpha|\frac{1}{2},\theta}(s|\lambda)=\frac{\mathbb{E}\big[{|B_{1}|}^{\frac{2\theta}{\alpha}+1}\big]\tilde{h}_{-(\frac{2\theta}{\alpha}+1)}(\lambda\sqrt{2s})g_{\alpha,\theta}(s)}{\mathrm{E}^{(\frac{2\theta}{\alpha}+1)}_{\frac{\alpha}{2},\theta+1}(-\lambda)}.
$$
When $\alpha=\frac12,$ $\big(P_{\ell,1}\big(\frac{\lambda}{\sqrt{2}}\big)\big)\sim\mathbb{L}^{(1)}_{\frac{1}{2}|\frac{1}{2},\theta}\big(\frac{\lambda}{\sqrt{2}}\big)$ results in a local time $\hat{L}_{1}(\lambda)=\hat{Z}_{1}\big(\frac{\lambda}{\sqrt{2}}\big)/\sqrt{2}$ with density
$$
\frac{\mathbb{E}\big[{|B_{1}|}^{4\theta+1}\big]\tilde{h}_{-(4\theta+1)}(\lambda\sqrt{x})f_{L_{1,\theta}}(x)}{\mathrm{E}^{(4\theta+1)}_{\frac{1}{4},\theta+1}\big(-\frac{\lambda}{\sqrt{2}}\big)}.
$$
Furthermore, for $\hat{Z}_{1}(\frac{\lambda}{\sqrt{2}})$, where $\alpha\delta=1/2,$ one has, for $\alpha\in (\frac{1}{2},1),$ densities of the form
$$
\hat{g}^{(1)}_{\alpha|\frac{1}{2\alpha},\theta}\left(s\bigg|\frac{\lambda}{\sqrt{2}}\right)=\frac{\mathrm{E}^{(2{\theta}+1)}_{\frac{1}{2\alpha},\frac{\theta}{\alpha}+1}\left(-{\frac{\lambda}{\sqrt{2}} s^{\frac{1}{2\alpha}}}\right)
g_{\alpha,\theta}(s)}
{\mathbb{E}\big[{|B_{1}|}^{2\theta+1}\big]\tilde{h}_{-(2\theta+1)}(\lambda)}.
$$


\appendix\section{Results derived from the $\mathrm{Coag}_{\delta,\frac{\theta}{\alpha}}$ operator}
\subsection{Identities associated with the $\mathrm{Coag}_{\delta,\frac{\theta}{\alpha}}$ operator}
We now proceed to derive a variety of results related to the~coagulation operator, $\mathrm{Coag}_{\delta,\frac{\theta}{\alpha}}.$
As shown in \cite{BerLegall00}, see also \cite{BerFrag,Pit99Coag,Pit06}, the coagulation results are the easiest to obtain in the important case of $\theta=0,$ corresponding to the Bolthausen-Sznitman coalescent.  This in  effect corresponds to well-known results about composition of independent stable subordinators 
$\hat{S}_{\alpha\delta}(t)\overset{d}=\hat{S}_{\alpha}(\hat{S}_{\delta}(t)),$ and corresponding random variables 
$S_{\alpha,\delta}\overset{d}=S_{\alpha}S^{1/\alpha}_{\delta}.$ 
The next lemma, which is known in some form (see~\cite{DevroyeJames2}), illustrates that many properties in the general case can be deduced from the case of $\theta=0$ by a simple change of measure argument. It in effect presents an alternative illustration of  how the choices of $(\alpha,\theta)$ and $({\delta},\frac{\theta}{\alpha})$ arise in~\cite{Pit99Coag}.

\begin{lem}\label{coaglemma}Let $0<\alpha_{i}<1,$ for $i=1,\ldots,m$, denote indices for independent stable random variables and their two-parameter counterparts.  Set $\tilde{\alpha}_{k}:=\prod_{i=1}^{k}\alpha_{i}$, for $k=1,\ldots, m,$ and choose $\theta>-\tilde{\alpha}_{m}.$
The decomposition,
\begin{equation}
S_{\tilde{\alpha}_{m}}=S_{\alpha_{1}}\times S^{1/\alpha_{1}}_{\alpha_{2}}\times \cdots\times S_{\alpha_{m}}^{1/\tilde{\alpha}_{m-1}},
\label{stabledecomp}
\end{equation}
implies, under a $\mathrm{PD}(\tilde{\alpha}_{m},\theta)$ distribution,
\begin{equation}
S_{\tilde{\alpha}_{m},\theta}=S_{\alpha_{1},\theta}\times S^{1/\alpha_{1}}_{\alpha_{2},\frac{\theta}{\alpha_{1}}}\times \cdots\times S_{\alpha_{m},\frac{\theta}{\tilde{\alpha}_{m-1}}}^{1/\tilde{\alpha}_{m-1}}.
\label{coagdecomp}
\end{equation}
As a special case $S_{\alpha^{m},\theta}=\prod_{i=1}^{m}S^{\alpha^{1-i}}_{\alpha,\theta\alpha^{1-i}}.$
\end{lem}
\begin{proof}
The law of $S_{\tilde{\alpha}_{m},\theta}$ is characterized via the law of $S_{\tilde{\alpha}_{m}}$ by 
$$
\mathbb{E}\big[S^{-\theta}_{\tilde{\alpha}_{m}}\big]
\mathbb{E}[\varphi(S_{\tilde{\alpha}_{m},\theta})]=
\mathbb{E}\big[S^{-\theta}_{\tilde{\alpha}_{m}} \varphi(S_{\tilde{\alpha}_{m}})\big].
$$
Substituting $S_{\tilde{\alpha}_{m}}$ with its decomposition in (\ref{stabledecomp}) deduces the change of measure on the individual components leading to~(\ref{coagdecomp}).
\end{proof}

A version of the next, simple but important, consequence of Pitman's~\cite{Pit99Coag} coagulation operation expressed in terms of random partitions of $[n]$, appears in~\cite{JamesRoss}.~(See~\cite[Section~4]{Camer} for related results.)

\begin{lem}\label{lemmablockscoag} Consider two independent mass partitions with respective distributions $(P_{k,1})\sim\mathrm{PD}(\alpha,\theta)$ and $(P_{k,2})\sim\mathrm{PD}(\delta,\frac{\theta}{\alpha}),$ employed in the coagulation operation described by~$\mathrm{Coag}_{\delta,\frac{\theta}{\alpha}}.$ For positive integers $m$ and $n,$ let $K^{(1)}(n):=K^{(1)}_{n}$ and $K^{(2)}(m):=K^{(2)}_{m}$ denote independent random variables corresponding to the number of blocks in partitions of $[n]$ and $[m]$ generated from  the respective $\mathrm{CRP}(\alpha,\theta)$ and $\mathrm{CRP}(\delta,\frac{\theta}{\alpha})$ processes. Let $K_{n}$ denote the number of blocks of the resulting $\mathrm{PD}(\alpha\delta,\theta)$ partition of $[n],$ corresponding to the $(P_{k,0})\sim \mathrm{PD}(\alpha\delta,\theta)$ mass partition obtained from~$(P_{k,0})=\mathrm{Coag}_{\delta,\frac{\theta}{\alpha}}((P_{k,1}))$. Then,
\begin{enumerate}
\item[(i)] $
K_{n}:=K^{(2)}(K^{(1)}(n)).
$
\item[(ii)]As $n\rightarrow \infty,$ $n^{-\alpha\delta}K_{n}\overset{a.s.}\sim S^{-\alpha\delta}_{\alpha\delta,\theta}=S^{-\alpha\delta}_{\alpha,\theta}\times S^{-\delta}_{\delta,\frac{\theta}{\alpha}}.$
\end{enumerate}  
\end{lem}

\subsection{More Beta identities and results for $K_{n}$}
\begin{prop}
Let $\alpha=\frac1{r^{m}}$, for $r=2,3,\ldots,$ and $m=1,2,\ldots$. Then, there is the identity,
\begin{equation}
\prod_{i=1}^{n}\beta^{r^{m}}_{r^{m}(\theta+j-1)+1,r^{m}-1}\overset{d}=
\prod_{i=1}^{r^{m}-1} \beta_{\theta+\frac{i}{r^{m}},n}\overset{d}=
\prod_{j=1}^{m}\prod_{i=1}^{r-1} \beta^{r^{j-1}}_{\theta r^{j-1}+\frac{i}{r},nr^{j-1}}
\label{betaprodsimple2},
\end{equation}
which represents a reduction from $r^{m}-1$ terms to $m(r-1)$ terms.
For $\alpha=\frac14,$ setting $r=2$ and $m=2$ yields
\begin{equation}
\prod_{i=1}^{n}\beta^{4}_{4(\theta+j-1)+1,3} \overset{d}=\beta_{\theta+\frac{1}{4},n}
\beta_{\theta+\frac{1}{2},n} \beta_{\theta+\frac{3}{4},n}\overset{d}=\beta_{\theta+\frac{1}{2},n}
\beta^{2}_{2\theta+\frac{1}{2},2n}.
\label{betaprodsimple3}
\end{equation}
\end{prop}
\begin{proof}
Specializing (\ref{jamesidPDrecurse2})
and~(\ref{stablegeninteger}) to the case of $\alpha=\frac1{r^{m}}$, and further applying Lemma~\ref{coaglemma}, leads to 
$$
\prod_{j=1}^{m}\prod_{i=1}^{r-1} G^{r^{j-1}}_{\theta r^{j-1}+\frac{i}{r}}
\overset{d}=
\prod_{j=1}^{m}\prod_{i=1}^{r-1} G^{r^{j-1}}_{(\theta+n)r^{j-1}+\frac{i}{r}}
\times \prod_{i=1}^{n}\beta^{r^{m}}_{r^{m}(\theta+j-1)+1,r^{m}-1}.
$$
The result follows as in Lemma~\ref{Propbetaid1}.
\end{proof}

\begin{prop}\label{propKncoag} Applying  Lemma~\ref{lemmablockscoag} leads to the following identities.
\begin{enumerate}
\item[(i)]
$
\mathbb{P}_{\alpha\delta,\theta}^{(n)}(k)=\displaystyle\sum_{\ell=k}^{n}\mathbb{P}_{\alpha,\theta}^{(n)}(\ell)\mathbb{P}_{\delta,\frac{\theta}{\alpha}}^{(\ell)}(k),
$
for $\theta>-\alpha\delta$.
\item[(ii)]There is the easily checked  result,
\begin{equation}
\frac{\mathbb{E}_{\alpha\delta,0}[S^{-\theta}_{\alpha\delta}|K_{n}=k]}{\mathbb{E}[S^{-\theta}_{\alpha\delta}]}=\frac{\mathbb{E}_{\alpha,0}[S^{-\theta}_{\alpha}|K_{n}=\ell]}{\mathbb{E}[S^{-\theta}_{\alpha}]}\times 
\frac{\mathbb{E}_{\delta,0}\big[S^{-\frac{\theta}{\alpha}}_{\delta}|K_{\ell}=k\big]}{\mathbb{E}\big[S^{-\frac{\theta}{\alpha}}_{\delta}\big]}.
\label{appendixdecomp}
\end{equation}
\item[(iii)]Statement (ii) coincides with the fact that  statement (i) is determined by the case of $\theta=0,$ 
\begin{equation}
\mathbb{P}_{\alpha\delta,0}^{(n)}(k)=\sum_{\ell=k}^{n}\mathbb{P}_{\alpha,0}^{(n)}(\ell)\mathbb{P}_{\delta,0}^{(\ell)}(k),
\label{Kncoagid}
\end{equation}
which leads to identities for the generalized Stirling numbers,
 \begin{equation}
a(n,k,\alpha\delta)=\sum_{\ell=k}^{n}a(n,\ell,\alpha)a(\ell,k,\delta),
\label{Stirlingid}
\end{equation}
or, equivalently,
 \begin{equation}
S_{\alpha\delta}(n,k)=\sum_{\ell=k}^{n}\alpha^{\ell-k}S_{\alpha}(n,\ell)S_{\delta}(\ell,k),
\label{Stirlingid2}
\end{equation}
where $S_\alpha(n,\ell)=\alpha^{-\ell}a(n,\ell,\alpha)=\sum_{j=1}^{\ell}\frac{{(-1)}^{j}}{\ell!}{\ell\choose j}(-j\alpha)_{n}$.
\end{enumerate}
\end{prop}
\begin{proof}Lemma~\ref{lemmablockscoag} implies that $\mathbb{P}_{\alpha\delta,\theta}(K_{n}=k|K^{(1)}_{n}=\ell)=\mathbb{P}_{\alpha,\theta}(K^{(2)}_{\ell}=k),$ which leads to statement (i) and~(\ref{Kncoagid}). (\ref{Stirlingid}) and~(\ref{Stirlingid2}) follow from this. The remainder of the results are easily checked.
\end{proof}
\begin{lem}\label{PropConKn} Consider the converse to the setting in Lemma~\ref{lemmablockscoag}, where $(P_{k,0})\sim \mathrm{PD}(\alpha\delta,\theta)$ and $(P_{k,1})\sim\mathrm{PD}(\alpha,\theta)$ as defined by the fragmentation operator~(\ref{fragoperator}). Then, for all $\theta>-\alpha\delta,$ the conditional distribution of the number of blocks of the partition of $[n]$ arising from fragmenting the $\mathrm{PD}(\alpha\delta,\theta)$ partition of $[n]$ with $k$ blocks by iid partitions of $[n]$ from a $\mathrm{PD}(\alpha,-\alpha\delta)$ distribution, that is, $K^{(1)}_{n}|K_{n}=k,$ can be expressed as
$$
\mathbb{P}^{(n)}_{\alpha|\alpha\delta}(\ell|k):=\frac{\mathbb{P}_{\delta,0}(K^{(2)}_{\ell}=k)\mathbb{P}_{\alpha,0}(K^{(1)}_{n}=\ell)}{
\mathbb{P}_{\alpha\delta,0}(K_{n}=k)},\qquad \ell=k,\ldots,n,
$$
which does not depend on $\theta.$ 
\end{lem}

\subsubsection{$K_{n}$ in the case of 
$\alpha=\frac14$}
Proposition~\ref{propKncoag} leads to a new simpler representation in the case of $K_{n}$ derived from a $\mathrm{PD}(\frac{1}{4},0)$ distribution.

\begin{cor}Using 
$
\mathbb{P}_{\frac{1}{2},0}(K_{n}=\ell)={{2n-\ell-1}\choose{n-1}}2^{\ell+1-2n},$ it follows that
\begin{enumerate}
\item[(i)]$\mathbb{P}^{(n)}_{\frac{1}{4},0}(k):=\mathbb{P}_{\frac{1}{4},0}(K_{n}=k)$ can be expressed as
\begin{equation}
\mathbb{P}^{(n)}_{\frac{1}{4},0}(k)=2^{k+2-2n}\sum_{\ell=k}^{n}2^{-\ell}{{2n-\ell-1}\choose{n-1}}
{{2\ell-k-1}\choose{\ell-1}},
\end{equation}
which further simplifies to 
\begin{equation}
\mathbb{P}^{(n)}_{\frac{1}{4},0}(k)=
2^{2-2n} \binom{2n-k-1}{n-1} 
\pFq{3}{2}{\frac{k+1}{2}, \frac{k}{2}, k-n}{k, 1+k-2n}{2}.
\label{quarterKn}
\end{equation}
\item[(ii)]As a special case of Lemma~\ref{PropConKn}, set $\alpha=\delta=\frac12,$ for $\ell=k,\ldots,n,$
$$
\mathbb{P}^{(n)}_{\frac{1}{2}|\frac{1}{4}}(\ell|k)=\frac{2^{k-\ell}{\displaystyle\binom{2\ell-k-1}{\ell-1}}{\displaystyle\binom{2n-\ell-1}{n-1}}}{\displaystyle\binom{2n-k-1}{n-1} \, 
\pFq{3}{2}{\frac{k+1}{2}, \frac{k}{2}, k-n}{k, 1+k-2n}{2}}.
$$
\end{enumerate}
\end{cor}
\begin{rem} The simplification in (\ref{quarterKn}) was communicated to us by Jim Pitman using \textsc{Mathematica}.
\end{rem}

\section{Conditional representations for $S_{\alpha\delta,\theta}$}
We now describe the key distributional results to obtain the description of the EPPF in Theorem~\ref{TheoremGibbsparitionFrag}. Throughout, similar to~\cite{BerPit2000},
we set $\tau_{\alpha\delta}(\cdot):=\tau_{\alpha}(\tau_{\delta}(\cdot))$ for independent generalized gamma subordinators $(\tau_{\alpha},\tau_{\delta}).$

\begin{lem}\label{PropSalphadelta} 
Consider $S_{\alpha\delta,\theta}
$ under a 
$\mathrm{PD}(\alpha\delta,\theta)$ distribution, with $K_n$ number of blocks of a partition of $[n]$. For each $n\ge 1,$ $S_{\alpha\delta,\theta}$ has an equivalent representation,
\begin{equation}\label{YSalphadelta} 
Y^{(n-K_{n}\alpha\delta)}_{\alpha\delta,\theta+K_{n}\alpha\delta}
=S_{\alpha,\theta+n}
\times
{\left[Y^{(\frac{n}{\alpha}-K_{n}\delta)}_{\delta,\frac{\theta}{\alpha}+K_{n}\delta}\right]}^{\frac1\alpha}.
\end{equation}
\end{lem}

\begin{proof}
An initial description of $S_{\alpha\delta,\theta}$ in terms of $Y^{(n-K_{n}\alpha\delta)}_{\alpha\delta,\theta+K_{n}\alpha\delta}:= S_{\alpha \delta, \theta+n} \times \beta^{-\frac1{\alpha\delta}}_{\frac{\theta}{\alpha\delta} + K_n, \frac{n}{\alpha\delta}-K_n}$ can be read from (\ref{SKn}). Note that the first variable in the product,
$$
S_{\alpha\delta,n+\theta}=S_{\alpha,n+\theta}\times S^{\frac1\alpha}_{\delta,\frac{\theta+n}{\alpha}}
=\frac{\tau_{\alpha}(\tau_{\delta}(G_{\frac{\theta+n}{\alpha\delta}}))}{{\big(\tau_{\delta}(G_{\frac{\theta+n}{\alpha\delta}})\big)}^{\frac1\alpha}}
{\left[\frac{\tau_{\delta}(G_{\frac{\theta+n}{\alpha\delta}})}{G^{\frac1\delta}_{\frac{\theta+n}{\alpha\delta}}}\right]}^{\frac1\alpha},
$$
where equivalences and independence properties can be read from \cite[Proposition 21]{PY97}. Hence, 
\begin{equation}\label{tempYalphadelta}
Y^{(n-K_{n}\alpha\delta)}_{\alpha\delta,\theta+K_{n}\alpha\delta}= S_{\alpha,\theta+n} \times \left(\frac{S_{\delta,\frac{\theta+n}{\alpha}}}{\beta^{1/\delta}_{\frac{\theta}{\alpha\delta} + K_n, \frac{n}{\alpha\delta}-K_n}}\right)^{\frac1\alpha}.
\end{equation}
The result follows by noticing that the ratio with an exponent $\frac1\alpha$ in the last expression equals $Y^{(\frac{n}{\alpha} - K_n \delta)}_{\delta, \frac{\theta}{\alpha} + K_n \delta}$ by definition in~(\ref{SKn}).
\end{proof}

\begin{thm}\label{PropCondSKn}
Consider $S_{\alpha\delta,\theta}:=S_{\alpha,\theta}\times S^{\frac1\alpha}_{\delta,\frac{\theta}{\alpha}}$ under a 
$\mathrm{PD}(\alpha\delta,\theta)$ distribution. As in Lemma~\ref{lemmablockscoag}, $K_{n}:=K^{(2)}(K^{(1)}_{n})$ denotes the number of blocks of a $\mathrm{PD}(\alpha\delta,\theta)$ partition of $[n]$. The random variable 
\begin{equation}\label{EqnCondSKn}
Y^{\big(n-K^{(2)}_{\ell}\alpha\delta\big)}_{\alpha\delta,\theta+K^{(2)}_{\ell}\alpha\delta}
=\frac{S_{\alpha,\theta+n}}{\beta^{\frac1\alpha}_{\frac{\theta}{\alpha}+\ell,\frac{n}{\alpha}-\ell}}\times \left(S_{\delta,\frac{\theta}{\alpha}}\right)^{\frac1\alpha}
\end{equation}
has the conditional distribution of $S_{\alpha\delta,\theta}|K^{(1)}_{n}=\ell.$ Hence, the conditional density is given by
\begin{equation}
\mathbb{P}(S_{\alpha\delta,\theta}\in dt|K^{(1)}_{n}=\ell)/dt:=\sum_{j=1}^{\ell}\mathbb{P}^{(\ell)}_{\delta,\frac{\theta}{\alpha}}(j)f^{(n-j\alpha\delta)}_{\alpha\delta,\theta+j\alpha\delta}(t),
\label{conddensityalphathetaell}
\end{equation}
which gives an alternative expression for the density of the random variable in (\ref{basicSadcond}) for the case of $\ell=k,$ $t=y$ and  $\theta=0$.
\end{thm}

\begin{proof}
Combining Lemma~\ref{PropSalphadelta} and $K_{n}:=K^{(2)}\big(K^{(1)}_{n}\big)$ yields 
$$
S_{\alpha\delta,\theta} = Y^{\big(n-K^{(2)}(K^{(1)}_{n})\alpha\delta\big)}_{\alpha\delta,\theta+K^{(2)}(K^{(1)}_{n})\alpha\delta}.
$$ 
Then, given $K_n^{(1)} = \ell$, the right hand side becomes $Y^{\big(n-K^{(2)}_\ell \alpha\delta\big)}_{\alpha\delta,\theta+K^{(2)}_\ell\alpha\delta}$ by writing $K^{(2)}(\ell) := K^{(2)}_\ell$, and is equivalent to the distribution of $S_{\alpha\delta,\theta}|K^{(1)}_n = \ell$. The equality in~(\ref{EqnCondSKn}) can be obtained by inserting the conditional statement $\{K^{(1)}_{n}=\ell\}$ in the equation~(\ref{tempYalphadelta}) in the proof of Lemma~\ref{PropSalphadelta}. Consider the ratio with an exponent $\frac1\alpha$ in~(\ref{tempYalphadelta}) given by
$$
\frac{S_{\delta,\frac{\theta+n}{\alpha}}}{\beta^{\frac 1\delta}_{\frac{\theta}{\alpha\delta}+K_{n},\frac{n}{\alpha\delta}-K_{n}}}=\frac{S_{\delta,\frac{\theta}{\alpha}+K_{n}\delta}}{\beta_{\frac{\theta}{\alpha}+K_{n}\delta,\frac{n}{\alpha}-K_{n}\delta}}
=\frac{S_{\delta,\frac{\theta}{\alpha}+K_{n}\delta}}{\beta_{\frac{\theta}{\alpha}+K^{(1)}_{n},\frac{n}{\alpha}-K^{(1)}_{n}}\,\beta_{\frac{\theta}{\alpha}+K_{n}\delta,K^{(1)}_{n}-K_{n}\delta}},
$$
where the first equality is a special case of~(\ref{jamesid}) and the last is just a beta identity using the fact that $\frac{n}{\alpha}>K^{(1)}_{n}\ge K_{n}.$ Replacing $K_n := K^{(2)}\big(K^{(1)}(n)\big)$ from Lemma~\ref{lemmablockscoag} in the latter expression yields that the ratio under consideration conditioning on $K^{(1)}_n = \ell$ equals $\beta_{\frac{\theta}{\alpha}+\ell,\frac{n}{\alpha}-\ell}^{-1}$ times
$$
\frac{S_{\delta,\frac{\theta}{\alpha}+[K^{(2)}(\ell)]\delta}}{\beta_{\frac{\theta}{\alpha}+[K^{(2)}(\ell)]\delta,\ell-[K^{(2)}(\ell)]\delta}} = Y^{\big(\ell - K^{(2)}_\ell \delta\big)}_{\delta, \frac{\theta}{\alpha} +  K^{(2)}_\ell \delta} = S_{\delta, \frac{\theta}{\alpha}},
$$
where the first equality follows from~(\ref{SKn}), and the last equivalence is a special case of (\ref{SKn}) based on the fact that for any fixed integer $\ell,$  $K^{(2)}_{\ell}$ is the number of blocks of a partition of $[\ell]=\{1,\ldots,\ell\}$ generated from $\mathrm{PD}\big(\delta,\frac\theta\alpha\big).$ 
Combining all these results justifies the right hand side of~(\ref{EqnCondSKn}) to be the conditional distribution of $S_{\alpha\delta,\theta}|K^{(1)}_{n}=\ell.$ What remains is to set $K^{(2)}_{\ell}=j$ so as to recognize that $Y^{(n-j\alpha\delta)}_{\alpha\delta,\theta+j\alpha\delta}$ is equivalent to~(\ref{conddentheta}) with $k=j$ and $\alpha\delta$ in place of $\alpha,$ and thus it has density $f^{(n-j\alpha\delta)}_{\alpha\delta,\theta+j\alpha\delta}.$ Otherwise one mixes over 
$\mathbb{P}_{\delta,\frac{\theta}{\alpha}}^{(\ell)}(j):=\mathbb{P}_{\delta,\frac{\theta}{\alpha}}(K^{(2)}_{\ell}=j)$ as indicated to get~(\ref{conddensityalphathetaell}).
\end{proof}

We conclude this subsection with some more distributional results about $S_{\alpha\delta,\theta}$ following from Lemma~\ref{PropSalphadelta} and Theorem~\ref{PropCondSKn}.

\begin{cor}
Consider the same setting as in Lemma~\ref{PropSalphadelta}.
\begin{enumerate}
\item[(i)] The random variable $Y^{(\frac{n}{\alpha}-K_{n}\delta)}_{\delta,\frac{\theta}{\alpha}+K_{n}\delta} = \frac{S_{\delta,\frac{\theta+n}{\alpha}}}{\beta^{\frac1\delta}_{\frac{\theta}{\alpha\delta} + K_n, \frac{n}{\alpha\delta}-K_n}}$ in the representation~(\ref{YSalphadelta}) of $S_{\alpha\delta, \theta}$ can be alternatively expressed as
\begin{equation}
\label{coagid}
Y^{(\frac{n}{\alpha}-K_{n}\delta)}_{\delta,\frac{\theta}{\alpha}+K_{n}\delta}
=\frac{S_{\delta,\frac{\theta}{\alpha}}}{\beta_{\frac{\theta}{\alpha}+K^{(1)}_{n},\frac{n}{\alpha}-K^{(1)}_{n}}}
\overset{d}=\frac{S_{\delta,\frac{\theta}{\alpha}}}
{\prod_{j=1}^{n}\beta_{\frac{\theta+\alpha+j-1}{\alpha},\frac{1-\alpha}{\alpha}}}.
\end{equation}
\item[(ii)] The distribution of $S_{\alpha\delta,\theta}|K_{n}=k$ corresponds to the random variable $S_{\alpha,\theta+n}
\times
{\left[Y^{(\frac{n}{\alpha}-k\delta)}_{\delta,\frac{\theta}{\alpha}+k\delta}\right]}^{\frac1\alpha},$ where
\begin{equation}
Y^{(\frac{n}{\alpha}-k\delta)}_{\delta,\frac{\theta}{\alpha}+k\delta}=\frac{S_{\delta,\frac{\theta+n}{\alpha}}}{\beta^{\frac1\delta}_{\frac{\theta}{\alpha\delta}+k,\frac{n}{\alpha\delta}-k}}\overset{d}=\frac{S_{\delta,\frac{\theta}{\alpha}}}{\beta_{\frac{\theta}{\alpha}+K^{(1)}_{n|k},\frac{n}{\alpha}-K^{(1)}_{n|k}}},
\label{condCoag}
\end{equation}
and $K^{(1)}_{n|k}$ has distribution $\mathbb{P}^{(n)}_{\alpha|\alpha\delta}(\ell|k).$
\end{enumerate}
\end{cor}
\begin{proof}
In~(\ref{coagid}), the first equality follows from a similar proof to that of Theorem~\ref{PropCondSKn}, while the second equality is due to the identity (\ref{betaKid}).
\end{proof}


\end{document}